\documentclass[final,1p,times,authoryear]{elsarticle}

\usepackage{amssymb}
\usepackage{amsmath}
\usepackage{amsthm}

\usepackage{color}

\journal{EcoSta}

\usepackage[normalem]{ulem}
\newcommand{\RV}{\mathcal{RV}}
\newcommand{\dd}{\mathrm{d}}

\newcommand{\sg}[1]{{\color{black}#1}}
\newcommand{\cp}[1]{{\color{black}#1}}

\newcommand{\rtworev}[1]{{\color{r2mot}#1}}
\newcommand{\rae}[1]{{\color{r2mot}#1}}
\definecolor{r2mot}{rgb}{0,0,0}
\newcommand{\rtwomot}[1]{{\color{r2mot}#1}}
\newcommand{\rorange}[1]{{\color{r2mot}#1}}

\newcommand{\R}{\mathbb{R}}

\newcommand{\eps}{\varepsilon}

\newcommand{\argmax}{\operatornamewithlimits{\arg\max}}

\newcommand{\E}{\op{\mathbb{E}}}

\newcommand{\p}{\mathbb{P}}

\makeatletter
\def\z@first#1#2{#1}
\def\z@second#1#2{#2}
\def\z@zp@selectchar#1#2{
	\IfStrEqCase{#2}{%
		{p}{#1{(}{)}}%
		{P}{#1{)}{(}}%
		{c}{#1{[}{]}}%
		{C}{#1{]}{[}}%
		{a}{#1{\{}{\}}}%
		{A}{#1{\}}{\{}}%
		{i}{#1{[}{]}\!#1{[}{]}}%
		{I}{#1{]}{[}\!#1{]}{[}}%
		{t}{#1{<}{>}}%
		{T}{#1{>}{<}}%
		{b}{#1{|}{|}}%
		{n}{#1{\|}{\|}}%
		{v}{#1{.}{.}}%
	}[#1{(}{)}]%
}
\def\z@zp#1#2\fin#3{
	\z@zp@selectchar{\left\z@first}{#1}#3
	\zifempty{#2}%
	{\z@zp@selectchar{\right\z@second}{#1}}%
	{\z@zp@selectchar{\right\z@second}{#2}}%
}
\newcommand{\zp}[2][p]{\zifempty{#1}{\left(#2\right)}{\z@zp#1\fin{#2}}}
\makeatother

\usepackage{xstring}
\def\zifempty#1#2#3{\def\foo{#1}\ifx\foo\empty\relax#2\else#3\fi}

\newcommand{\lam}{\lambda}

\newcommand{\vfi}{\varphi}
\newcommand{\al}{\alpha}

\newcommand{\bgt}{\begin{itemize}}
	\newcommand{\ent}{\end{itemize}}

\newcommand{\op}{\operatorname}

\newcommand{\Cov}{\operatorname{Cov}}

\newcommand{\f}{\frac}
\newcommand{\ff}{\frac{1}}

\newcommand{\bbm}{\begin{bmatrix}}
	\newcommand{\ebm}{\end{bmatrix}}
\newcommand{\bes}{\begin{equation*}}
	\newcommand{\ees}{\end{equation*}}
\newcommand{\be}{\begin{equation}}
	\newcommand{\ee}{\end{equation}}
\newcommand{\beqy}{\begin{eqnarray}}
	\newcommand{\eeqy}{\end{eqnarray}}
\newcommand{\beq}{\begin{eqnarray*}}
	\newcommand{\eeq}{\end{eqnarray*}}

\newcommand{\bpm}{\begin{pmatrix}}
	\newcommand{\epm}{\end{pmatrix}}

\long\def\symbolfootnote[#1]#2{\begingroup
	\def\thefootnote{\fnsymbol{footnote}}\footnote[#1]{#2}\endgroup}

\makeatletter
\def\@addpunct#1{%
	\relax\ifhmode
	\ifnum\spacefactor>\@m \else#1\fi
	\fi}

\newtheorem{theorem}{Theorem}
\newtheorem{Prop}{Proposition}
\newtheorem{lem}{Lemma}

\theoremstyle{definition}
\newtheorem{defn}{Definition}
\newtheorem{Example}{Example}

\theoremstyle{remark}
\newtheorem{Remark}{Remark}

\begin{document}

	\begin{frontmatter}

		\title{Functional Extreme-PLS}

		\author[label1]{Stéphane Girard}
		\author[label2]{Cambyse Pakzad\corref{cor1}}
		\ead{cambyse.pakzad@parisnanterre.fr}
		\cortext[cor1]{Corresponding author.}

		\affiliation[label1]{
			organization={Université Grenoble Alpes, Inria, CNRS, Grenoble INP, LJK},
			city={Grenoble},
			postcode={38000},
			country={France}
		}

		\affiliation[label2]{
			organization={Université Paris Nanterre, Modal’X UMR CNRS 9023},
			addressline={200 Avenue de la République},
			city={Nanterre},
			postcode={92001},
			country={France}
		}

		\begin{abstract}
			\rae{An extreme dimension reduction method is proposed, extending the Extreme-PLS approach to the discretized functional framework, where the covariate lies in the infinite-dimensional Hilbert space ($L^2([0,1])$) but is partially observed on a dense time grid. The ideas are partly borrowed from both Partial Least-Squares (PLS) and Sliced Inverse Regression (SIR) techniques. Accordingly, the method relies on the projection of the covariate onto a subspace and maximizes the covariance between its projection and the response conditionally on an extreme event capturing tail information. The covariate and the heavy-tailed response are supposed to be linked through a nonlinear inverse single-index model, and the objective is to infer the index in this regression framework. A new family of estimators is proposed, and its asymptotic consistency with convergence rates under the model is established. Assuming mild conditions on the noise, most of the assumptions are stated in terms of regular variation, unlike those commonly imposed for SIR and single-index regression. In addition, the theoretical analysis is expanded to include a model-free consistency result in probability for the empirical tail-moments in a general separable Hilbert space. Finally, the results are illustrated through a finite-sample study with synthetic functional data and high-frequency financial data, highlighting the effectiveness of the dimension reduction in capturing tail dependence and supporting extreme risk management.}
		\end{abstract}

	\end{frontmatter}

		\section{Introduction}

		\rae{Partial Least Squares (PLS) is a widely used method for dimension reduction. Initiated} by \citet{Wold1975} and adapted to $L^2([0,1])$ by \citet{Preda2005}, \rae{it} has proven to be successful in both theoretical and practical sides; especially in chemometrics \citep{martens1992multivariate} where it takes root. It combines characteristics of Principal Component Analysis (PCA) and regression. Its core idea is to \rae{identify} a projection of the covariate/predictor $X$ \rae{that exhibits} high variance and \rae{a strong} correlation with the response $Y$. Another standard line of research is Sliced Inverse Regression (SIR), originally introduced by \citet{Li1991} and then by \citet{DauxoisFerreYao2001} in the functional framework. It hinges on the inverse regression `$X$ against $Y$' that happens to be simpler than the forward regression, usually at the price of a certain condition of linearity and constant variance. \sg{This \rae{approach} departs from the Fisherian viewpoint, in which inference is carried out conditionally on the covariate $X$; see \citet{Cook2007} for a discussion \rae{of} inverse regression.}  

		\rae{However,} the aforementioned techniques require the full sample and are not tailored \rae{to} extreme-value analysis. \rae{This motivates the investigation of their behavior} 
		in the extreme regime, \emph{i.e.,} when the focus is on the conditional tail of $Y$
		given a high-dimensional covariate $X$.
		\sg{\rae{Dimension reduction methods have recently been extended to extreme-value settings.}} \rae{In this setting,} a major challenge is \sg{a twofold scarcity}, in the sense that both the curse of dimensionality and the rarity of tail observations operate. Such hindrance may apply in the estimation of certain statistics or risk measures in the functional extreme regime (see \emph{e.g.}, \citealp{GG2012,GSU2022}). To circumvent this issue, one would build up the inference on the existence of a projection that captures all the information on the extreme values of $Y$. \sg{Next, instead of using standard estimation methods on $X$, one would replace $X$ by its projection as the covariate, opting for a lower-dimensional method expected to yield comparable or even better performances.} The pioneering work may be identified as \citet{Gardes2018}, later followed up by \citet{Gardes2020}, which proposes a new notion of tail conditional independence. Applying this to the extremal quantile regression, the author then studies the inference under a single and multi-index model when the dimension is large but finite and when the dimension reduction subspace is known. \sg{Still, some theoretical guarantees remain open and the computational cost is substantial.} In this direction, \citet{wang2020extreme} develop a three-step estimation procedure for such large conditional quantiles and \sg{show $\sqrt{n}$-consistency under a new tail single-index model with unknown indices, at a reduced computational cost.} Alternatively, focusing on the prediction of tail events, \citet{Aghbalou2024} propose a new framework which introduces a revised definition of tail conditional independence and provides theoretical guarantees for estimating a low-dimensional projection of covariates that is sufficient for predicting extremes, using inverse regression under standard SIR assumptions.

		Recently, a novel method based on the PLS in the extreme regime was introduced by \citet{Bousebata2023}. Therein, the authors introduce an inverse single-index model. This means that  the covariate is written in terms of the response and some unknown link function \rtworev{(instead of the response being written in terms of the covariate, as in a forward regression model)} which results in a reduction of dimension. In particular, under some conditions and within the regular variation framework, this inverse point of view may also be interpreted as a forward single-index model. Next, they propose an estimator of the index which is shown to be consistent with explicit rate and under a deterministic threshold level. \rae{Their assumptions differ substantially from those commonly imposed in} dimension reduction\rae{,} \sg{as neither conditions on the inverse conditional mean $\E(X|Y)$, such as linearity, nor independence or centring of the noise are required.} \rae{See also} \citet{STCO7} for an adaptation to the Bayesian framework allowing the introduction of prior information on the index.

		\rae{The methodology} introduced in \citet{Bousebata2023} for the supervised branch of extreme dimension reduction \rae{is taken up below and converted} to the functional covariate setting. Specifically, \rae{the predictor/covariate is projected} onto a one-dimensional subspace \rae{so as to capture} most of the information explaining the extreme values of the response variable. To this end, \rae{covariates} that lie in the infinite-dimensional Hilbert space $H=L^2([0,1])$ \rae{are considered,} and the projection is expressed in terms of the associated inner product. In addition, \rae{the more realistic case is addressed} where the functional covariate is discretized as in \citet{Cardot2003,Crambes2009}, \emph{i.e.}, partially observed on a deterministic time grid. \rae{The proposed methodology} would equally apply when the functional covariate is ideally/fully observed (in this case, $H$ could be any \sg{separable} Hilbert space) and also when $H=\R^d$ with $d$ large but finite, \sg{thereby encompassing \citet{Bousebata2023} as a particular case}. Following the philosophy of PLS, \rae{the covariance} between the projected predictor/covariate and the response \rae{is maximized} conditionally on the response being larger than an increasing threshold. This corresponds to the extreme framework where the response is assumed to be heavy-tailed. As for the covariate, \rae{the method echoes} SIR by modeling its relationship to the response through a non-linear inverse regression model which relies on a single index, \emph{i.e.,} a deterministic vector in $H$. One may again recover a forward single-index model interpretation. Thereby, the method at hand induces an effective reduction of dimension as \rae{it provides} an estimation of the aforementioned index based on the observations. \rae{Through} a simple geometrical interpretation, \rae{a new family} of estimators \rae{is} \rae{proposed, which generalizes} the one in \citet{Bousebata2023} and hence \sg{gives the statistician more freedom in choosing which part of the extreme regime to take into account.} The consistency (with rate of convergence) of these estimators 
		is established under similar conditions as in \citet{STCO7,Bousebata2023} while extending these two works to the more realistic case of a random threshold \sg{given by an extreme order statistic}. The assumptions are expressed in the language of regular variation and require minimal conditions about the noise, unlike \rae{those commonly imposed for SIR and forward single-index models}. In practice, the finite-sample properties of the method show good results on numerical simulations with synthetic functional data and on high-frequency financial data. 

		\rae{Section~\ref{sec-theo} first introduces} the mathematical framework for the Functional Extreme Partial Least-Squares (FEPLS for short) and, more specifically, the inverse single-index model. The estimators are presented in Section~\ref{sec-inf} with the associated consistency results \sg{within and outside the inverse regression model}. Next, the FEPLS finite-sample properties are illustrated in Section~\ref{sec-sim} on simulated data. Beyond simulations, \rae{Section~\ref{sec-real} illustrates} the method’s practical value in a financial risk assessment context, using high-frequency asset data to estimate conditional extreme quantiles and compare tail dependence capture against standard dimension-reduction techniques. The corresponding \texttt{Python} code is publicly available on the website of the authors.
		The mathematical details needed for the derivation of \rae{the} main result are gathered in \ref{sec-proofs}.

		\section{Theoretical framework}
		\label{sec-theo}
		Let $H=L^2([0,1],\R)$ endowed with its usual scalar product $\langle \cdot,\cdot \rangle$ and Borel $\sigma$-field $\mathcal{B}(H)$, and let $(\Omega,\mathcal{A},\p)$ be a complete probability space. Consider a random pair $(Y,X)$, where $Y:\Omega\to \R$ and $X:\Omega\to H$ are random variables, \emph{i.e.}, $Y$ is Borel measurable and $X$ is strongly $\p$-measurable in the sense of Bochner, see \citet[Definition 1.1.3]{bochner}.

		\rae{Consider first} the direction defined as
		\begin{align}
			\label{eq:fepls}
			w(y):=\argmax_{\|  w\|=1} \Cov\big( \langle  w,   X \rangle,Y \mid Y\ge y\big),\quad y\in \R, 
		\end{align} 
		where $\|\cdot \|$ is the norm induced by the inner product $\langle \cdot,\cdot\rangle$ and $y$ is a threshold, ultimately taken large. The expectation underlying the covariance in~\eqref{eq:fepls} is understood in the sense of Lebesgue. 
		The interpretation of~\eqref{eq:fepls} follows the core idea of PLS but in the extreme regime. Using an optimization framework, \rae{one seeks} the element in $H$ which maximizes the tail-information, by considering the covariance
		conditionally on the tail event of $Y$, shared between $Y$ and the projection of $X$ on a single direction.

		In the sequel, the cumulative distribution function (cdf) of $Y$ is denoted by $F$ and the survival function by $\bar{F}=1-F$. In addition, \rae{the tail-moment} of a random object~$W$ \rae{is defined} by $m_W(y) := \E(W 1_{\{Y\ge y\}})$ whenever it exists. In the case where $W\in H$, \rae{the expectation} at hand \rae{is understood} as the Bochner expectation denoted by $\E_{\mathcal{B}}$. \rae{Since $H$ is separable, the Pettis integral could equivalently be used} instead. \rae{See} \citet{bochner} for more details about the notion. 

		A key fact is that the solution of the optimization problem~\eqref{eq:fepls} is explicit:
		\cp{%
		\begin{Prop}\label{prop:solution}
			Suppose that $\E\big((1+\|X\|)(1+|Y|)\big)<+\infty$. Then, the unique solution of the optimization problem \eqref{eq:fepls} is given by, for any $y\in \R$ such that $\bar{F}(y)>0$ and $v(y)\neq 0$,
			\begin{align*}
				w(y) &= v(y)/{\|v(y)\|} \mbox{ where } v(y) = \bar{F}(y) m_{XY}(y)- m_X(y) m_Y(y).
			\end{align*} 
		\end{Prop} 
}%
		\noindent It readily appears that $w(y)$ is a linear
		combination of $m_{XY}(y)\in H$ and $ m_X(y)\in H$\sg{, two tail-moments of the form $m_{X\vfi(Y)}(y)$ associated with the respective weights $\vfi(y)=y$ and $\vfi\equiv 1$. Under the inverse regression model introduced below, both directions asymptotically align with a common index. \rae{The use} of more general tail-moments \rae{is} thus \rae{investigated} by considering}
		\begin{align}
			\label{eq:solution2}
			w_\varphi(y) &= \sg{m_{X\varphi(Y)}}(y)/{\|\sg{m_{X\varphi(Y)}}(y)\|},%
		\end{align}
		where \sg{$\vfi:\R\to \R^+$} is some test function such that \begin{align}
			\label{hyp:test}
		\sg{ \vfi}\in \RV_\tau(+\infty), \quad \tau\in \R.
		\end{align}
		\rae{The standard notation} $\vfi \in \RV_\tau(+\infty)$ \rae{means} that $\varphi$ is regularly varying at infinity with index~$\tau$. \sg{The definition is recalled below for the sake of completeness, see \cite{bingham_goldie_teugels_1987} for further details.
		\begin{defn}[$\RV_\tau(+\infty)$]
		A positive measurable function $V$ defined on some neighborhood $[x_0,+\infty)$ of infinity and such that, for all $x>0$, $$ \lim\limits_{t\to +\infty} \f{V(tx)}{V(t)}=x^\tau$$ 
		is said to be regularly varying at infinity with index $\tau$.
		\end{defn}}
			\sg{The influence of $\tau$ (and thus of $\varphi$) on the FEPLS direction $w_\varphi(y)$ is discussed in Remark~\ref{rmk:role_phi} and illustrated on simulations in Section~\ref{sec-sim}.  }
		Before undertaking an inference procedure based on the empirical tail-moments in \eqref{eq:solution2}, \rae{consider} the \sg{four model assumptions \eqref{hyp:2rv}, \eqref{eq:single_index_model}, \eqref{hyp:link} and~\eqref{hyp:vonmises}}.
		An extreme-value framework is considered where the response variable $Y$ is heavy-tailed to the second order. To this end, denote the generalized inverse function or quantile function of the cdf $F$ by $u\mapsto F^-(u):=\inf \{ x \in \R : F(x)\ge u\}$. The tail quantile function is then $t\geq 1\mapsto U(t):=F^-(1-1/t)$ and \sg{is assumed to belong to the class of second-order regularly varying functions:
		\begin{equation}
		    \label{hyp:2rv}
		    U\in 2\RV_{\gamma,\rho}(+\infty), \quad \gamma\in(0,1) \mbox{ and } \rho\leq 0.
		\end{equation}
		\begin{defn}[$2\RV_{\gamma,\rho}(+\infty)$]
			A positive measurable function $V$ defined on some neighborhood $[x_0,+\infty)$ of infinity is
			said to be second-order regularly varying with parameters $(\gamma\in\R,\rho\leq0)$, if, for all $x>0$,
			\begin{align*}		\lim\limits_{t\to +\infty}	\ff{A(t)}\left( \f{V(tx)}{V(t)}- x^{\gamma} \right)&=x^{\gamma} H_\rho(x):=x^{\gamma}\int_1^x u^{\,\rho-1}\mathrm{d}u,	\end{align*} 
		where $A$ is an auxiliary function  ultimately of constant sign with $A(t)\to 0$ as $t\to+\infty$.
				\end{defn}}

		\rae{For} $y>0$,  $H_\rho(y)=\log(y)1_{\{\rho=0\}} +\f{y^{\,\rho}-1}{\rho}1_{\{\rho<0\}}$. 
			\sg{\rae{Since $A(t)\to 0$,} $2\RV_{\gamma,\rho}(+\infty)\subset \RV_{\gamma}(+\infty)$.} Besides, by \citet[Theorem~2.3.9]{Haan2007}, the following correspondence in terms of regular variation indices holds: 
		$
		U\in 2\RV_{\gamma,\rho}(+\infty)  \iff \bar{F}\in 2\RV_{-1/\gamma,\rho/\gamma}(+\infty) .
		$
		A direct consequence of~\eqref{hyp:2rv} is that $U\in \RV_{\gamma}(+\infty)$ while $\bar{F}\in \RV_{-1/\gamma}(+\infty)$.
	\sg{	The latter first-order condition is equivalent to assuming that the distribution of $Y$ is in the Fr\'echet maximum domain of attraction with positive tail-index ${\gamma}$,
			see \citet[Theorem~1.2.1]{Haan2007}.
			This domain of attraction consists of heavy-tailed distributions, such as Pareto, Burr and Student distributions, see~\citet{beigoesegteu2004} for further examples.
			The larger ${\gamma}$ is, the heavier the tail. Going beyond the first order, the control of some tail-moment as in \citet{Stupfler2019_Random_threshold} requires \eqref{hyp:2rv} which is expressed in terms of tail quantile functions. 
			The restriction to ${\gamma}<1$ ensures that the first-order moment { $\mathbb{E}(Y 1_{\{Y\geq y\}})$ exists for all $y\geq 0$}.}
		On the auxiliary functions level, \citet[Theorem~2.3.3]{Haan2007} provides $|A|\in {\RV}_{\rho}(+\infty)$ and $ |A| \circ (1/\bar{F}) \in {\RV}_{{\rho}/{\gamma}}(+\infty)$. 

		Next, in order to provide theoretical guarantees on the inference method described in the next section and based on~\eqref{eq:solution2}, \rae{the following (non-linear) inverse single-index functional model is considered,}
		\begin{align}
			\label{eq:single_index_model}
			{X} = g(Y)\, {\beta} +  {\eps},\quad  {\beta}\in H, \quad \|\beta\|=1,
		\end{align}
		with $\eps:\Omega\to H$ a random variable being the noise and an unknown deterministic link function \sg{$g:\R\to\R^+$} 
			\sg{ultimately behaving like a power function {\it i.e.}}
		\begin{align}
			\label{hyp:link}
		\sg{ g}\in \RV_{\kappa}(+\infty), \quad \kappa>0.  
		\end{align}
		\sg{Model~\eqref{eq:single_index_model} is referred to as an inverse regression model since the functional covariate $X$ is written as a function of the response variable $Y$.
The interest of such inverse models is highlighted in~\cite{Cook2007} to compare the theoretical properties of PCA, principal fitted components and SIR, among others.}
		The appeal of model~\eqref{eq:single_index_model} stems from the fact that, for instance when $\eps$ is independent of $Y$ and centered, the FEPLS direction~\eqref{eq:solution2} coincides with the true index: $w_\varphi(y)=\beta$ 
		for any test function $\varphi\in\RV_\tau(+\infty)$ and \sg{any $y\ge 0$ such that $m_{g\varphi(Y)}(y)\in(0,+\infty)$}. Heuristically, assuming more generally that $Y$ and $\eps$ are dependent but $\eps$ has a small contribution in the tail regime of $Y$, \sg{{\it i.e.} $X \simeq g(Y)\,\beta$}, one expects that 
		$w_\varphi(y)\to \beta$ as $y\to+\infty$. 
		\sg{Moreover, in such a situation, model~\eqref{eq:single_index_model} yields the approximate single-index forward model $Y\simeq g^{-1}(\langle \beta, X\rangle)$, where $g^{-1}$ denotes an asymptotic inverse of $g$, which exists since $\kappa>0$. }%

		\cp{%
		To deal with the conditional expectation involving the tail event $\{Y\ge y\}$ for large $y$, as well as with the distribution of order statistics, \rae{some regularity} on the distribution of $Y$ \rae{is needed. The density} $f:=F'$ of $Y$ \rae{is therefore assumed to exist and to satisfy} the von Mises condition
		\begin{align}
			\label{hyp:vonmises}
			f\in\RV_{-1/\gamma - 1}(+\infty).
		\end{align}
		By Karamata's theorem, \eqref{hyp:vonmises} implies $\bar{F}\in \RV_{-1/\gamma}(+\infty)$, consistently with~\eqref{hyp:2rv}.
}%

		\section{Inference and consistency}
		\label{sec-inf}

		Let $(X_i,Y_i)_{1\le i \le n}$ be independent copies of $(X,Y)\in H\times \R$. Let $k:=k_n\to+\infty$ be an intermediate sequence, {\it i.e.} integer deterministic such that $ k/ n\to 0$ as $n\to+\infty$, and consider a deterministic sequence $y_{n,k}$ such that $y_{n,k}\sim U(n/k)$ \sg{{\it i.e.} $y_{n,k}/U(n/k)\to 1$} as $n\to +\infty$. Denoting the order statistics of the sample $(Y_i)_{1\le i \le n}$ by $Y_{1,n}\le \cdots \le Y_{n,n}$, the empirical version of $U(n/k)$ is $Y_{n-k+1,n}$. In the sequel and for $\alpha > 0$, \rae{denote} by $\mathcal{C}^{0,\alpha}([0,1], \mathbb{R})$ the space of $\alpha$-Hölder continuous functions, \emph{i.e.}, the set of  $f: [0,1] \to \mathbb{R}$ such that there exists a finite constant $c > 0$ satisfying
		\[
		|f(x) - f(y)| \le c |x-y|^\alpha \qquad \text{for all } x, y \in [0,1].
		\]
		\sg{For $h\in\mathcal{C}^{0,\alpha}([0,1], \R)$, the H\"older norm of $h$ is defined as
		\[
		\|h\|_{0,\alpha}:=\sup_{x\in[0,1]}|h(x)|+\sup_{x\neq y}\f{|h(x)-h(y)|}{|x-y|^{\alpha}},
		\]
		with the convention $\|h\|_{0,\alpha}=+\infty$ whenever $h\notin\mathcal{C}^{0,\alpha}([0,1],\R)$.}
		Concerning the noise, \sg{denoting by $\eps_i:=X_i-g(Y_i)\beta$, $1\le i\le n$, the noise associated with the $i$-th observation under model~\eqref{eq:single_index_model},} \rae{one requires} the existence of $q>2$ and $\alpha_e > 0$ independent of $n$ such that \begin{align}
			\label{hyp:noise_cond}
			\limsup_{n\to +\infty}\sup_{y\geq 0} \E\left(\|\eps_1\|_{0,\alpha_e}^{q}\mid Y_{n-k+1,n} = y\right)&<+\infty.
		\end{align}
		\sg{This uniform control of the $q$-mean of $\|\eps_1\|_{0,\alpha_e}$ \sg{given}
			\sg{the extreme order statistic} $Y_{n-k+1,n}$ implies that $\mathbb{E}(\|\eps\|^q)<+\infty$. Since $q>2$, one has $\E(\| \eps \|)<+\infty$ and, by \citet[Proposition~1.2.2]{bochner}, $\eps$ is also Bochner-integrable since $\eps$ is strongly $\p$-measurable by definition.}

		For any random variable $W$, the empirical counterpart of the tail-moment $m_W(y)$ (whenever it exists) is defined by 
		$$
		\hat m_W(y):= \ff{n}\sum_{i=1}^n W_{i} 1_{\{ Y_i\ge y\}},
		$$
		based on the $n$-sample $(W_i,Y_i)_{1\le i \le n}$.
		Two situations are investigated. When the full functions $X_1,\dots,X_n$ are observed, it is then possible to estimate $w_\varphi(y)$ by replacing the tail-moment in~\eqref{eq:solution2} by its empirical counterpart defined as
		$$ \sg{{\hat \beta}_{\vfi}(y) := \f{ {\hat m_{X\vfi(Y)}}(y)}{\| {\hat m_{X\vfi(Y)}}(y)\|} } \mbox{ with }%
		\hat m_{X\vfi(Y)  }(y) = \ff{n} \sum_{i=1}^n   X_i \vfi(Y_i)1_{\{ Y_i\ge y \}}.$$
		The FEPLS estimator ${\hat \beta}_{\vfi}(y)$
		is thus a linear combination of the $X_i$'s whose  associated $Y_i$'s are located in the tail of $\bar F$. 

		\cp{%
		\rae{Consider now} the more realistic situation where each function $X_i$, $i\in\{1,\dots,n\}$, is observed on a finite deterministic grid $0\le t_1<\dots<t_N\le 1$ only. Let $\{I_1,\dots,I_N\}$ be a partition of $[0,1]$ into intervals such that $t_j\in I_j$ for any $j\in\{1,\dots,N\}$ and
		\begin{align}
			\label{hyp:grid}
			\max_{1\le j\le N}|I_j|\le c_0/N,
		\end{align}
		for some constant $c_0\ge1$ independent of $N$, where $|I_j|$ denotes the length of $I_j$; the regular design $t_j=(2j-1)/(2N)$ with $I_j=[(j-1)/N,j/N)$ is the prototypical example. The discrete record $(X_i(t_1),\dots,X_i(t_N))$ is then summarized by the piecewise-constant reconstruction
		$$
		\tilde X_i := \sum_{j=1}^N X_i(t_j)\, 1_{I_j}\in H .
		$$
		\rae{Under} model \eqref{eq:single_index_model} with $\beta\in\mathcal{C}^{0,\al_b}([0,1],\R)$ and $\eps\in\mathcal{C}^{0,\al_e}([0,1],\R)$ a.s., the trajectories of $X$ are continuous, so that the pointwise evaluations $X_i(t_j)$ are well defined. In this context, the estimator of $\beta$ is given by
		$$
		{\tilde \beta}_{\vfi}(y) := \f{ {\hat m_{\tilde  X\vfi(Y)}}(y)}{\| {\hat m_{\tilde  X\vfi(Y)}}(y)\|} \mbox{ with }
		\hat m_{\tilde  X\vfi(Y)  }(y) = \ff{n} \sum_{i=1}^n   \tilde X_i \vfi(Y_i)1_{\{ Y_i\ge y \}}.
		$$
}%

		\subsection{\sg{Consistency under the inverse single-index model}}

		\rae{The consistency result} under the inverse single-index model \rae{is now stated.}
		\cp{%
		\begin{theorem}\label{th:main}
			Assume $\beta \in \mathcal{C}^{0,\alpha_b}([0,1],\R)$ for some $\alpha_b \in(0,1]$ and that \eqref{eq:single_index_model} holds. Assume 
			\begin{align}\label{hyp:2cgamma}
				0< 2(\kappa+\tau)\gamma &<1\quad\text{and}\\
				\label{hyp:qcgamma}
				q\kappa\gamma &>1,
			\end{align} 
			along with \eqref{hyp:test}, \eqref{hyp:2rv}, \eqref{hyp:link}, \eqref{hyp:vonmises}, \eqref{hyp:noise_cond} and the grid condition \eqref{hyp:grid}. Assume that $\vfi$ and $g$ are continuously differentiable in a neighborhood of infinity such that $t(\vfi g)'(t)/(\vfi g)(t) \to \tau+\kappa$ as $t\to+\infty$. 
			\rtworev{Let $k:=k_n\to+\infty$ be an intermediate sequence such that $\sqrt{k_n}\,A(n/k_n)=O(1)$ as $n\to+\infty$ and let $y_{n,k}\sim U(n/k)$ and $N:=N_n\to+\infty$ be deterministic sequences. Then}
			\begin{align*}
				\| \tilde \beta_\vfi(Y_{n-k+1,n}) - \beta \| = O_\p\big( \delta_{n,k} \vee N^{-\al_b}\big) \xrightarrow[n\to +\infty]{\p}0,
			\end{align*} 
			where $\delta_{n,k}:=(g(y_{n,k})(k/n)^{1/q})^{-1}$.
		\end{theorem}
}%

		\rae{The} sub-rate $\delta_{n,k}$ 
		is a regularly varying function of $k/n$ with index $\gamma \kappa -1/q>0$.
		It only depends on the tail distribution of $Y$, the integrability order of the noise $\eps$ and the asymptotic behavior of the deterministic link function~$g$.

		\rae{The following remarks discuss} the set of assumptions and their implications\rae{.}
		\begin{Remark}[\sg{The noise $\varepsilon$ does not contribute to the extremes}]\label{rmk:noise_extremes}
			Assumption \eqref{hyp:link} 
			combined with $U\in \RV_{\gamma}(+\infty)$ implies that  $g(Y)$ is heavy-tailed with tail-index $\gamma_{g(Y)}:=\kappa{\gamma}$.
			Similarly, \sg{when $\tau>0$,} \eqref{hyp:test} and $U\in \RV_{\gamma}(+\infty)$ yield that  $\vfi(Y)$ is heavy-tailed with tail-index $\gamma_{\vfi(Y)}:=\tau{\gamma}$.
			\rae{As already stated, condition}~\eqref{hyp:noise_cond} entails  $\mathbb{E}(\|\eps\|^q)<+\infty$ which may be interpreted as an assumption on the tail of $\|\eps\|$. It is satisfied, for instance, by distributions with exponential-like tails such as Gaussian, Gamma or Weibull distributions. 
			Heuristically, 
			it would imply that the tail-index associated with
			$\|\eps\|$ is such that $\gamma_{\|\eps\|}<1/q$.
			Condition~\eqref{hyp:qcgamma} thus imposes that $\gamma_{g(Y)}>\gamma_{\|\eps\|}$, meaning that $g(Y)$ has a heavier right tail than $\|\eps\|$.
			Under model~\eqref{eq:single_index_model}, \sg{one has $\gamma_{\|X\|}={\gamma_{g(Y)}}$ and thus the tail behaviors of $\langle\beta, X\rangle$ and $\|X\|$ are driven by $g(Y)$, 
			which is the desired property. }
		\end{Remark}
		\begin{Remark}[\sg{Rate of convergence}]\label{rmk:rate}
			\sg{(i)} The regime $y_{n,k}\sim U(n/k)$ entails $n\bar{F}(y_{n,k})\sim k\to \infty$ as $n\to\infty$ and it still holds when $y_{n,k}$ is replaced by its empirical counterpart, namely $Y_{n-k+1,n}$. \sg{Hence the expected number of tail observations, {\it i.e.} of observations above the (random or deterministic) threshold, grows with the sample size.} \rae{Informally, this ensures that the threshold does not grow too fast, so that enough data points enter the inference scheme.} The number of such extreme points is ruled by the usual condition $\sqrt{k}A(n/k)=O(1)$ in extreme-value theory, which is fulfilled when \sg{$k \sim cn^{-2\rho/(1-2\rho)}$}, $\rho<0$, $c>0$, \sg{provided $|A(t)|$ is asymptotically proportional to $t^{\rho}$,} see \citet[Equation~(3.2.10)]{Haan2007}.

			\sg{(ii)} Condition~\eqref{hyp:2cgamma} implies the existence of the second moment of $\varphi g(Y)\mathbf{1}_{\{Y\geq y\}}$ for all $y\geq 0$. To this end, one should pick $\tau<{1}/(2\gamma)-\kappa$, this upper bound being either positive or negative. 

			\sg{(iii)} The constraint~\eqref{hyp:qcgamma} is the translation in the regular variation language of the fact that $\delta_{n,k}\to 0$ as $n\to\infty$ regardless of the underlying slowly varying part of $g$ and $\bar{F}$. 
			\sg{Under the choice of $k$ in {\rm(i)}, the rate $\delta_{n,k}$ is of order $ n^{(1/q-\gamma\kappa)/(1-2\rho)}$, up to slowly varying factors.}

\rtworev{%
{\rm(iv)} The rate of Theorem~\ref{th:main} is the maximum of a statistical term, $\delta_{n,k}$, which no refinement of the grid improves, and of the approximation error $N^{-\al_b}$, which depends on the grid only. Since no condition relates $N$ to $n$ or to $k$, the grid may be refined independently of the sample: as soon as $N\ge\delta_{n,k}^{-1/\al_b}$, the rate is $\delta_{n,k}$, that of Proposition~\ref{prop:1} in~\ref{sec-proofs} for fully observed curves, so that discretizing costs nothing asymptotically.
}%

		\end{Remark}

\sg{
\begin{Remark}[On the test function $\vfi$]
\label{rmk:role_phi}
The optimization problem of Proposition~\ref{prop:solution} naturally leads to the weight $\vfi(y)=y$, but there is no reason to restrict to it, so \rae{it is replaced} by a general $\vfi\in\RV_\tau$ in~\eqref{eq:solution2}. The function $\vfi$ tunes the weight applied to each selected tail observation. It is not entirely free: the conditions of Theorem~\ref{th:main} restrict $\tau$ to $-\kappa<\tau<1/(2\gamma)-\kappa$. In particular, $\tau=0$, and thus the unweighted choice $\vfi\equiv 1$, is admissible as soon as $2\kappa\gamma<1$. The main point is that the choice of $\vfi$ does not affect the estimation of $\beta$ at first order. Under the inverse single-index model, the limiting direction
is $\beta$ for every admissible $\vfi$, because the noise does not contribute to the extremes and the conditional mean is aligned with $\beta$ in the tail. The rate $\delta_{n,k}$ is not affected either and in parallel, condition~\eqref{hyp:qcgamma} only involves $\kappa$. Hence no admissible $\vfi$ gives a wrong direction or a slower rate. The function $\vfi$ still keeps two effects. First, it can change the direction, but only outside the single-index model, as in Proposition~\ref{prop:modelfree}. The element $w_\vfi(y)$ points in the direction of the average of $\E(X\mid Y=t)$ over the tail $t\ge y$, each level $t$ being weighted by $\vfi(t)$, and this average depends on $\vfi$ only when the direction of $\E(X\mid Y=t)$ varies with $t$. In contrast, under the single-index model this direction is constant and equal to $\beta$. Second, $\vfi$ changes the finite-sample behavior, since it controls the weight given to the most extreme observations, hence the variance and the effective number of exceedances. In this respect $\vfi$ is a
tuning parameter like the number~$k$ of order statistics, without changing the rate. Finally, this invariance gives a simple check of the model: since the limiting direction does not depend on $\vfi$ under the single-index model, a direction that changes substantially with $\vfi$ indicates that this model does not hold.
\end{Remark}
}

		\cp{
		\begin{Remark}[Comparison with empirical-process inverse regression]\label{rmk:comparison_ep}
The authors in \cite{Portier2016} and \cite{Aghbalou2024} also study the conditional mean of the covariate given the response, but through an
empirical-process lens. \rae{Three essential differences may be distinguished.} First, their setting is finite-dimensional,
whereas \rae{that of FEPLS} is functional. Next and most importantly, they describe the entire
quantile-indexed process and prove a functional central limit theorem, at the
parametric rate in the central regime of \cite{Portier2016} and at the tail
rate in the extreme regime of \cite{Aghbalou2024}. \rae{FEPLS does not target}
the whole process: the estimator is evaluated at a single data-driven extreme
level, where the effective information is carried by the largest
observations rather than by the full sample. \rae{The conclusion of Theorem~\ref{th:main}} is accordingly a
convergence in probability with an explicit rate and not weak
convergence. Finally, concerning the model and tools, the empirical-process route rests on a linearity condition on the covariate and on
bounded index weights. \rae{Theorem~\ref{th:main}} assumes no linearity, centring or
independence of the noise, and replaces them by a heavy-tail,
regularly varying structure, at the price of a constraint linking the admissible
growth of the weight to the tail index. The machinery
is correspondingly that of extreme-value theory, regular variation and order
statistics, rather than Donsker classes and the functional delta method.
\end{Remark}
}

\sg{%
\begin{Example}\label{ex:burr}
Let $Y$ be Burr distributed, with $\bar F(y)=(1+y^{-\rho/\gamma})^{1/\rho}$ for $y\ge 0$, $\gamma\in(0,1)$ and $\rho<0$, so that \eqref{hyp:2rv} and \eqref{hyp:vonmises} hold. Let $g(y)=y^{\kappa}$ with $\kappa>0$ and $\vfi(y)=y^{\tau}$ with $\tau\in\R$, so that \eqref{hyp:link} and \eqref{hyp:test} hold. Finally, let $\eps$ be independent of $Y$ and distributed as $\sigma B^{H}+\mu$, where $\sigma>0$, $\mu:[0,1]\to\R$ is a deterministic Lipschitz function and $B^{H}$ is a fractional Brownian motion on $[0,1]$ with Hurst parameter $H\in(0,1)$. The trajectories of $B^{H}$ are almost surely $\alpha_e$-H\"older continuous for any $\alpha_e<H$ and, by Fernique's theorem, $\E\big(\|B^{H}\|_{0,\alpha_e}^q\big)<+\infty$ for all $q>0$, so that, by independence, condition~\eqref{hyp:noise_cond} holds for every $q>2$. Conditions~\eqref{hyp:2cgamma} and~\eqref{hyp:qcgamma} then reduce to $-\kappa<\tau<1/(2\gamma)-\kappa$ and $q>1/(\kappa\gamma)$: for instance, with $\gamma=1/3$ and $\kappa=1$, every $\tau\in(-1,1/2)$ is admissible. Consequently, for any $\beta\in\mathcal{C}^{0,\alpha_b}([0,1],\R)$ with $\|\beta\|=1$ and $\alpha_b\in(0,1]$, all the assumptions of Theorem~\ref{th:main} are satisfied, and its conclusion holds. %
\rtworev{In this setting where $\kappa=1$, model~\eqref{eq:single_index_model} also admits the forward interpretation $Y=\langle\beta,X\rangle + \tilde\varepsilon$ 
where $\tilde\varepsilon=-\langle \beta,\varepsilon\rangle$ is a Gaussian random variable. A variant of this example, in which the noise scale is allowed to depend on $Y$, underlies the simulation design of Section~\ref{sec-sim}.}
\end{Example}}

\subsection{\sg {A model-free asymptotic result on the empirical tail-moment}}

\rtwomot{%
Tail-moments such as $m_W(y)$ recur in extreme-value statistics: the marginal expected shortfall is estimated by averaging one variable over the observations at which another variable takes its $k$ largest values \citep{all_stars2015}; conditional tail-moments of a heavy-tailed loss serve as extreme risk measures \citep{elmethni}; average excesses above an order statistic are studied for random vectors by \citet{Stupfler2019_Random_threshold}; and the Extreme-PLS method of \citet{Bousebata2023} rests on tail-moments of a covariate in $\R^d$. All these quantities are real-valued or finite-dimensional, whereas Proposition~\ref{prop:modelfree} below holds in a possibly infinite-dimensional separable Hilbert space.
}

\sg{
The consistency established in Theorem~\ref{th:main} and in
Proposition~\ref{prop:1} of \ref{sec-proofs} ultimately relies, \rae{apart from the grid term,} only on the behavior of the empirical tail-moment $\hat m_W$, the inverse single-index model~\eqref{eq:single_index_model} being used solely to identify the limiting direction $\beta$. \rae{This ingredient is} now \rae{isolated} for a generic Hilbert space $H$ and $W\in H$, and \rae{model~\eqref{eq:single_index_model} is dropped} together with any assumption on the direction of the conditional mean.

In the following, $\E_\mathcal{F}$ denotes the conditional expectation given a $\sigma$-field $\mathcal{F}$ and, if moreover the integrand is Banach-valued, $\E_{\mathcal{F},\mathcal{B}}$ \rae{denotes} the Bochner version of the conditional expectation, see \citet{bochner} for the technical details. Let $W:\Omega\to H$ be Bochner-integrable. By construction, $W$ may be decomposed along its
conditional mean as
\begin{align}
	\label{eq:cond_mean}
	W=\mu(Y)+\xi,\qquad \mu(Y):=\E_{\mathcal{F},\mathcal{B}}(W)\ \text{with}\ \mathcal{F}=\sigma(Y),\qquad \xi:=W-\mu(Y),
\end{align}
where $\mu:\R\to H$ denotes the associated regression function and $\E_{\mathcal{F},\mathcal{B}}(\xi)=0$,
so that $m_W(y)=m_{\mu(Y)}(y)$ by conditioning. \rtworev{This decomposition imposes no model on $(W,Y)$, in particular no linearity.}
\rae{Denote} by $s(y):=\|\mu(y)\|$ the norm of the conditional mean and by
$a(y):=\mu(y)/s(y)$ its direction, for $y$ large enough so that $s(y)>0$. \rtworev{Beyond the non-degeneracy condition~\eqref{hyp:nondeg} below, no} assumption is made on
$a(y)$: it need not converge. \rtworev{The target direction is then}
\begin{align*}
	\bar a(y):=\f{m_W(y)}{\|m_W(y)\|},\qquad\text{whenever } m_W(y)\neq 0,
\end{align*}
\rtworev{the normalized average of the directions $a(Y)$ over $\{Y\ge y\}$, weighted by $s(Y)$. It may move with $y$, and it is constant when $a$ is constant in the tail, as for $W=X\vfi(Y)$ under model~\eqref{eq:single_index_model} with conditionally centered noise, where $\bar a\equiv\beta$.}
\rtworev{Concerning the residual $\xi$, it is assumed that there exists $q>2$ such that}
\begin{align}
	\label{hyp:res_cond}
	\sup_{y} \E\left(\|\xi\|^{q}\mid Y = y\right)<+\infty,
\end{align}
where the supremum is over the support of $Y$\rtworev{, and, in the spirit of~\eqref{hyp:link}, \eqref{hyp:2cgamma} and~\eqref{hyp:qcgamma}, that $s:\R\to\R^+$ is ultimately positive with}
\begin{align}
	\label{hyp:mu_norm}
	s\in \RV_\lambda(+\infty),\qquad \f{1}{q\gamma}<\lambda<\f{1}{2\gamma}.
\end{align}
Beyond the conditional centring $\E_{\mathcal{F},\mathcal{B}}(\xi)=0$, which holds by construction, and the moment control~\eqref{hyp:res_cond}, no assumption is made on the residual $\xi$; in particular, it is not assumed independent of~$Y$. Finally, \rae{define} $c_W(y):=\|m_W(y)\|/m_{s(Y)}(y)\in[0,1]$ and consider the non-degeneracy condition
\begin{align}
	\label{hyp:nondeg}
	\liminf_{y\to +\infty} c_W(y)>0 ;
\end{align}
\rtworev{it rules out an asymptotic cancellation of the directions $a(t)$, $t\ge y$, and holds with $c_W\equiv1$ when $a$ is constant in the tail.}
}

\rtworev{
\begin{Prop}\label{prop:modelfree}
Assume $  U\in 2\RV_{\gamma,\rho}(+\infty)$ and
$	f\in\RV_{-1/\gamma - 1}(+\infty)$ with $\gamma>0$ and $\rho\leq 0$.
	Suppose that~\eqref{eq:cond_mean}--\eqref{hyp:nondeg} hold  and that $s$ is continuously differentiable in a neighborhood of infinity with $ts'(t)/s(t)\to\lambda$ as $t\to+\infty$. Let $k:=k_n\to+\infty$ be a deterministic integer sequence with $k/n\to0$ and $\sqrt{k}A(n/k)=O(1)$. Then the event on which $m_W(Y_{n-k+1,n})\neq0$ and $\hat m_W(Y_{n-k+1,n})\neq0$ has probability tending to one and, the left-hand side below being set to $0$ off that event,
	\begin{align*}
		\Big\|\f{\hat m_W(Y_{n-k+1,n})}{\|\hat m_W(Y_{n-k+1,n})\|}-\bar a(Y_{n-k+1,n})\Big\|=O_\p\big(k^{-1/2}\big).
	\end{align*}
\end{Prop}
}

\rtworev{%
Theorem~\ref{th:main} does not depend on Proposition~\ref{prop:modelfree} but rather rests on Propositions~\ref{prop:1} and~\ref{prop:2} of~\ref{sec-proofs}.
}
\rtwomot{%
Proposition~\ref{prop:modelfree} is moreover general: it assumes neither a model for $(W,Y)$ nor a limit for the direction $a(t)$, which may keep moving as long as~\eqref{hyp:nondeg} holds. The price is the target, the tail direction $\bar a(Y_{n-k+1,n})$, which moves with $n$. Moving targets also appear in results on tail-moments at intermediate levels: \citet[Theorem~2]{Stupfler2019_Random_threshold} compares an empirical expected $f$-shortfall above $Y_{n-k,n}$ with its population counterpart at the deterministic level $U(n/k)$.
}

\rtworev{%
Under model~\eqref{eq:single_index_model} with conditionally centered noise, $\E_{\mathcal{F},\mathcal{B}}(\eps)=0$, the choice $W=X\vfi(Y)$ gives $\mu(Y)=\vfi(Y)g(Y)\beta$, hence $s=\vfi g$, $\lambda=\tau+\kappa$, $\xi=\vfi(Y)\eps$, $\bar a\equiv\beta$ and $c_W\equiv1$: condition~\eqref{hyp:nondeg} holds, \eqref{hyp:mu_norm} reads $1/(q\gamma)<\tau+\kappa<1/(2\gamma)$, with $q$ as in~\eqref{hyp:res_cond}, and Proposition~\ref{prop:modelfree} recovers the consistency of $\hat\beta_\vfi$ toward $\beta$, provided~\eqref{hyp:res_cond} holds for $\vfi(Y)\eps$, which~\eqref{hyp:noise_cond} does not ensure in general.
}

\rtworev{%
No discretization enters Proposition~\ref{prop:modelfree}, whose estimator is built from the exact variables $W_i$: the grid contributes to Theorem~\ref{th:main} only the separate term $N^{-\al_b}$, through Proposition~\ref{prop:2} in the Appendix. The term $\delta_{n,k}$ of Theorem~\ref{th:main} is thus not a grid term but the price of the noise of~\eqref{eq:single_index_model}, which is not assumed conditionally centered: the empirical tail-moment of $\vfi(Y)\eps$ at $Y_{n-k+1,n}$ is $O_\p(\delta_{n,k})$ relative to $m_{\vfi g(Y)}(y_{n,k})$, by Lemma~\ref{lem:norm_noise} in the Appendix. In Proposition~\ref{prop:modelfree}, by contrast, the residual $\xi$ is conditionally centered and contributes only $o_\p(k^{-1/2})$ relative to $m_{s(Y)}(y_{n,k})$, so that the rate $k^{-1/2}$ is that of the empirical tail-moment of the conditional mean. Finally, by~\eqref{hyp:res_cond} and the lower bound in~\eqref{hyp:mu_norm}, the tail index of $\|\xi\|$, when it exists, is at most $1/q$, with $q$ the exponent of~\eqref{hyp:res_cond}, below the tail index $\lambda\gamma$ of $s(Y)$: the residual is lighter-tailed than the conditional mean, which is the model-free counterpart of the constraint $\gamma_{g(Y)}>\gamma_{\|\eps\|}$ discussed in Remark~\ref{rmk:noise_extremes}.
}

\sg{
\begin{Remark}[Bivariate functional covariates]
\label{rmk:multivariate}
\rtwomot{%
Propositions~\ref{prop:solution} and~\ref{prop:modelfree}, and Proposition~\ref{prop:1} of~\ref{sec-proofs}, use only the inner product of $H$ and the norm of the noise: they hold in any separable Hilbert space, \eqref{hyp:noise_cond} being read with the norm of $H$, for instance for a bivariate covariate $X=(X^{(1)},X^{(2)})$ with $H=L^2([0,1],\R^2)$. Theorem~\ref{th:main} carries over with the reconstruction on a grid carried out component-wise: the bound~\eqref{eq:prop2-step1} of~\ref{sec-proofs} holds for each component, with $\beta^{(j)}$ and $\eps^{(j)}_i$ in place of $\beta$ and $\eps_i$, the H\"older conditions on $\beta$ and in~\eqref{hyp:noise_cond} being read component-wise, and $\|\tilde X_i-X_i\|\le\|\tilde X^{(1)}_i-X^{(1)}_i\|+\|\tilde X^{(2)}_i-X^{(2)}_i\|$, so that the rate is the worst of the two component rates. Model~\eqref{eq:single_index_model}, however, imposes the same link $g$ on both components.}
\rtworev{With one link per component, $X^{(j)}=g_j(Y)\beta^{(j)}+\eps^{(j)}$, $j=1,2$, with $g_j:\R\to\R^+$, $g_j\in\RV_{\kappa_j}(+\infty)$, $\kappa_1>\kappa_2$, $\beta^{(1)}\neq0$, a Bochner-integrable noise $\eps=(\eps^{(1)},\eps^{(2)})$ with $\E_{\mathcal{F},\mathcal{B}}(\eps)=0$, $\vfi$ as in~\eqref{hyp:test} and $W=X\vfi(Y)$ Bochner-integrable, one may show, under~\eqref{hyp:2rv}, that the direction $\bar a(y)$ of the tail-moment of $W$ tends to $(\beta^{(1)}/\|\beta^{(1)}\|,0)$ as $y\to+\infty$ (details omitted). Since $Y_{n-k+1,n}\to+\infty$ in probability, Proposition~\ref{prop:modelfree}, under its other assumptions with $\lambda=\tau+\kappa_1$, then yields that $\hat\beta_\vfi(Y_{n-k+1,n})$ converges in probability to $(\beta^{(1)}/\|\beta^{(1)}\|,0)$: FEPLS recovers the direction of the component with the heavier link.}
\end{Remark}
}

		\section{Illustration on simulated data} \label{sec-sim}

		The performance of \rae{the} method is assessed by a Monte Carlo simulation experiment with $500$ independent repetitions. The ambient space is $H=L^2([0,1])$ endowed with its usual inner product $\langle f, g \rangle = \int_0^1 f(t)g(t)dt$. Functional observations are discretized on a regular grid of length $N \in \{101, 1001\}$ with nodes $t_j = (j-1)/(N-1)$ for $1\le j\le N$. \cp{To implement the $L^2$-inner product in the discretized setting, \rae{consider} the diagonal positive definite matrix $\Psi\in\mathbb{R}^{N\times N}$ given by $\Psi_{j,j'}=|I_j|\,1_{\{j=j'\}}$ for all $(j,j')\in\{1,\dots,N\}^2$, so that $\langle \tilde X_i,\tilde X_\ell\rangle = \mathcal{X}_i^\top\Psi\,\mathcal{X}_\ell$ exactly, where $\mathcal{X}_i:=(X_i(t_1),\dots,X_i(t_N))^\top$, $1\le i \le n$.} For any discretized functions $u=(u(t_j))_{1\le j\le N}, v=(v(t_j))_{1\le j\le N}\in\mathbb{R}^N$, the quadratic form $\langle u,v\rangle_{\Psi} := u^\top \Psi v$ induces an inner product and all subsequent computations use $\|\cdot\|_{\Psi} := \sqrt{\langle\cdot,\cdot\rangle_{\Psi}}$ as the induced norm.

		\subsection{Experimental design}
		\label{sub-expe}

		\rae{A data sample} of $n=500$ independent replications $(X_i,Y_i)$ of $(X,Y)$ \rae{is drawn} using the following scheme:
		\begin{itemize}
			\item $Y$ has a Burr distribution, \emph{i.e.}, $\bar{F}(y)=(1+y^{-\rho/\gamma})^{1/{\rho}} \in 2\RV_{-1/\gamma,\rho/\gamma}(+\infty)$ with ${y}\ge 0$, $\gamma \in (0,1)$ being the tail-index and $ \rho <0$ being the second-order parameter.  Here, $\gamma\in\{1/3, 1/2, 9/10\}$ while $\rho \in\{-2\gamma,-\gamma/2\}$. 
			\item $\eps\mid Y={y}$ has the same distribution as $\sigma ({y})B^{H}_{{y}}+\mu({y})$ where $B^H_{{y}}$ is the fractional Brownian motion (fBm) on $[0,1]$ with conditional Hurst parameter $H({y})\in (0,1)$, $\sigma({y})$ is the conditional noise deviation and $\mu({y})$ the conditional mean. 
			\item $\beta:t\in[0,1]\mapsto \sqrt{2}\sin(2\pi t)$ and $g:y\in \R^+ \mapsto y^\kappa$ with $\kappa\in \{ 1,3/2,2\}$.
		\end{itemize}
		Hence, conditionally on $Y=y$, the noise process based on $B^H_{{y}}$ is centered Gaussian with conditional covariance function 
		$$
		(s,t)\mapsto \f{\sigma^2({{y}})}{2}\left( t^{2H({{y}})}+s^{2H({{y}})}-\left|t-s\right|^{2H({{y}})}\right).
		$$
		Assumption~\eqref{hyp:noise_cond} thus holds for all $q>0$.
		\rae{The choice} $H({{y}})=1/2$ yields the Brownian motion $B_{{y}}$ with mean $\mu({{y}})$ and variance $\sigma^2({{y}})$.

		Once both the response variable $Y$ and the noise $\eps$ are simulated, and the deterministic functions $\beta$ and $g$ are chosen, the covariate sampling $\{\tilde X_1,\dots,\tilde X_n\}$ is readily derived from~\eqref{eq:single_index_model}. 
		The choice of $\sigma$ is guided by the following reasonable requirement: One should ensure that the contribution in \eqref{eq:single_index_model} of $\eps$ does not overwhelm $g(Y)$\rtworev{, and that $g(Y)$ does not overwhelm the contribution of $\eps$ either}, in order to avoid a trivial case where the noise is not impacting. To this end, $\sigma({{y}})=g({{y}})/10$ \rae{is proposed}. Moreover, $(H,\mu)\equiv (1/3,200)$ \rae{is fixed} regardless of $Y=y$.

		The FEPLS estimator is implemented with $\varphi(y)=y^\tau$ as test function. The selected values of $\tau$ depend on the pair $(\gamma,\kappa)$ in order to fulfill condition~\eqref{hyp:2cgamma}. They are given in Table~\ref{tab-tau}.
		\begin{table}[ht]
			\centering
			\begin{tabular}{lrrr}
				\hline
				& $\gamma=1/3$ & $\gamma=1/2$ & $\gamma=9/10$ \\
				\hline
				$\kappa=1$  & $\{0,-1,-2\}$ &  $\{-1,-2,-3\}$ &  $\{-1,-2,-3\}$ \\
				$\kappa=3/2$  & $\{-1,-2,-3\}$ &   $\{-1,-2,-3\}$ &  $\{-1,-2,-3\}$ \\
				$\kappa=2$  &  $\{-1,-2,-3\}$ &  $\{-2,-3,-4\}$ &   $\{-2,-3,-4\}$ \\
				\hline
			\end{tabular}
			\caption{Selected values of $\tau$ in the nine simulated situations.}
			\label{tab-tau}
		\end{table}

		\subsection{Selection of the threshold}
		\label{sub-seuil}

		\rae{Consider} the following data-driven selection rule of $k\in\{1,\dots,n\}$ based on the maximization of the sample covariance between  $Y$ and {$\langle \tilde \beta_{\vfi}( Y_{n-k+1,n}),\tilde X  \rangle$} when $Y\geq Y_{n-k+1,n}$:
		\begin{align*}
			\tilde k := \argmax_{5\le k\le n/5} r(k),
		\end{align*}
		with 
		\begin{align*}
			r(k)=\frac{1}{k} \sum_{i=1}^k Y_{n-i+1,n} \langle \tilde \beta_{\vfi}(Y_{n-k+1,n}),\tilde X_{(n-i+1,n)}\rangle - \frac{1}{k} \sum_{i=1}^k Y_{n-i+1,n} \frac{1}{k} \sum_{i=1}^k \langle \tilde \beta_{\vfi}(Y_{n-k+1,n}),\tilde X_{(n-i+1,n)} \rangle,
		\end{align*}
		and where $\tilde X_{(n-i+1,n)}$ denotes the concomitant of $Y_{n-i+1,n}$, \emph{i.e.}, the random variable $\tilde X_s$ with $s\in \{1,\ldots,n\}$ being the unique index such that $Y_s=Y_{n-i+1,n}$.
		\rae{The} above selection rule imposes $\tilde k\geq 5$ in order to prevent instabilities of
		estimates built on too few data points.
		\begin{Remark}   \rae{Denote} ${\mathcal X}_i=(X_i(t_1),\dots,X_i(t_N))^\top$. For any $i,\ell\in\{1,\dots,n\}$, one has $\langle \tilde X_i,\tilde X_\ell\rangle = \langle \mathcal{X}_i, \mathcal{X}_\ell\rangle_\Psi$ so that
			\begin{align*}
				\| \cp{\hat m_{\tilde X\vfi(Y)}}(y)\|^2
				&=\frac{1}{n^2} \sum_{i,\ell=1}^n \langle \tilde X_i,\tilde X_\ell\rangle \vfi(Y_i)\vfi(Y_\ell)1_{\{ Y_i\ge y \}}1_{\{ Y_\ell\ge y \}} \\
				&=\frac{1}{n^2} \sum_{i,\ell=1}^n \langle \mathcal{X}_i, \mathcal{X}_\ell\rangle_\Psi \vfi(Y_i)\vfi(Y_\ell)1_{\{ Y_i\ge y \}}1_{\{ Y_\ell\ge y \}}\\
				&= \left\| {\mathcal V}_\vfi(y)\right\|_\Psi^2,
			\end{align*}
			where 
			$
			{\mathcal V}_\vfi(y):= \frac{1}{n} \sum_{i=1}^n  \mathcal{X}_i  \vfi(Y_i) 1_{\{ Y_i\ge y \}} \in \R^N
			$
			can be interpreted as a sampled version of $\cp{\hat m_{\tilde X\vfi(Y)}}(y)$. \rae{To compute} $\langle \tilde \beta_{\vfi}( y),\tilde X_s \rangle$, \rae{one writes}, for the numerator, 
			\begin{align*}
				\langle \cp{\hat m_{\tilde X\vfi(Y)}}( y),\tilde X_s \rangle &= \ff{n} \sum_{i=1}^n  \langle \tilde X_i, \tilde X_s \rangle \vfi(Y_i)1_{\{ Y_i\ge y \}}
				= \ff{n} \sum_{i=1}^n  \langle  {\mathcal X}_i, {\mathcal X}_s \rangle_\Psi \vfi(Y_i)1_{\{ Y_i\ge y \}},
			\end{align*}
			the denominator being already given by:
			$   \| \cp{\hat m_{\tilde X\vfi(Y)}}(y)\| = \left\| {\mathcal V}_\vfi(y) \right\|_\Psi$.
		\end{Remark}

		\subsection{Results}

		Figure~\ref{fig:beta_estim_exceedance} displays the dot product curves 
		$k \in\{1,\dots,250\}\mapsto\langle \tilde \beta_{\vfi}(Y_{n-k+1,n}),\beta\rangle$
		averaged over the $500$ Monte Carlo replications of the nine considered simulated situations.

		\begin{Remark} Computation of $\langle \tilde \beta_{\vfi}(y),\beta\rangle$.
			\rae{Consider} the numerator:
			\begin{align*}
				\langle \cp{\hat m_{\tilde X\vfi(Y)}}( y), \beta \rangle = \ff{n} \sum_{i=1}^n  \langle \tilde X_i, \beta \rangle \vfi(Y_i)1_{\{ Y_i\ge y \}} =  \ff{n} \sum_{i=1}^n   \sum_{j=1}^N X_i(t_j) \cp{\langle 1_{I_j}, \beta \rangle} \vfi(Y_i)1_{\{ Y_i\ge y \}},
			\end{align*}
			with
			\begin{align*}
				\cp{\langle  1_{I_j} , \beta \rangle = \int_{I_j} \beta(u)\dd u =: \bar \beta_j,}
			\end{align*}
			and thus
			$
			\langle  \cp{\hat m_{\tilde X\vfi(Y)}}( y) , \beta \rangle = {\mathcal V}_\vfi(y) ^\top \bar\beta.
			$
		\end{Remark}
		It appears that the estimated index $\tilde \beta_{\vfi}(Y_{n-k+1,n})$ is close to the true one
		only for small values of $k$, or equivalently for large values of the threshold $Y_{n-k+1,n}$.
		This phenomenon is coherent with the fact that $\langle \beta, X\rangle$ is designed to capture 
		the tail-information of $Y$. 
		\rae{The} range of accurate values for $k$ depends on the difficulty of the estimation problem.
		The higher $\gamma$ and $\kappa$, the heavier the tail of $g(Y)$, and the easier the estimation of $\beta$ is.
		In practice, the selection of $k$ is performed using the above described selection procedure since 
		$\beta$ is unknown and therefore Figure~\ref{fig:beta_estim_exceedance} cannot be used in real data situations.

		Finally, the mean over the 500 replications of the curves $t\in[0,1]\mapsto\tilde \beta_{\vfi}(Y_{n-\tilde k+1,n})(t)$ is compared to the true one $t\in[0,1]\mapsto \beta(t)$ in Figure~\ref{fig:beta_estim_plot_conc}
		in the case where $\tau=-2$.
		It appears that the selected values of $k$ yield very accurate estimations on all nine considered situations whatever the values of $\gamma$ or $\kappa$. Pushing further, Figure~\ref{fig:beta_estim_plot_conc1001} is analogous to Figure~\ref{fig:beta_estim_plot_conc} when the underlying grid is of size $N=1001$. This aims at better reflecting the high-dimensional regime, since the sample size is $n=500$. One may observe a slight degradation of the estimation in comparison to Figure~\ref{fig:beta_estim_plot_conc} mostly when $\kappa=1$. It appears again that the higher $\gamma$ (heavier tail), the easier the inference becomes. Similarly, the larger $\kappa$, the more accurate the inference. Indeed, in view of the inverse regression model $X=g(Y)\beta+\eps$ and the composition rule for regularly varying functions, $\kappa$ also contributes to the leading tail-index.

		\begin{figure}[!htbp]
			\centering
			\begin{minipage}{0.32\textwidth}
				\centering
				\includegraphics[width=\linewidth]{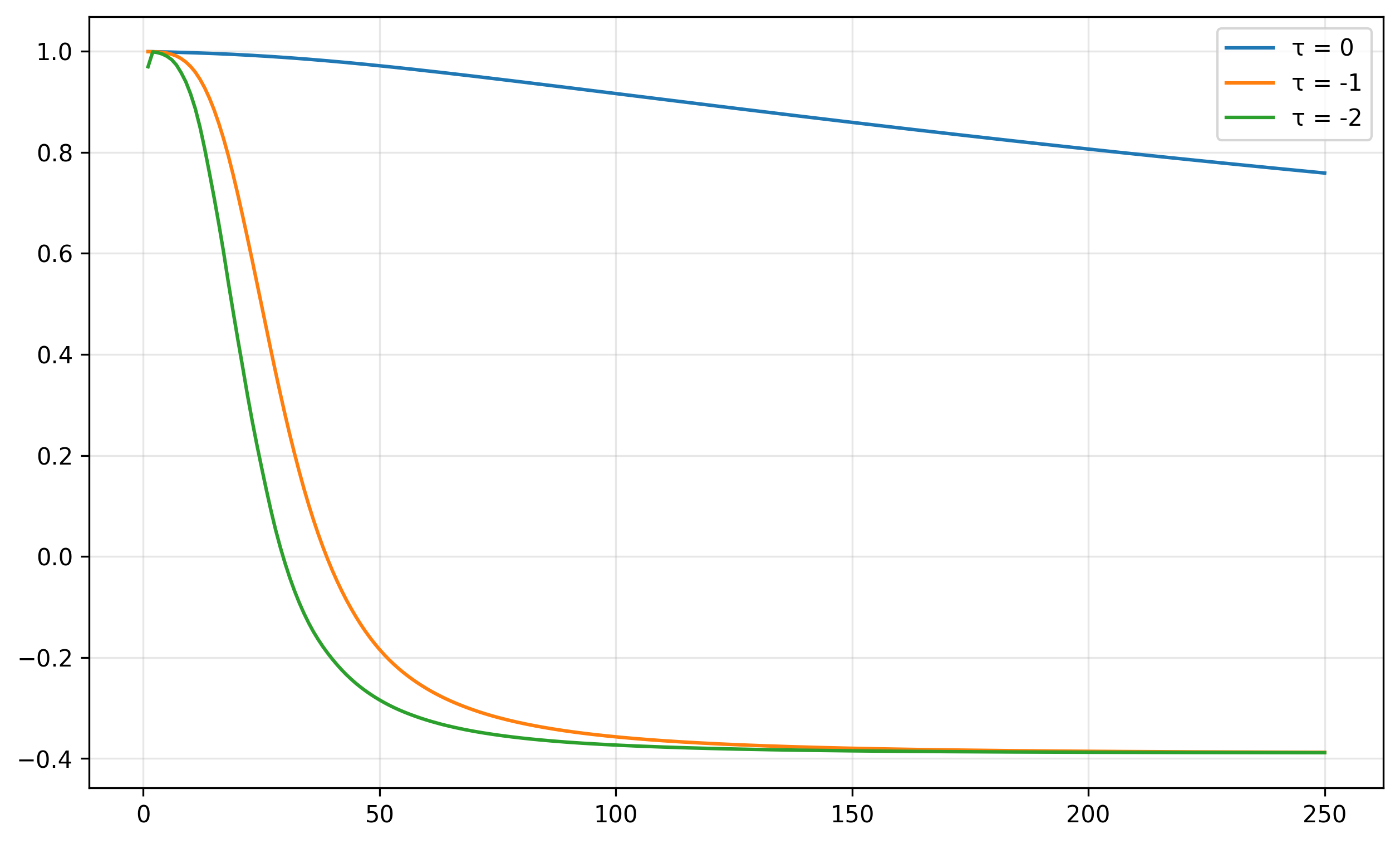}
				\par\medskip
				\centering\textbf{(a)} $\kappa=1$, $\gamma = 1/3$.
			\end{minipage}
			\hfill
			\begin{minipage}{0.32\textwidth}
				\centering
				\includegraphics[width=\linewidth]{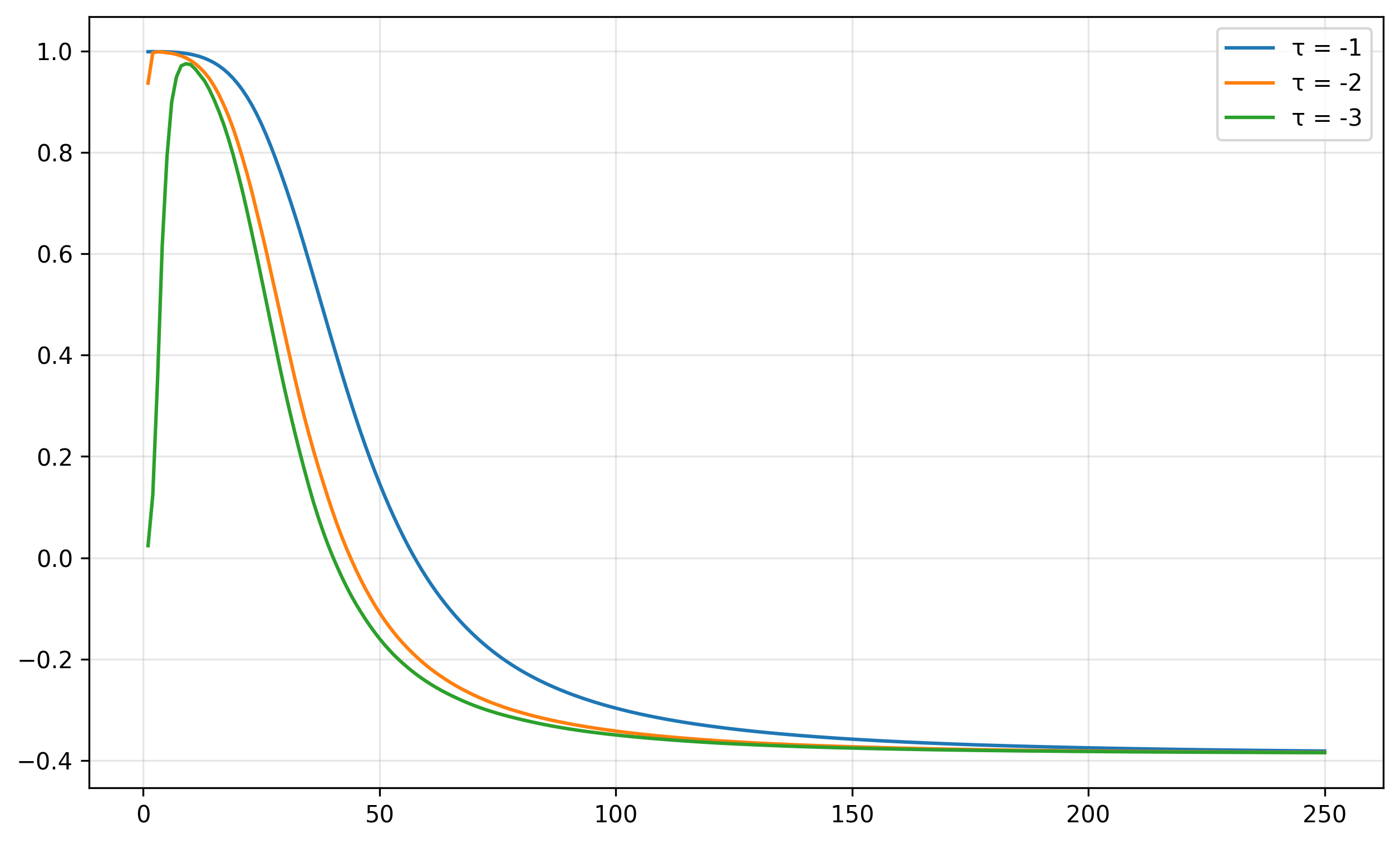}
				\par\medskip
				\centering\textbf{(b)} $\kappa=1$, $\gamma = 1/2$.
			\end{minipage}
			\hfill
			\begin{minipage}{0.32\textwidth}
				\centering
				\includegraphics[width=\linewidth]{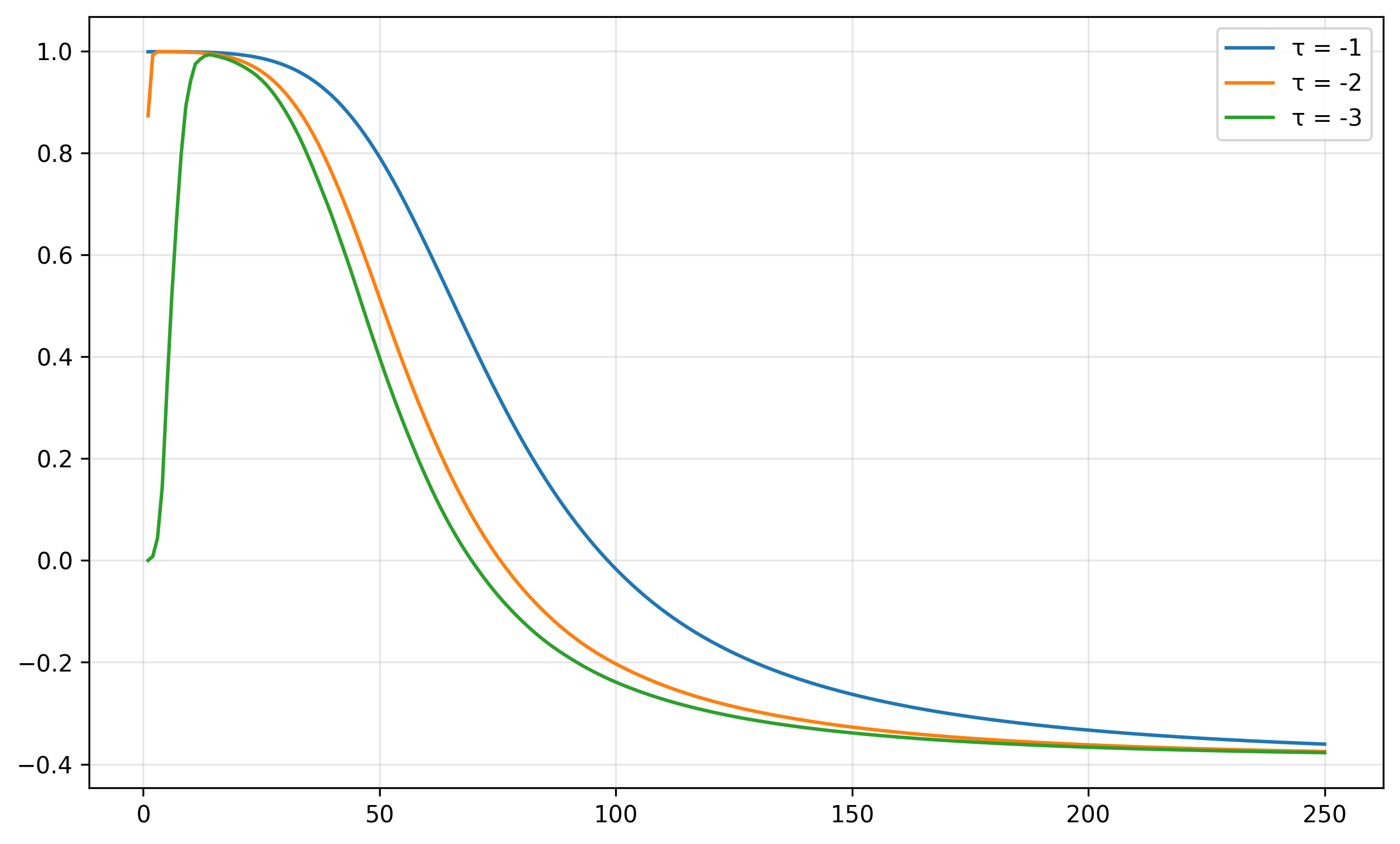}
				\par\medskip
				\centering\textbf{(c)} $\kappa=1$, $\gamma = 9/10$.
			\end{minipage}

			\medskip
			\begin{minipage}{0.32\textwidth}
				\centering
				\includegraphics[width=\linewidth]{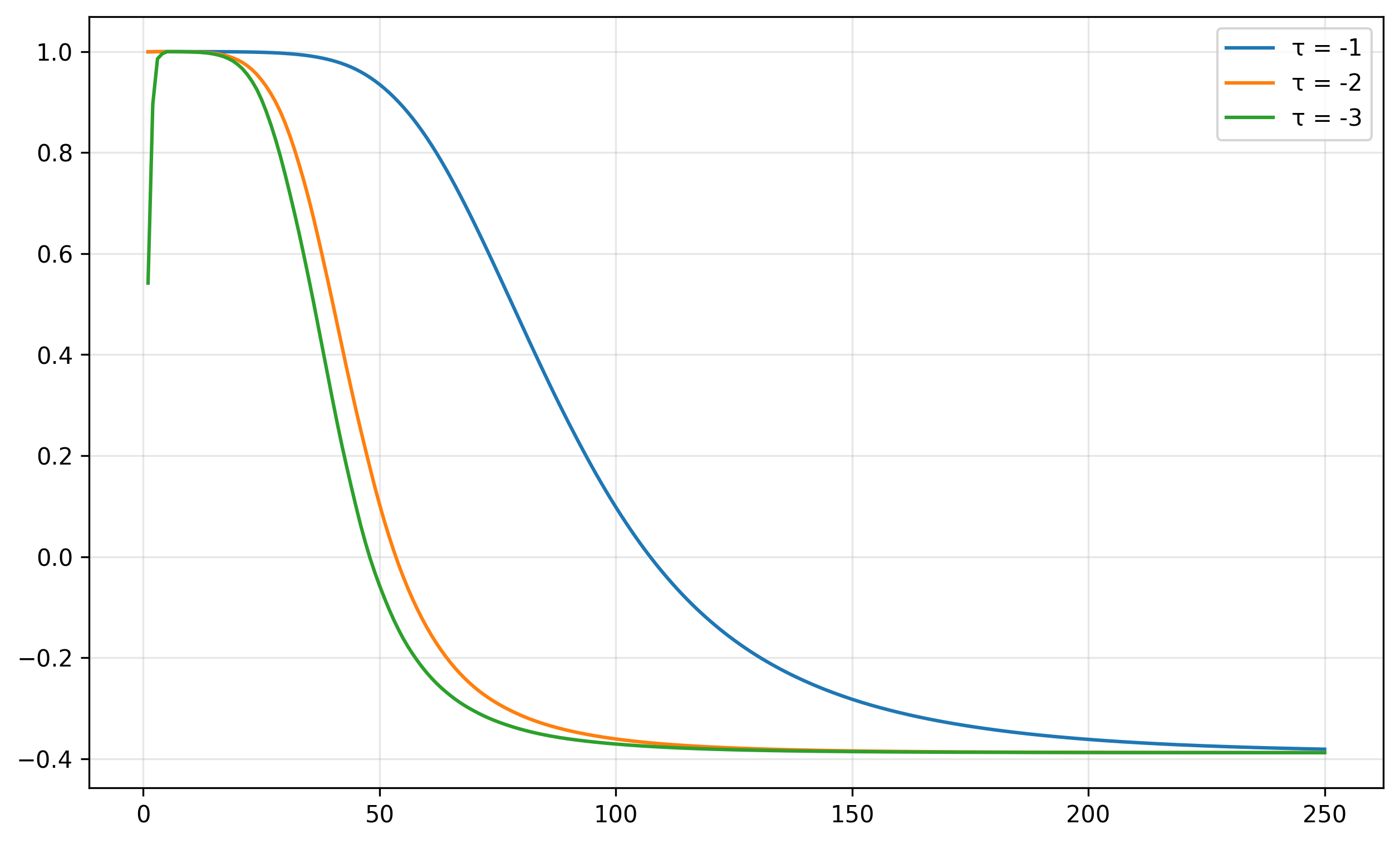}
				\par\medskip
				\centering\textbf{(d)} $\kappa=3/2$, $\gamma = 1/3$.
			\end{minipage}
			\hfill
			\begin{minipage}{0.32\textwidth}
				\centering
				\includegraphics[width=\linewidth]{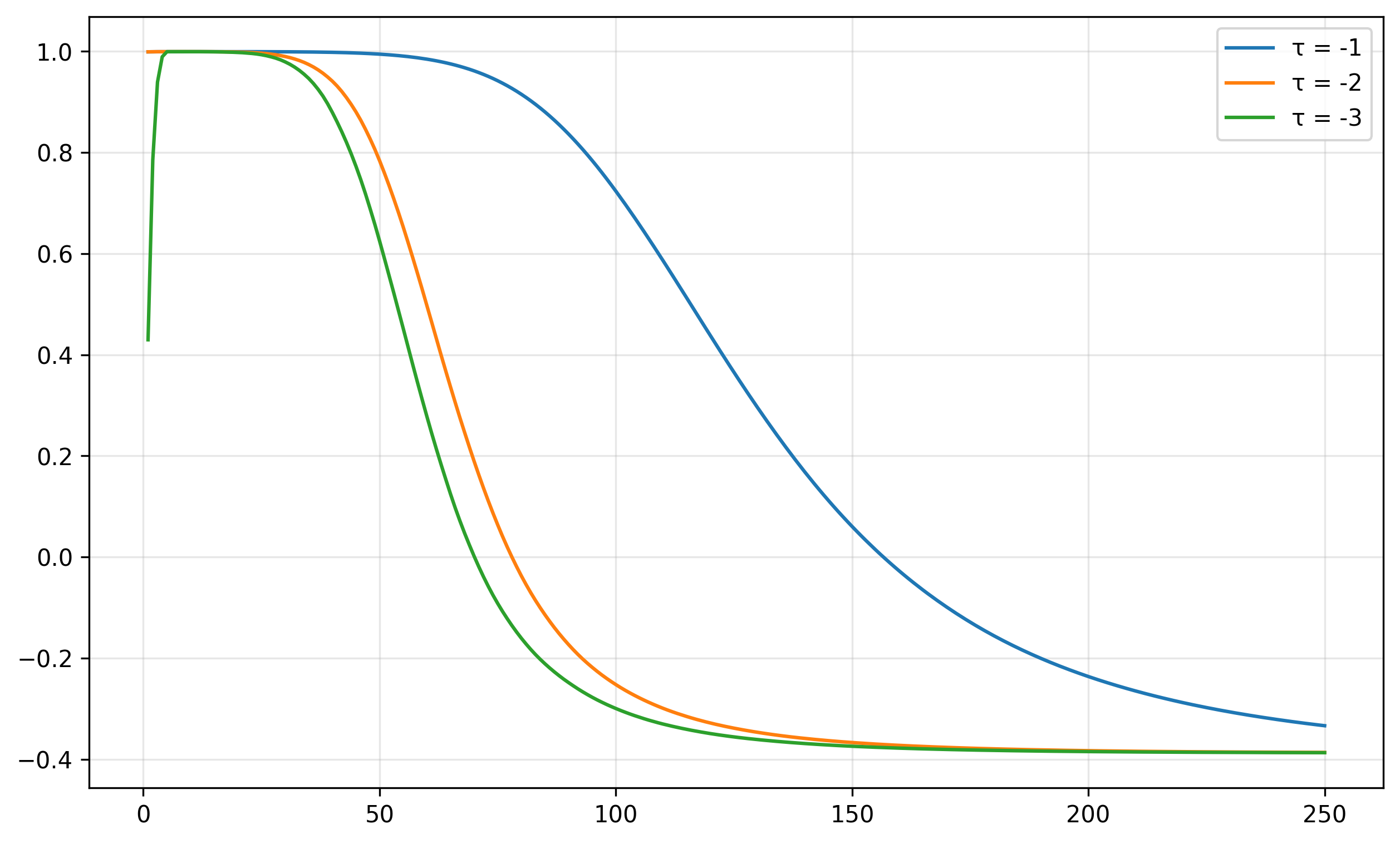}
				\par\medskip
				\centering\textbf{(e)} $\kappa=3/2$, $\gamma = 1/2$.
			\end{minipage}
			\hfill
			\begin{minipage}{0.32\textwidth}
				\centering
				\includegraphics[width=\linewidth]{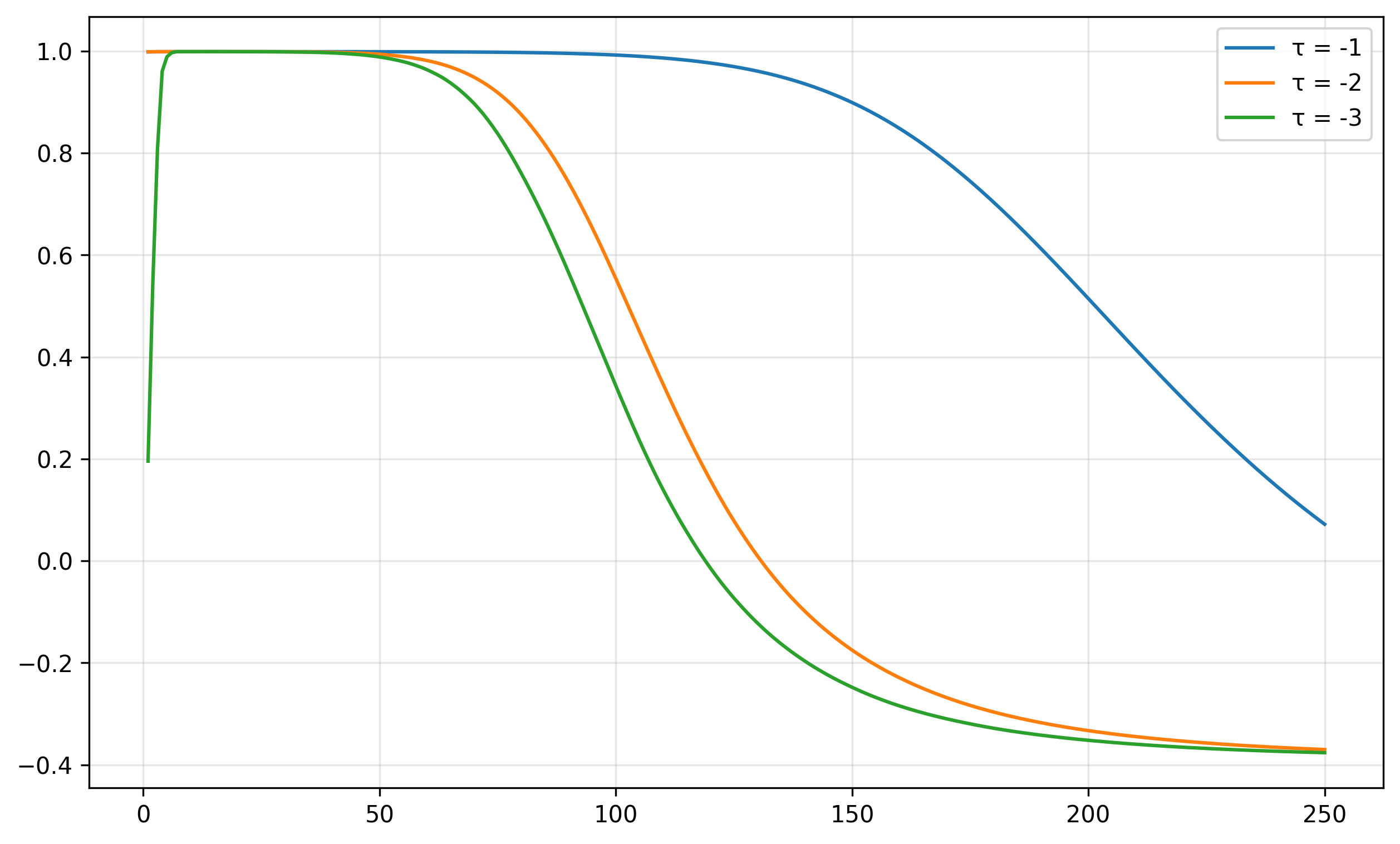}
				\par\medskip
				\centering\textbf{(f)} $\kappa=3/2$, $\gamma = 9/10$.
			\end{minipage}
			\medskip
			\begin{minipage}{0.32\textwidth}
				\centering
				\includegraphics[width=\linewidth]{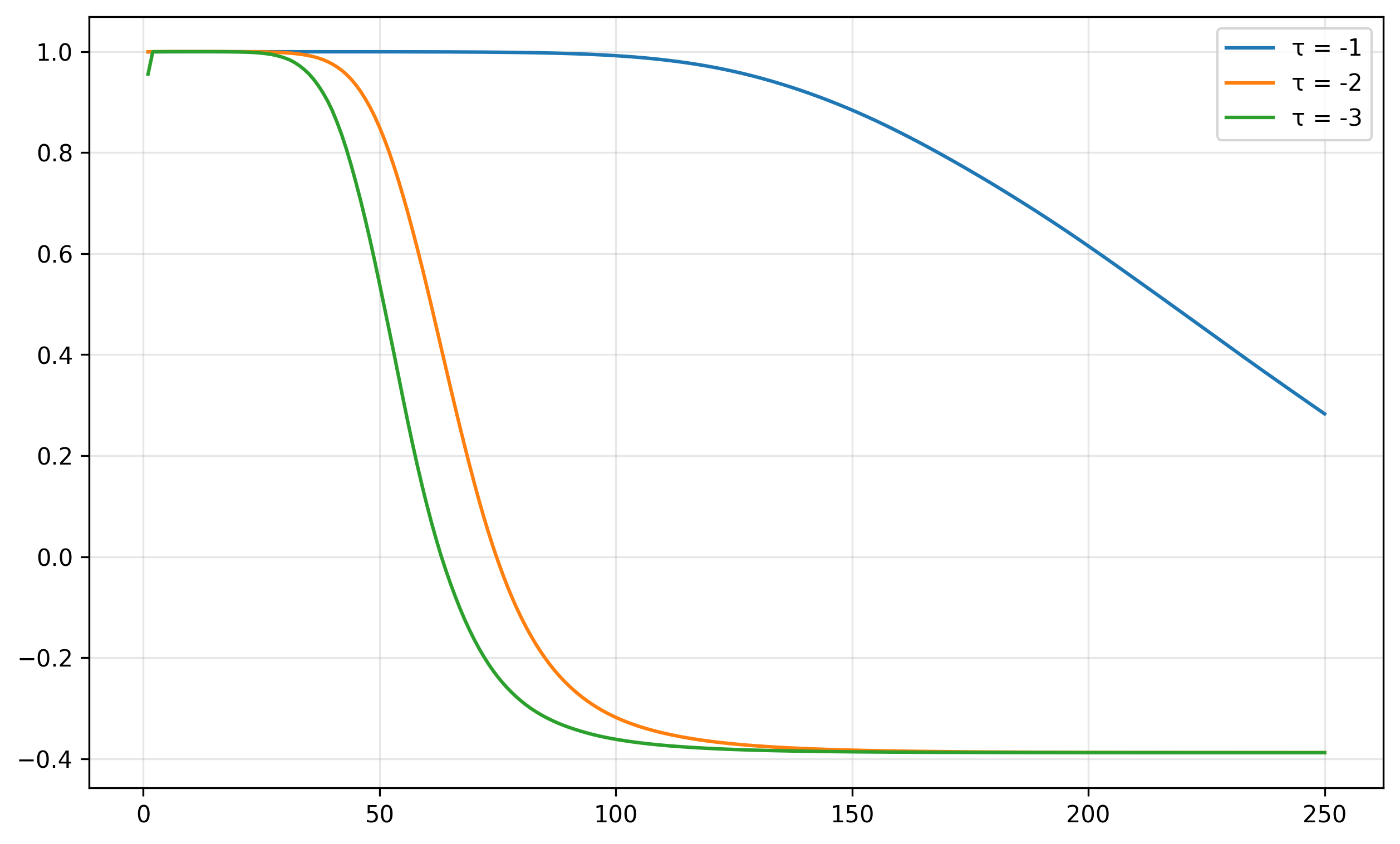}
				\par\medskip
				\centering\textbf{(g)} $\kappa=2$, $\gamma = 1/3$.
			\end{minipage}
			\hfill
			\begin{minipage}{0.32\textwidth}
				\centering
				\includegraphics[width=\linewidth]{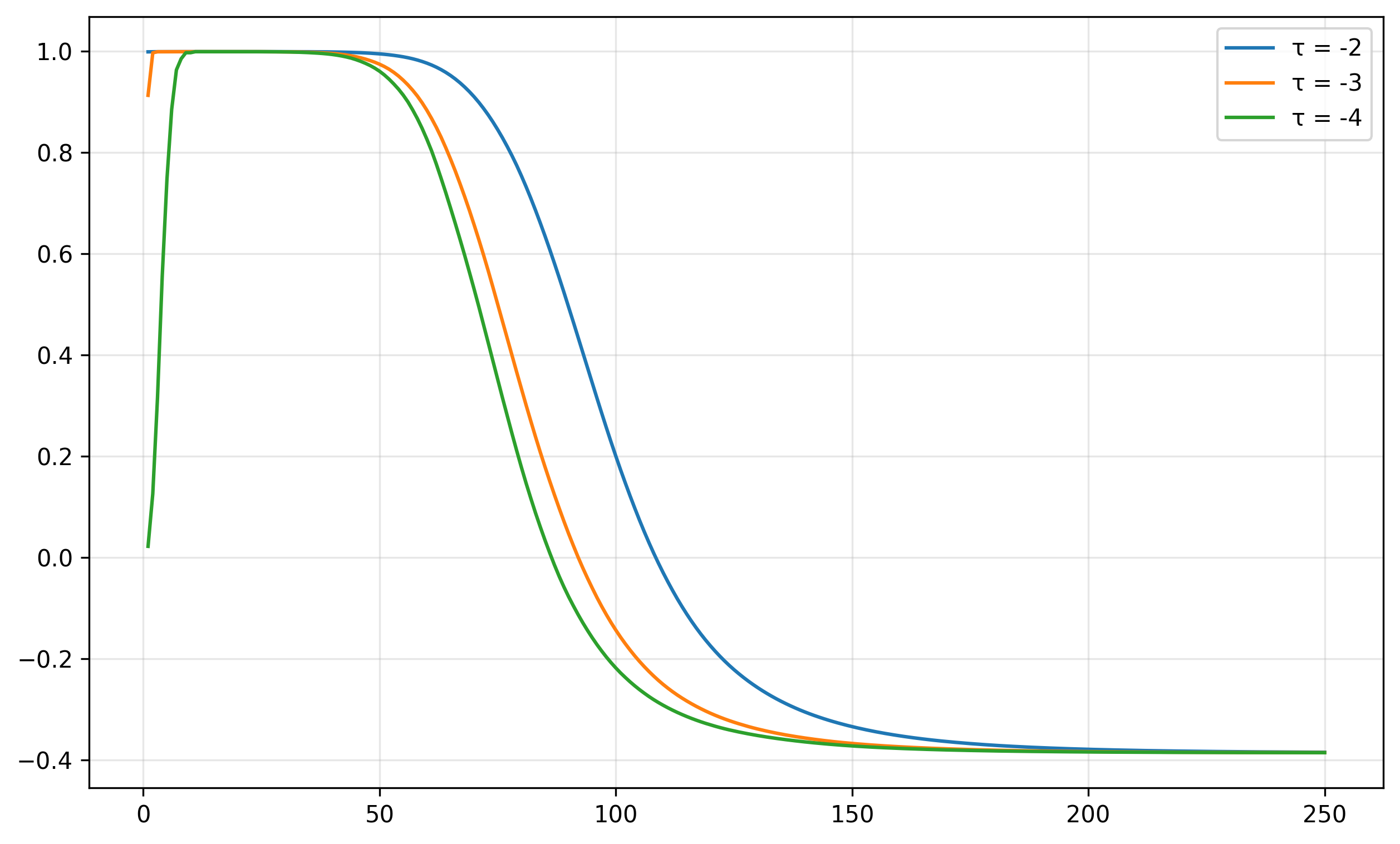}
				\par\medskip
				\centering\textbf{(h)} $\kappa=2$, $\gamma = 1/2$.
			\end{minipage}
			\hfill
			\begin{minipage}{0.32\textwidth}
				\centering
				\includegraphics[width=\linewidth]{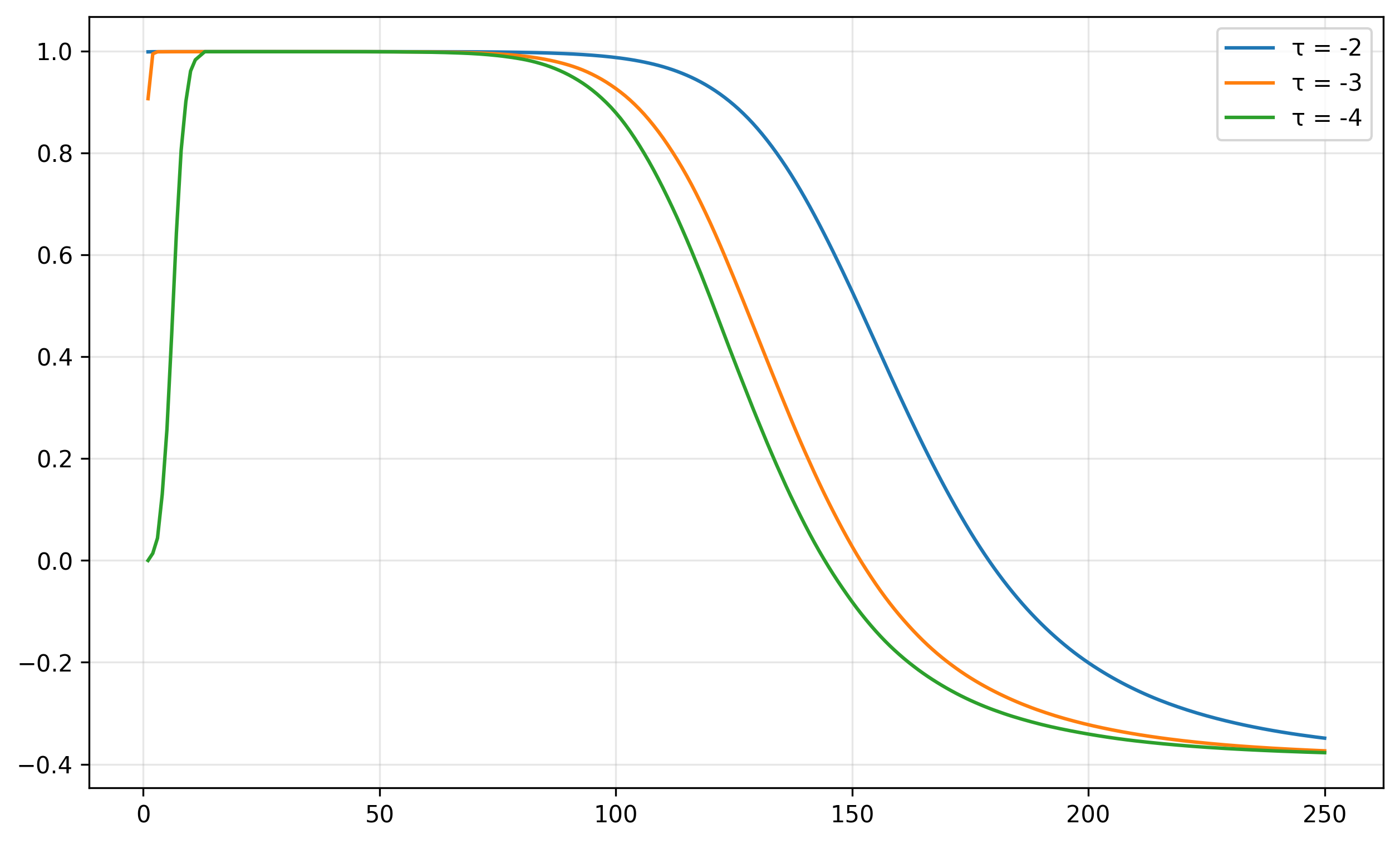}
				\par\medskip
				\centering\textbf{(i)} $\kappa=2$, $\gamma = 9/10$.
			\end{minipage}

			\caption{Finite-sample behavior of $\tilde \beta_{\vfi}(Y_{n-k+1,n})$ for different admissible values of $\tau$ (see Table~\ref{tab-tau}) under the simulated inverse model with Burr response and conditional fBm noise. The $x$-axis represents the number $k$ of exceedances and the $y$-axis is  the average value of $\langle \tilde \beta_{\vfi}(Y_{n-k+1,n}),\beta\rangle$ over $500$ Monte Carlo replications. Here, $\rho=-1/2$ and $N=101$.}  
			\label{fig:beta_estim_exceedance}
		\end{figure}

		\begin{figure}[!htbp]
			\centering
			\begin{minipage}{0.32\textwidth}
				\centering
				\includegraphics[width=\linewidth]{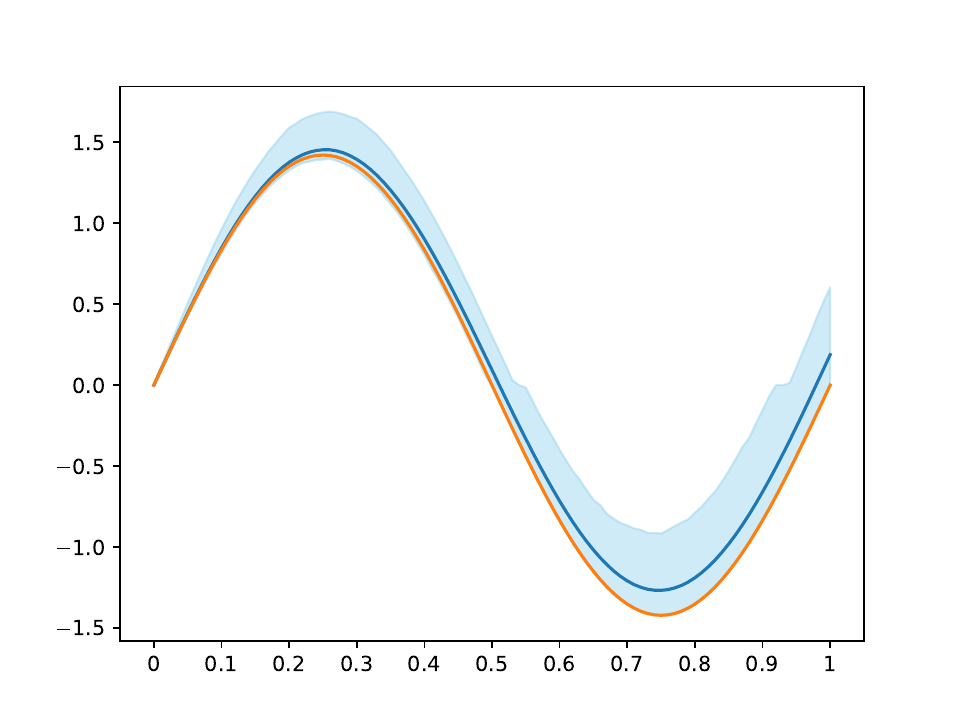}
				\par\medskip
				\centering\textbf{(a)} $\kappa=1$, $\gamma = 1/3$.
			\end{minipage} 
			\hfill
			\begin{minipage}{0.32\textwidth}
				\centering
				\includegraphics[width=\linewidth]{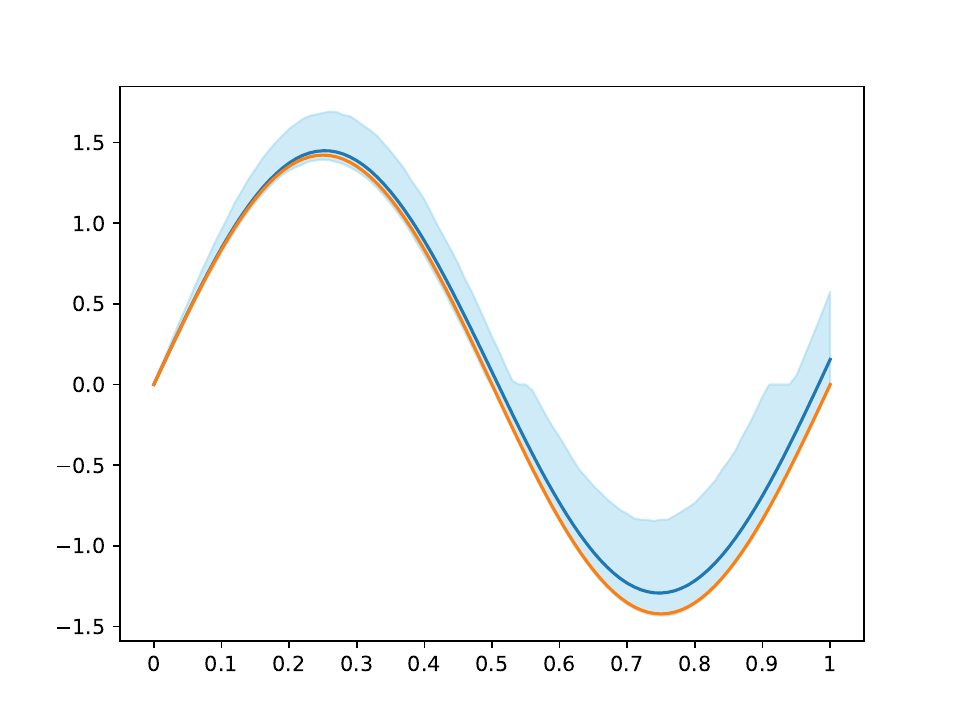}
				\par\medskip
				\centering\textbf{(b)} $\kappa=1$, $\gamma = 1/2$.
			\end{minipage}
			\hfill
			\begin{minipage}{0.32\textwidth}
				\centering
				\includegraphics[width=\linewidth]{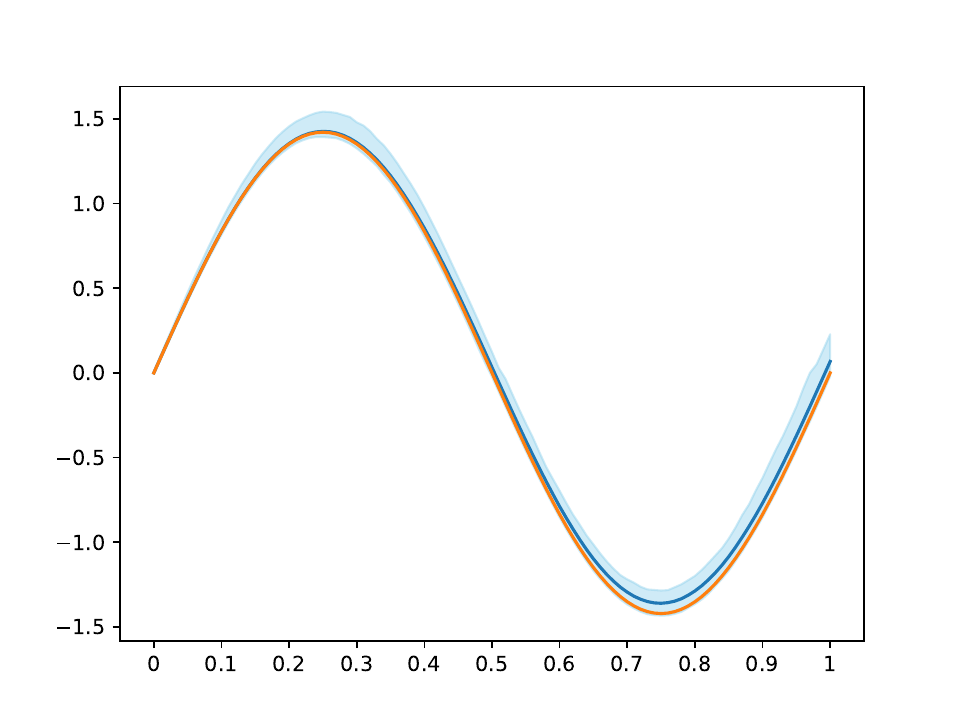}
				\par\medskip
				\centering\textbf{(c)} $\kappa=1$, $\gamma = 9/10$.
			\end{minipage}
			\medskip
			\begin{minipage}{0.32\textwidth}
				\centering
				\includegraphics[width=\linewidth]{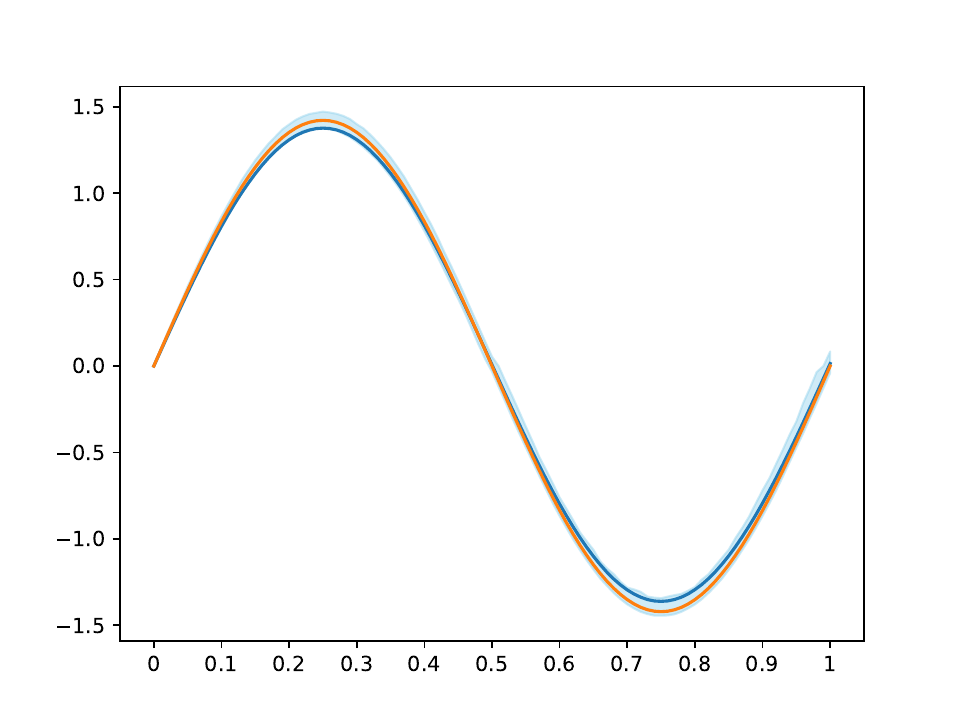}
				\par\medskip
				\centering\textbf{(d)} $\kappa=3/2$, $\gamma = 1/3$.
			\end{minipage}
			\hfill
			\begin{minipage}{0.32\textwidth}
				\centering
				\includegraphics[width=\linewidth]{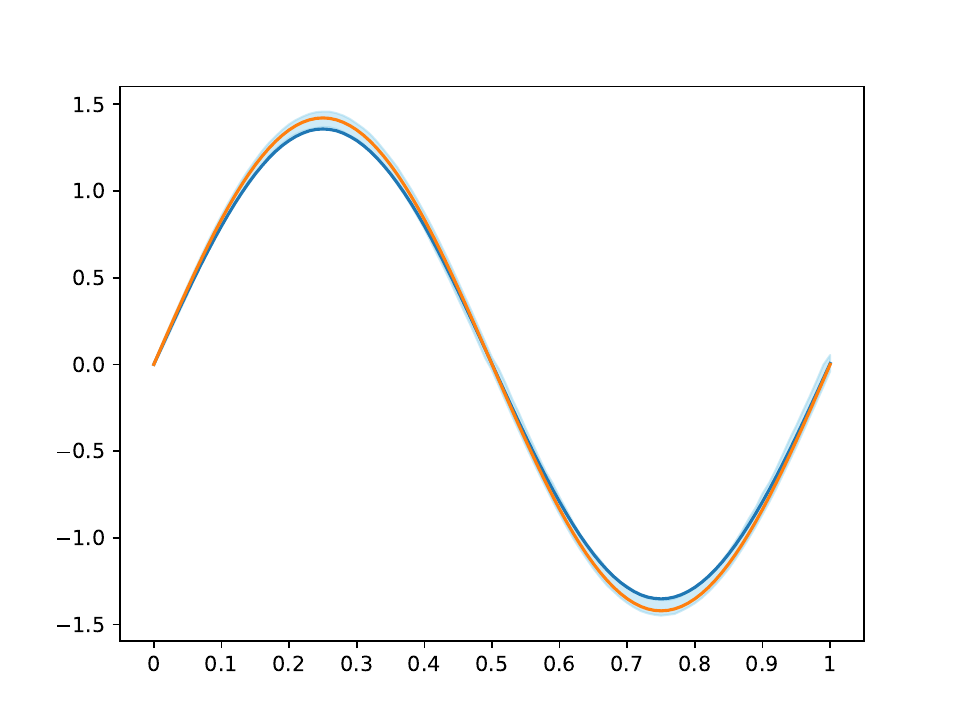}
				\par\medskip
				\centering\textbf{(e)} $\kappa=3/2$, $\gamma = 1/2$.
			\end{minipage}
			\hfill
			\begin{minipage}{0.32\textwidth}
				\centering
				\includegraphics[width=\linewidth]{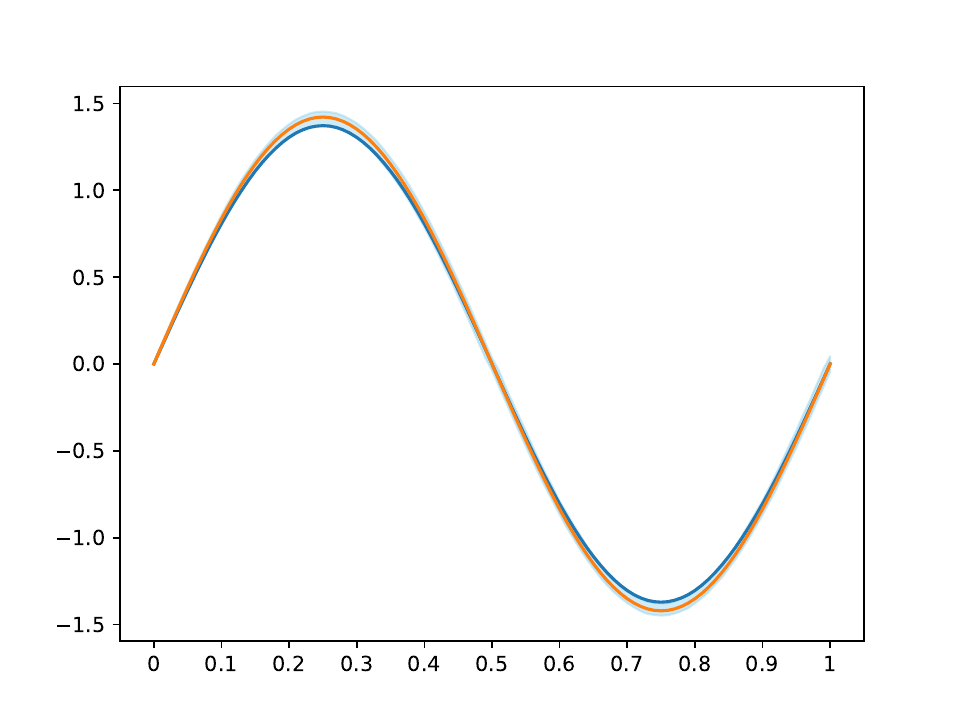}
				\par\medskip
				\centering\textbf{(f)} $\kappa=3/2$, $\gamma = 9/10$.
			\end{minipage}
			\medskip
			\begin{minipage}{0.32\textwidth}
				\centering
				\includegraphics[width=\linewidth]{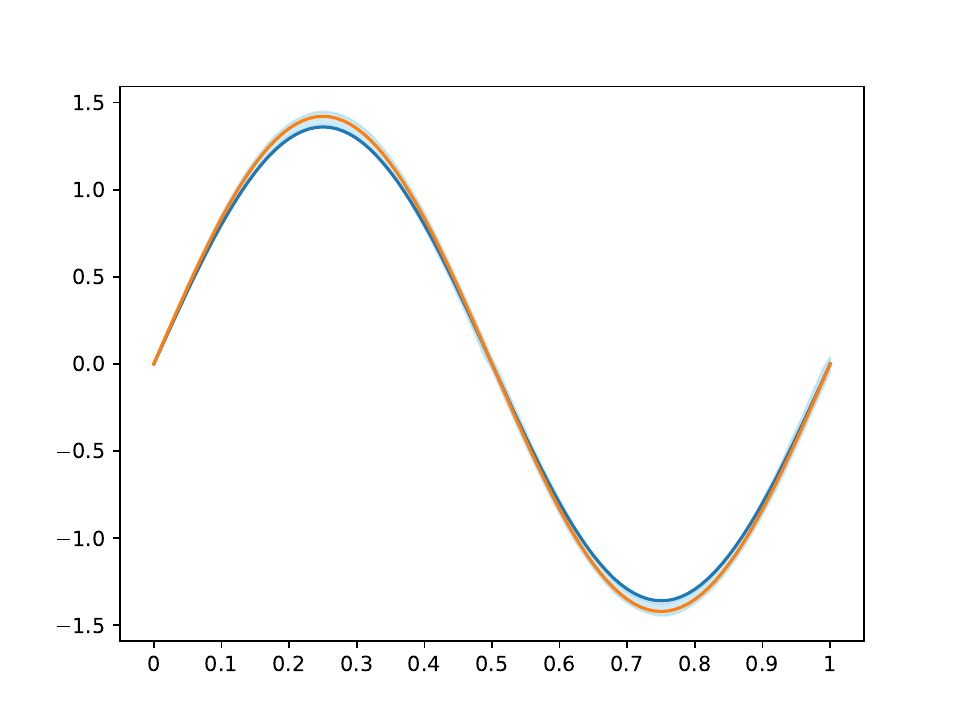}
				\par\medskip
				\centering\textbf{(g)} $\kappa=2$, $\gamma = 1/3$.
			\end{minipage}
			\hfill
			\begin{minipage}{0.32\textwidth}
				\centering
				\includegraphics[width=\linewidth]{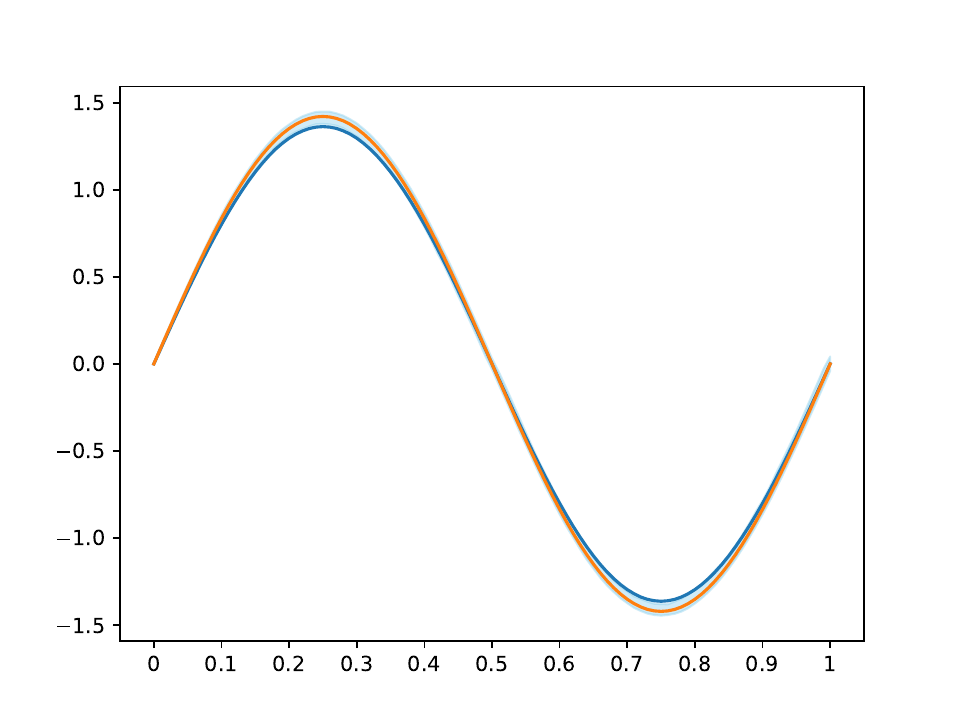}
				\par\medskip
				\centering\textbf{(h)} $\kappa=2$, $\gamma = 1/2$.
			\end{minipage}
			\hfill
			\begin{minipage}{0.32\textwidth}
				\centering
				\includegraphics[width=\linewidth]{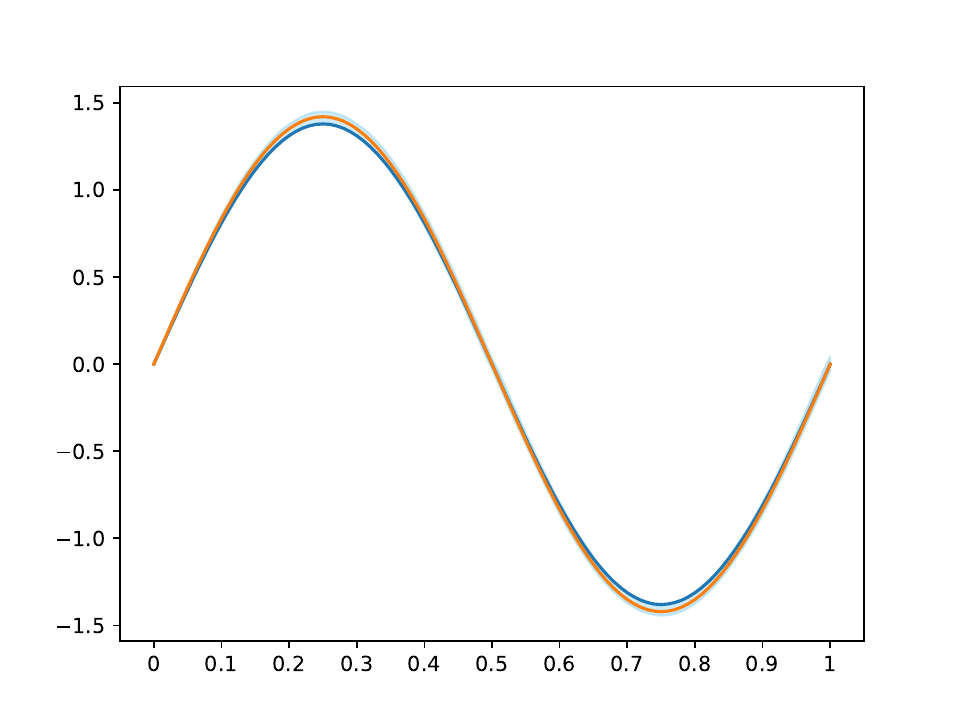}
				\par\medskip
				\centering\textbf{(i)} $\kappa=2$, $\gamma = 9/10$.
			\end{minipage}

			\caption{Simulation results on the inverse model with Burr response and conditional fBm noise. The orange curve is the graph of $t\in[0,1]\mapsto\beta(t)$ while the blue one is the averaged value of $t\in[0,1]\mapsto\tilde \beta_{\vfi}(Y_{n-\tilde k+1,n})(t)$ over $500$ Monte Carlo replications, where $\tilde k=\argmax_{5\le k\le n/5} r(k)$. The light blue area corresponds to the pointwise $5$--$95\%$ inter-quantile envelope of the Monte Carlo replications. Here, $\rho=-1/2$, $\tau=-2$ and $N=101$.}
			\label{fig:beta_estim_plot_conc}
		\end{figure}

		\begin{figure}[!htbp]
			\centering
			\begin{minipage}{0.32\textwidth}
				\centering
				\includegraphics[width=\linewidth]{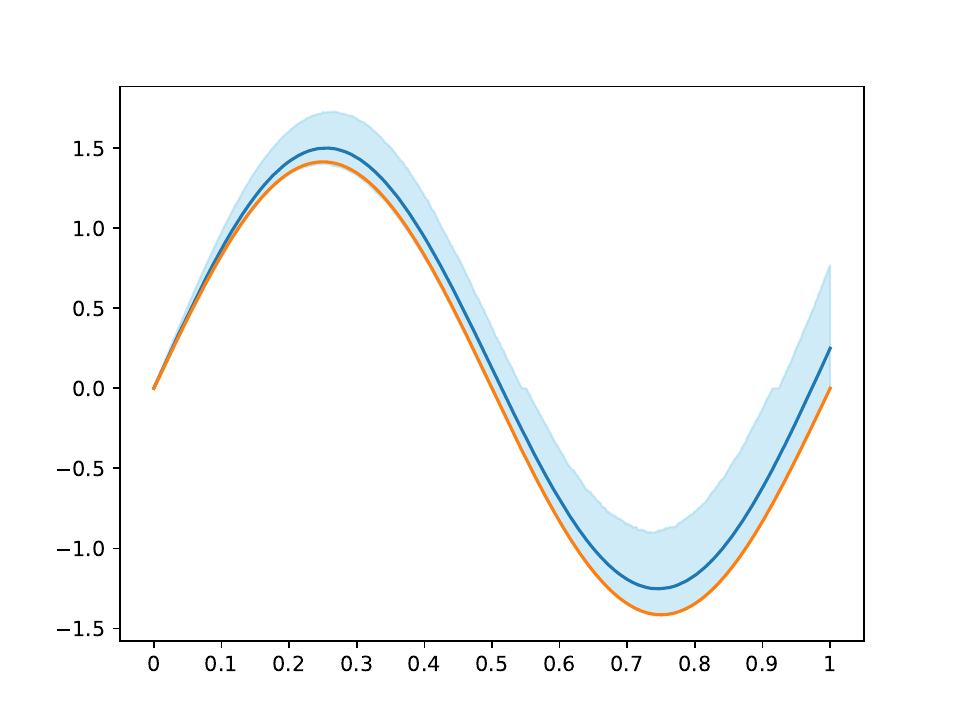}
				\par\medskip
				\centering\textbf{(a)} $\kappa=1$, $\gamma = 1/3$.
			\end{minipage}
			\hfill
			\begin{minipage}{0.32\textwidth}
				\centering
				\includegraphics[width=\linewidth]{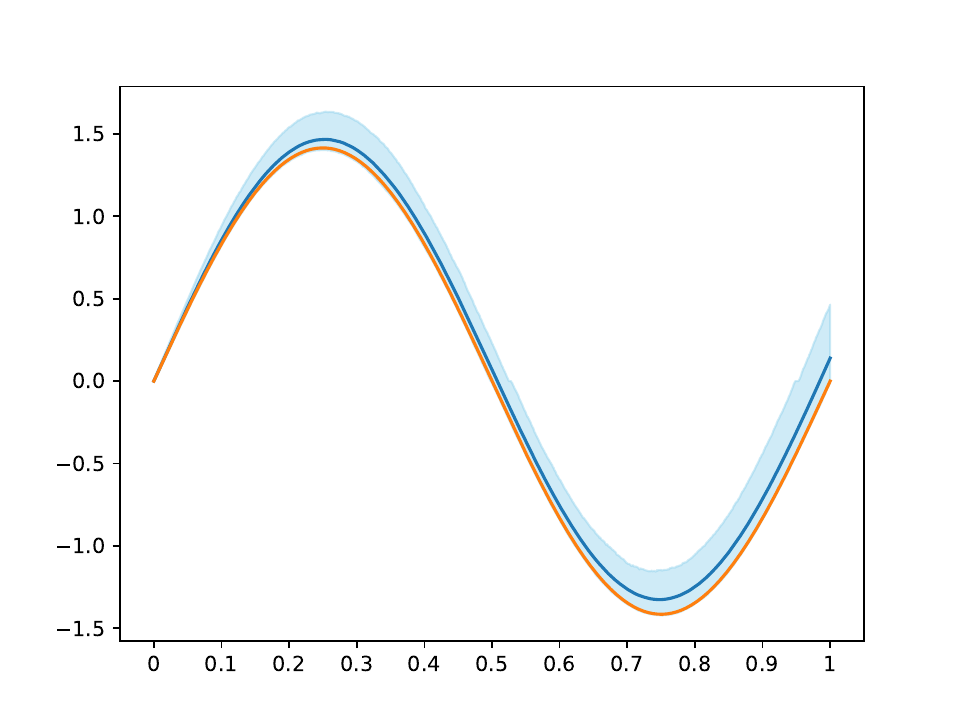}
				\par\medskip
				\centering\textbf{(b)} $\kappa=1$, $\gamma = 1/2$.
			\end{minipage}
			\hfill
			\begin{minipage}{0.32\textwidth}
				\centering
				\includegraphics[width=\linewidth]{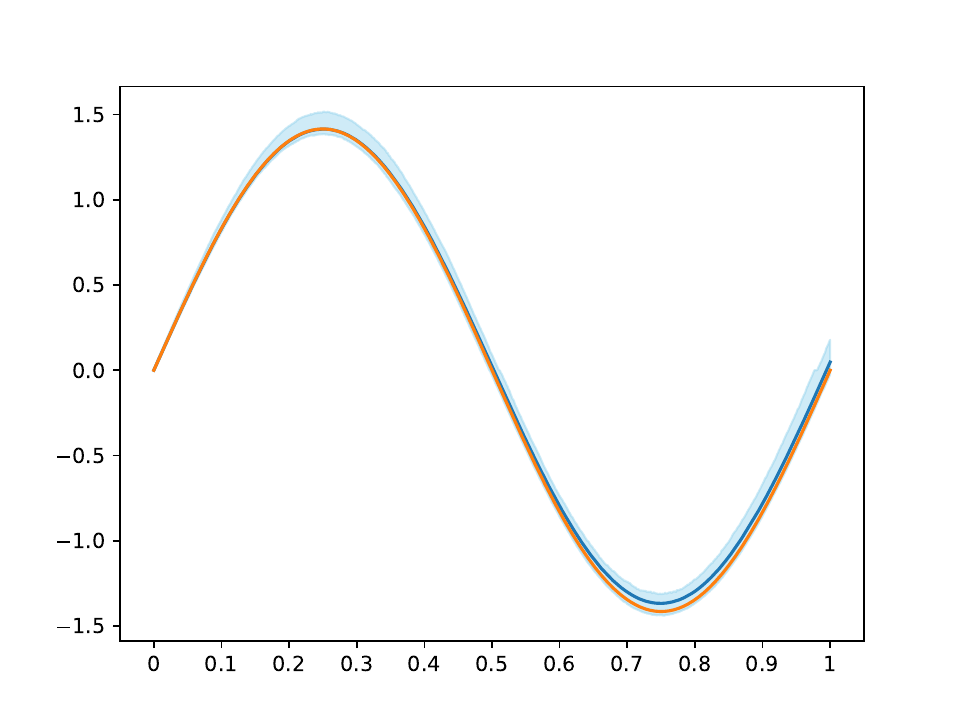}
				\par\medskip
				\centering\textbf{(c)} $\kappa=1$, $\gamma = 9/10$.
			\end{minipage}

			\medskip
			\begin{minipage}{0.32\textwidth}
				\centering
				\includegraphics[width=\linewidth]{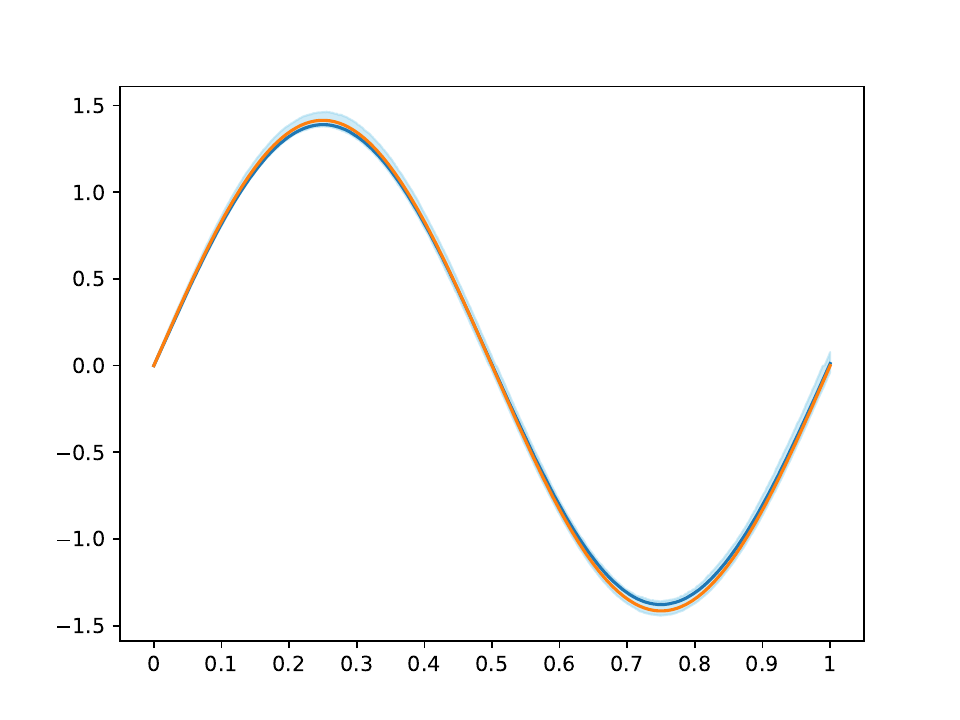}
				\par\medskip
				\centering\textbf{(d)} $\kappa=3/2$, $\gamma = 1/3$.
			\end{minipage}
			\hfill
			\begin{minipage}{0.32\textwidth}
				\centering
				\includegraphics[width=\linewidth]{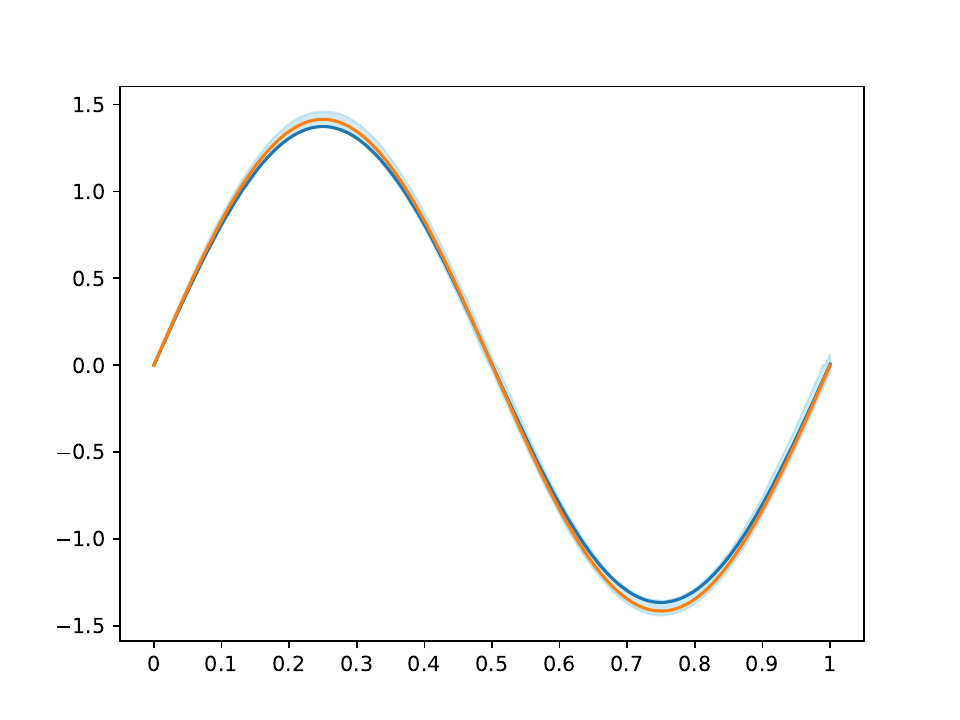}
				\par\medskip
				\centering\textbf{(e)} $\kappa=3/2$, $\gamma = 1/2$.
			\end{minipage}
			\hfill
			\begin{minipage}{0.32\textwidth}
				\centering
				\includegraphics[width=\linewidth]{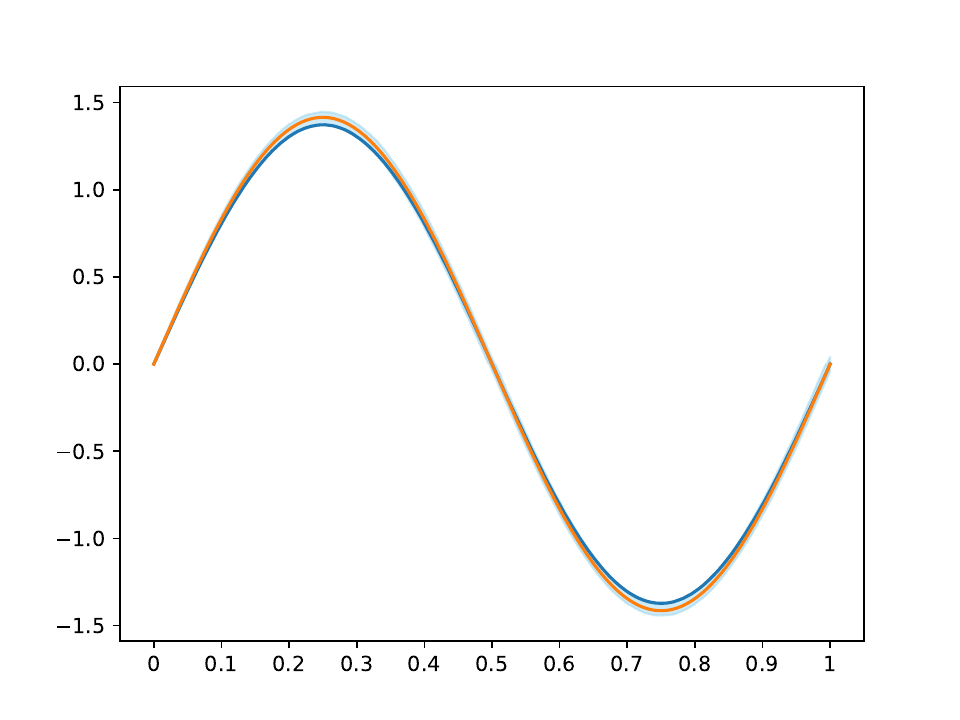}
				\par\medskip
				\centering\textbf{(f)} $\kappa=3/2$, $\gamma = 9/10$.
			\end{minipage}

			\medskip
			\begin{minipage}{0.32\textwidth}
				\centering
				\includegraphics[width=\linewidth]{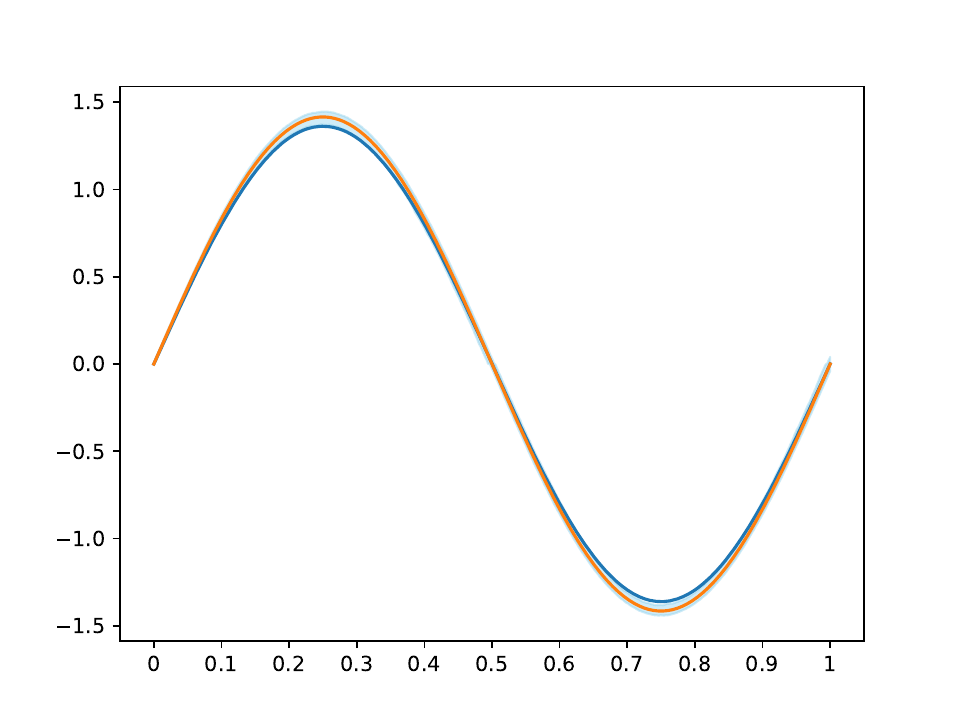}
				\par\medskip
				\centering\textbf{(g)} $\kappa=2$, $\gamma = 1/3$.
			\end{minipage}
			\hfill
			\begin{minipage}{0.32\textwidth}
				\centering
				\includegraphics[width=\linewidth]{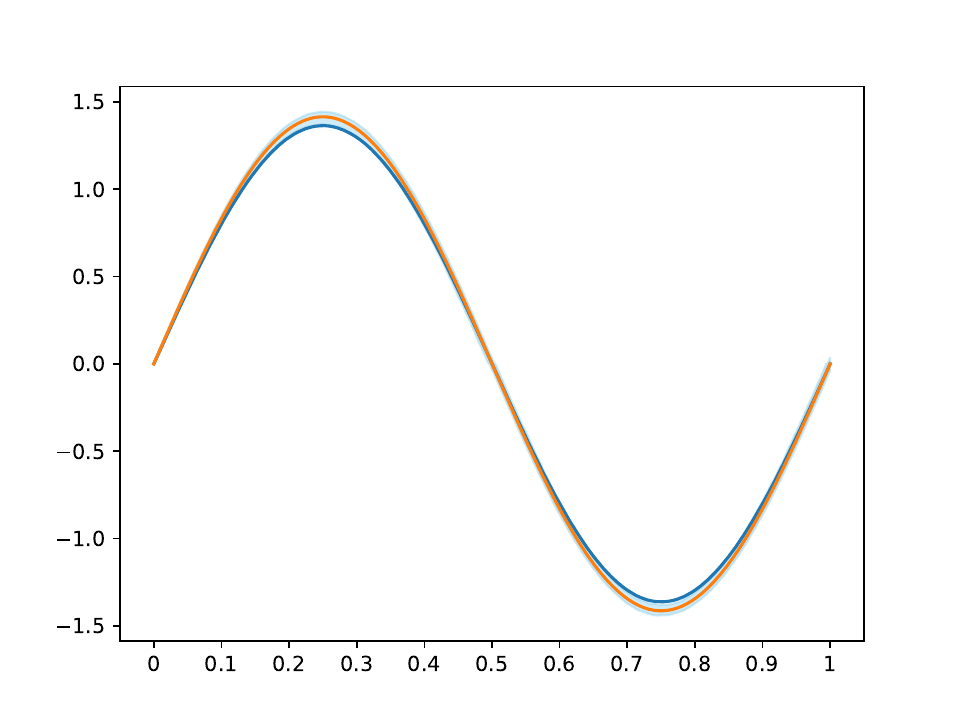}
				\par\medskip
				\centering\textbf{(h)} $\kappa=2$, $\gamma = 1/2$.
			\end{minipage}
			\hfill
			\begin{minipage}{0.32\textwidth}
				\centering
				\includegraphics[width=\linewidth]{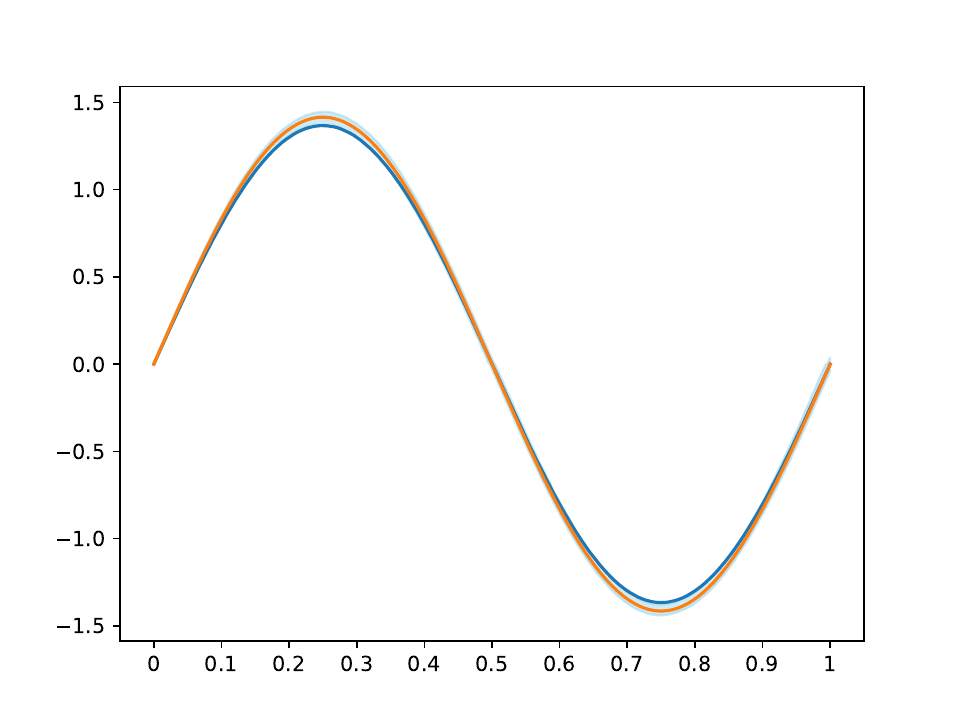}
				\par\medskip
				\centering\textbf{(i)} $\kappa=2$, $\gamma = 9/10$.
			\end{minipage}

			\caption{Simulation results akin to Figure~\ref{fig:beta_estim_plot_conc} for $N=1001$.}
			\label{fig:beta_estim_plot_conc1001}
		\end{figure}

		\section{Application on real data}
		\label{sec-real}

		\rae{The FEPLS method is illustrated} on a financial application where covariates naturally take a functional form due to high-frequency sampling. Introduce the list of assets \rae{(price series available at https://github.com/\allowbreak FutureSharks/\allowbreak financial-data/\allowbreak tree/\allowbreak master/\allowbreak pyfinancialdata/\allowbreak data/)} $$\mathcal{A}:= \{  \text{S\&P500}, \text{DAX30}, \text{Nikkei225}, \text{ETH}, \text{BTC} \},$$ where ETH and BTC respectively stand for Ethe\-reum and Bitcoin, everything understood in their canonical currency.  \rae{Consider} the daily maximum absolute log-return of some asset $\sg{a_2} \in \mathcal{A}$ as the response variable $Y$, and the intraday log-return curve of some asset $\sg{a_1} \in \mathcal{A}$, possibly the same as $\sg{a_2}$, as the functional covariate~$X$. \rae{Every asset of $\mathcal{A}$ is} recorded at one-minute frequency over the period 2013--2017.

		The principal motivation for applying FEPLS in this context is to facilitate the inference of conditional extreme risk measures. By replacing the functional covariate with a univariate projection $Z = \langle \tilde X, \tilde \beta\rangle_\Psi$, the estimation of quantities such as the conditional Value-at-Risk (VaR) or the conditional tail index becomes computationally simpler and statistically more stable in the extreme regime. The core question is then whether the one-dimensional projection retains the covariate information that is relevant for predicting extreme responses. To answer it, \rae{the performance} of FEPLS \rae{is compared} against natural competitors on the ability to capture tail dependence, based on the particular pair $(\sg{a_1},\sg{a_2})=(\text{DAX30}, \text{ETH})$ and then on all the possible asset pairs from $\mathcal{A}$. Ultimately and as a practical tool for extreme risk assessment, \rae{the estimation} of extreme quantiles \rae{is included.} 

		Concerning the tail tuning parameter of $\vfi(y)=y^\tau$, \rae{the value} $\tau = -1$ \rae{is arbitrarily set} throughout this section. 

		\subsection{Sample construction}
		\label{subsec:sample-construction}

		The raw price series of both $\sg{a_1}, \sg{a_2}$ are transformed into log-returns $
		R_t = \log(P_t/P_{t-1})$ where $P_t$ denotes the price at time $t$. The two resulting series are aligned by date so that only contemporaneous observations are kept. This yields approximately the same number of one-minute observations for both indices over the common period Jan 1, 2013 -- Dec 31, 2017. The discretized functional covariate sample $\{\mathcal X_i\}_{1\le i\le n}$ is constructed by cutting the $\sg{a_1}$ log-return series into consecutive blocks of $N=1440$ minutes (24 hours). Each block corresponds to one trading day, and the vector $\mathcal X_i = \big(X_i(t_1),\dots,X_i(t_N)\big)^\top$ contains the intraday log-returns of $\sg{a_1}$ on day $2i-1$. The response sample $\{Y_i\}_{1\le i\le n}$ is obtained from the $\sg{a_2}$ log-returns: for each even day $2i$, \rae{the maximum} of the absolute log-returns over that day \rae{is taken},
		\[
		Y_i = \max_{t\in\text{day }2i} |R_t^{\sg{a_2}}|.
		\]
		No missing data were detected and the final sample size $n$ is among $\{ 265, 365, 433, 442\}$. It may even be equal to $901$ when $\sg{a_1}=\sg{a_2}=\text{BTC}$. As for \rae{the} particular case study, the pair DAX30-ETH consists of $n=265$ observations.

		\begin{Remark}[Temporal dependence]
			Although the theoretical guarantees of Section~\ref{sec-theo} assume i.i.d. observations, financial returns typically exhibit serial dependence and volatility clustering. A common remedy is to thin the data by keeping only observations spaced by several days, see \citet[Section 4]{all_stars2015}. Here, it means that for some fixed day-gap $m\ge 3$, the value $Y_i$ corresponds to the $m i$-th day and the discretized curve $\tilde X_i$ to the day $m i-1$.  
		\end{Remark}

		\subsection{Response tail behavior} 
		Before applying FEPLS, \rae{one needs} to verify that the response variable $Y$ exhibits a heavy tail. The evidence for the pair DAX30-ETH is given by the Hill plot in Figure~\ref{fig:hill}, a standard diagnostic tool from extreme-value theory which represents the graph of the Hill estimator
		\begin{align*}
			\hat\gamma(k)&=\ff{ k}\sum_{i=1}^{ k} (\log Y_{n-i+1,n} -\log Y_{n- k,n})
		\end{align*}
		as a function of $k\in\{15,\dots,65\}$, see \citet{Hill}. \rae{Accordingly,} Figure~\ref{fig:hill} exhibits a stable behavior over a range of intermediate values of $k$, with an estimated tail index around $\hat \gamma \approx 1/2$.

		\begin{figure}[!htbp]
			\centering
			\includegraphics[width=0.8725\linewidth]{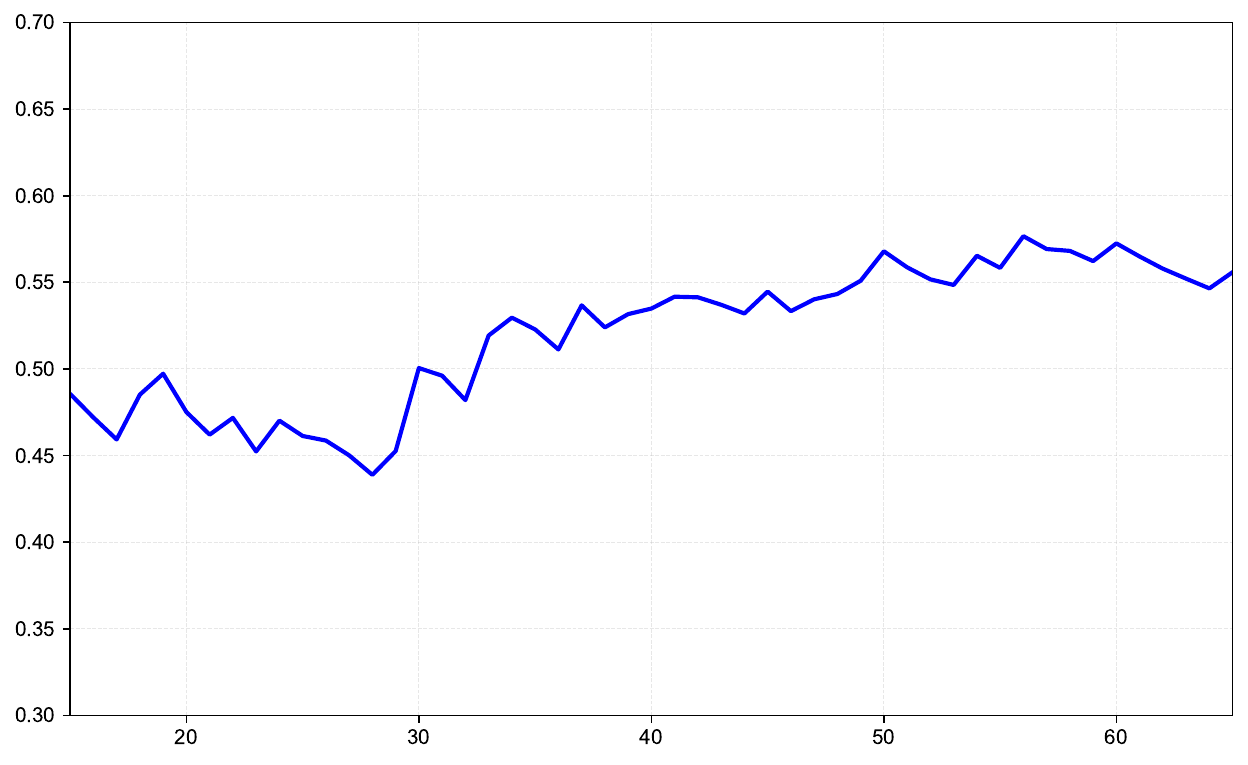}
			\caption{Hill plot for the response variable in the pair DAX30-ETH showing $k\mapsto \hat{\gamma}(k)$ where $k$ is the number of extreme observations.}
			\label{fig:hill}
		\end{figure}

		\subsection{Threshold selection}
		\rae{To apply} the selection rule of Section~\ref{sub-seuil}, \rae{the logarithm} of the empirical absolute covariance between $Y$ and $\langle \tilde \beta_{\vfi}, \tilde X\rangle$ \rae{is plotted} in Figure~\ref{fig:threshold_selection}, namely $k\mapsto \log|r(k)|$ for $k\in\{15,\dots,65\}$. Selecting the argument of $|r(k)|$ that gives its maximum value hence leads to $\tilde k=6$, which means that the extreme regime approximately consists of the top $2.26 \%$ of the $265$ response values. This is used as the effective sample size for all subsequent extreme‑tail analyses about the pair DAX30-ETH. \rae{Different values of} $\tau$ would lead to \rae{markedly} distinct outcomes for $\tilde k$, \emph{e.g.}, for $\tau =1$, one has $\tilde k = 17$. 

		\begin{figure}[!htbp]
			\centering
			\includegraphics[width=0.9\linewidth]{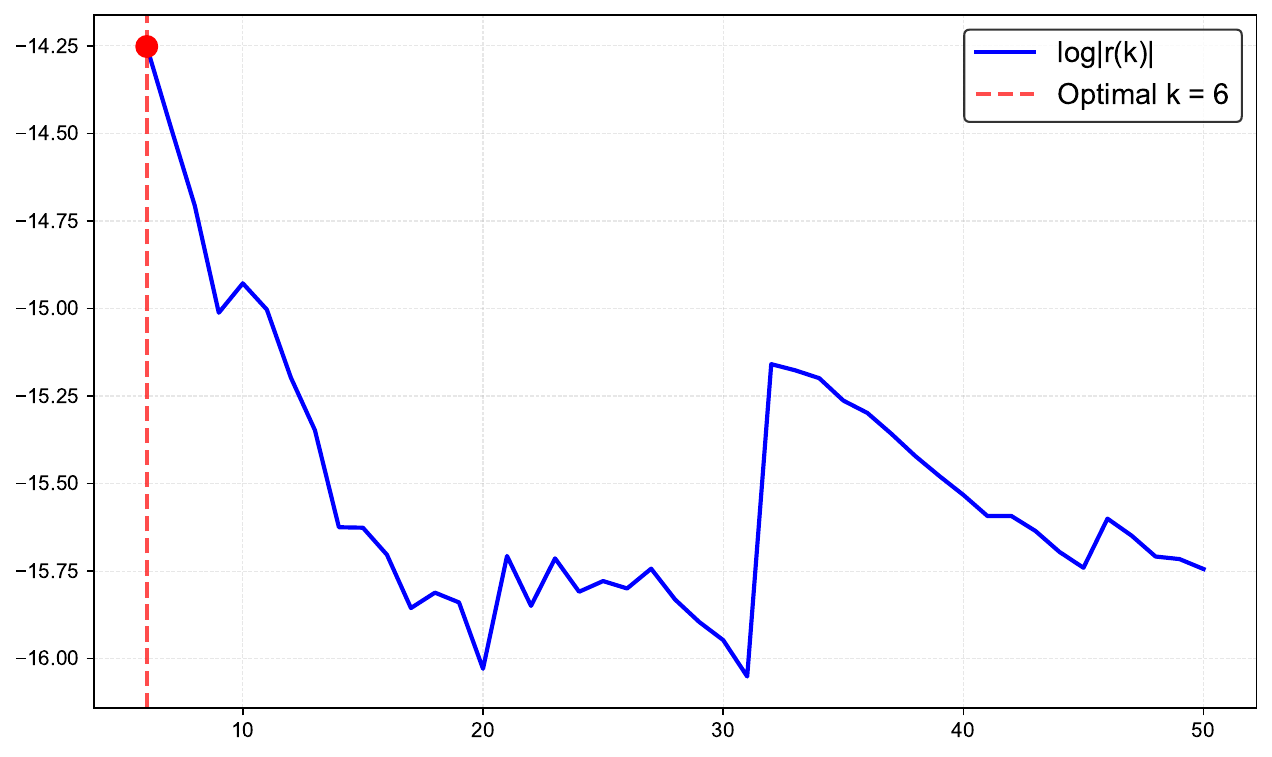}
			\caption{Threshold selection for the DAX30-ETH pair. The plot displays $k\mapsto \log|r(k)|$. The optimal threshold is obtained at $\tilde{k} = 6$ (red dot).}
			\label{fig:threshold_selection}
		\end{figure}

		\subsection{Tail dependence competition}
		\label{sub-extr-compet}
		To assess the effectiveness of \rae{FEPLS, its performance is evaluated} in a tail covariance competition against several projection strategies on the present high-dimensional discretized functional data, \emph{i.e.}, for the pair DAX30-ETH and then on all possible pairs. A variety of baseline competitor strategies is considered: random projections, along with standard dimension reduction methods and their extreme variants (with e-prefix) computed only on the tail region of the data sample. More precisely, fix a threshold $y>0$ and define the exceedances index set as
		\[
		\mathcal{I}_y \; :=\; \{\, i\in\{1,\dots,n\} : Y_i>y \,\},\qquad n_y:=|\mathcal{I}_{y}|.
		\]
		Let $\tilde{X}_y\in\mathbb{R}^{n_y\times N}$ denote the submatrix with rows $\{\tilde X_i^\top : i\in\mathcal{I}_y\}$ and $Y_y\in\mathbb{R}^{n_y}$ the corresponding tail responses. Let $\mathcal{M}:=\{ \text{FEPLS}, \text{eFPCA}, \text{eFSIR}, \text{FPCA}, \text{FSIR}\}$, then each method $m\in \mathcal{M}$ produces a projection score
		\[
		Z_{m,y} \; :=\; \tilde{X}_y\,\Psi\,\beta_m(y) \in \mathbb{R}^{n_y},
		\]
		where the method-specific direction $\beta_m(y)=(\beta_m(t_1),\dots,\beta_m(t_N))^\top\in\mathbb{R}^N$ (or $\beta_m$ if the method does not depend on $y$) is defined below,  and $\Psi\in\mathbb{R}^{N\times N}$ \cp{is as in Section~\ref{sec-sim}.}

		For each $m\in \mathcal{M}$, \rae{the unbiased empirical covariance} between the method's score and the tail responses \rae{is computed}, as the performance metric for \rae{the} analysis,
		\begin{align*}
			\rho_m(y):= &\frac{1}{n_y-1}\sum_{i\in\mathcal{I}_y}\left((Z_{m,y})_i-\bar Z_{m,y}\right)\left(Y_i-\bar Y_y\right),
		\end{align*}
		with $\bar Z_{m,y}=n_y^{-1}\sum_{i\in\mathcal{I}_y} (Z_{m,y})_i$ and $\bar Y_{y}=n_y^{-1}\sum_{i\in\mathcal{I}_y} Y_i$. The standard full-sample versions are recovered as special cases when $\mathcal{I}_y = \{1,\dots,n\}$. All competitors are implemented in their discretized form with the $\Psi$-inner product rather than the standard $L^2$ inner product. 

		\medskip
		(i) \emph{Extreme Functional PCA (eFPCA).}
		Let $\bar{\tilde{X}}_y := n_y^{-1}\sum_{i\in\mathcal{I}_y} \tilde X_i$ be the empirical mean over the tail. Define the empirical covariance matrix of the tail covariates:
		\[
		\widehat\Sigma_y := \frac{1}{n_y}\sum_{i\in\mathcal{I}_y} (\tilde X_i-\bar{\tilde{X}}_y)(\tilde X_i-\bar{\tilde{X}}_y)^\top \in \mathbb{R}^{N\times N}.
		\]
		The eFPCA direction $\beta_{\mathrm{eFPCA}}(y)$ is the eigenvector corresponding to the largest eigenvalue of $\widehat\Sigma_y$:
		\[
		\beta_{\mathrm{eFPCA}}(y) := \argmax_{\beta \in \mathbb{R}^N, \|\beta\|_{\Psi}=1} \beta^\top \widehat\Sigma_y \beta,
		\]
		where $\|\beta\|_{\Psi}^2 := \beta^\top \Psi \beta$.  

		\medskip
		(ii) \emph{Extreme Functional SIR (eFSIR).}
		Partition $\mathcal{I}_y$ into five equal-frequency slices $\mathcal{S}_1,\dots,\mathcal{S}_5$ according to the ordered values of $\{Y_i: i\in\mathcal{I}_y\}$. For each slice $1\le h \le 5$, define the slice mean $m_h := \frac{1}{|\mathcal{S}_h|}\sum_{i \in \mathcal{S}_h} \tilde X_i$. The between‑slice covariance matrix is
		\[
		\widehat\Sigma_{B,y} := \sum_{h=1}^{5} \frac{|\mathcal{S}_h|}{n_y}\,(m_h - \bar{\tilde{X}}_y)(m_h - \bar{\tilde{X}}_y)^\top.
		\]
		The eFSIR direction $\beta_{\mathrm{eFSIR}}(y)$ is the solution of the generalized eigenvalue problem
		\[
		\widehat\Sigma_{B,y} \beta = \lambda \, \widehat\Sigma_y \beta,
		\]
		taken for the eigenvalue $\lambda$ with largest absolute value. In practice\rae{, the regularized problem} $(\widehat\Sigma_y + \eta I_N)\beta = \lambda^{-1}\widehat\Sigma_{B,y}\beta$ \rae{is solved} with $\eta = 10^{-6}\,\mathrm{tr}(\widehat\Sigma_y)/N$ and $I_N$ being the identity matrix of size $N$. The full-sample FSIR direction $\beta_{\mathrm{FSIR}}$, obtained by replacing $\mathcal{I}_y$ with $\{1,\dots,n\}$, uses $10$ slices.

		\medskip
		(iii) \emph{Random projections.}
		As a naive baseline, $5000$ \rae{independent random directions} $({\beta^{(r)}}/{\|\beta^{(r)}\|_{\Psi}})_{1\le r\le 5000}$ \rae{are generated} by drawing $\beta^{(r)} \sim \mathcal{N}(0, I_N)$.

		\subsection{Competition results}
		\label{sec:significance-testing}

		\rae{Figure~\ref{fig:performance_comparison} displays} the graph of $k\mapsto |\rho_m(Y_{n-k+1,n})|$ for all $m\in \mathcal{M}$ and the fixed pair $(\sg{a_1},\sg{a_2})=(\text{DAX30}, \text{ETH})$. One may see that FEPLS dominates the competitors across almost all thresholds $k$. The gray shaded region represents the $5$--$95\%$ inter-quantile range of covariances obtained from $5000$ random directions. It highlights the fact that the observed performance of FEPLS is far above what would be expected by chance. Lastly, one may also notice that the extreme versions of the competitor strategies consistently outperform their full-sample counterpart. 

		\begin{figure}[!htbp]
			\centering
			\includegraphics[width=0.8\textwidth]{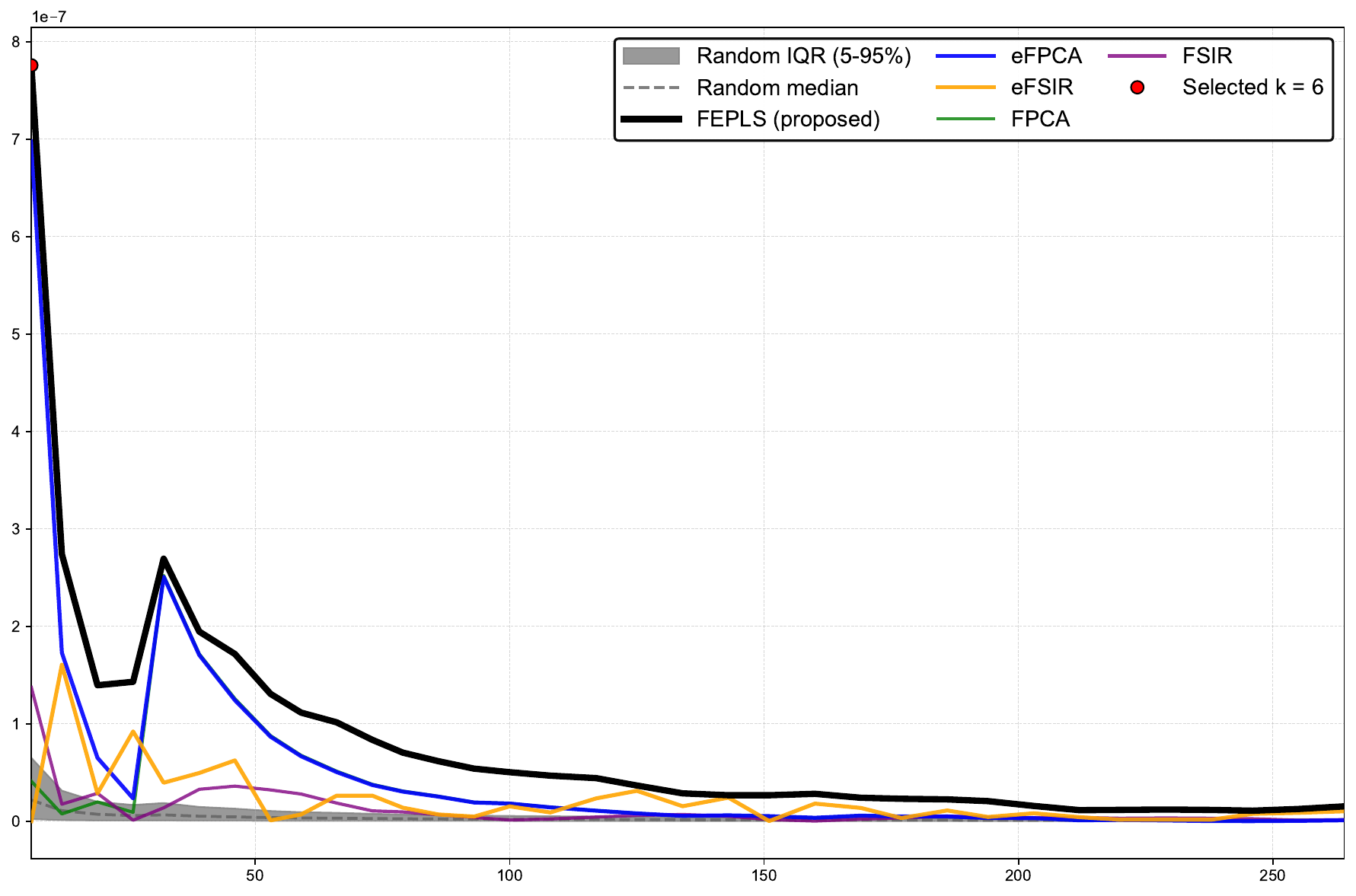}
			\caption{Comparison for the DAX30-ETH pair of FEPLS against four competing methods across threshold values $k$ in the x-axis. Each curve shows $|\rho_m(Y_{n-k+1,n})|$ the absolute unbiased covariance between the method's projection and the response for the top $k$ observations. The gray shaded region represents the 5–95\% inter-quantile range of 5000 random projections.}
			\label{fig:performance_comparison}
		\end{figure}

		Extending Figure~\ref{fig:performance_comparison} across dataset-threshold combinations, Figure~\ref{fig:competition} reports the evolution, as a function of the threshold level $k$, of the mean rank of the different dimension-reduction methods in terms of absolute tail covariance with the response, based upon the $25$ possible pairs and all threshold levels $k$. For each value of $k$, methods are ranked from best (rank~1) to worst, and the displayed curves correspond to ranks averaged over the 25 asset pairs; hence, lower values indicate better overall performance.  FEPLS consistently attains the lowest mean rank across the entire range of thresholds. Moreover, classical full-sample methods systematically exhibit higher mean ranks than their extreme counterparts, especially for smaller values of $k$.

		\begin{figure}[!htbp]
			\centering
			\includegraphics[width=0.9\linewidth]{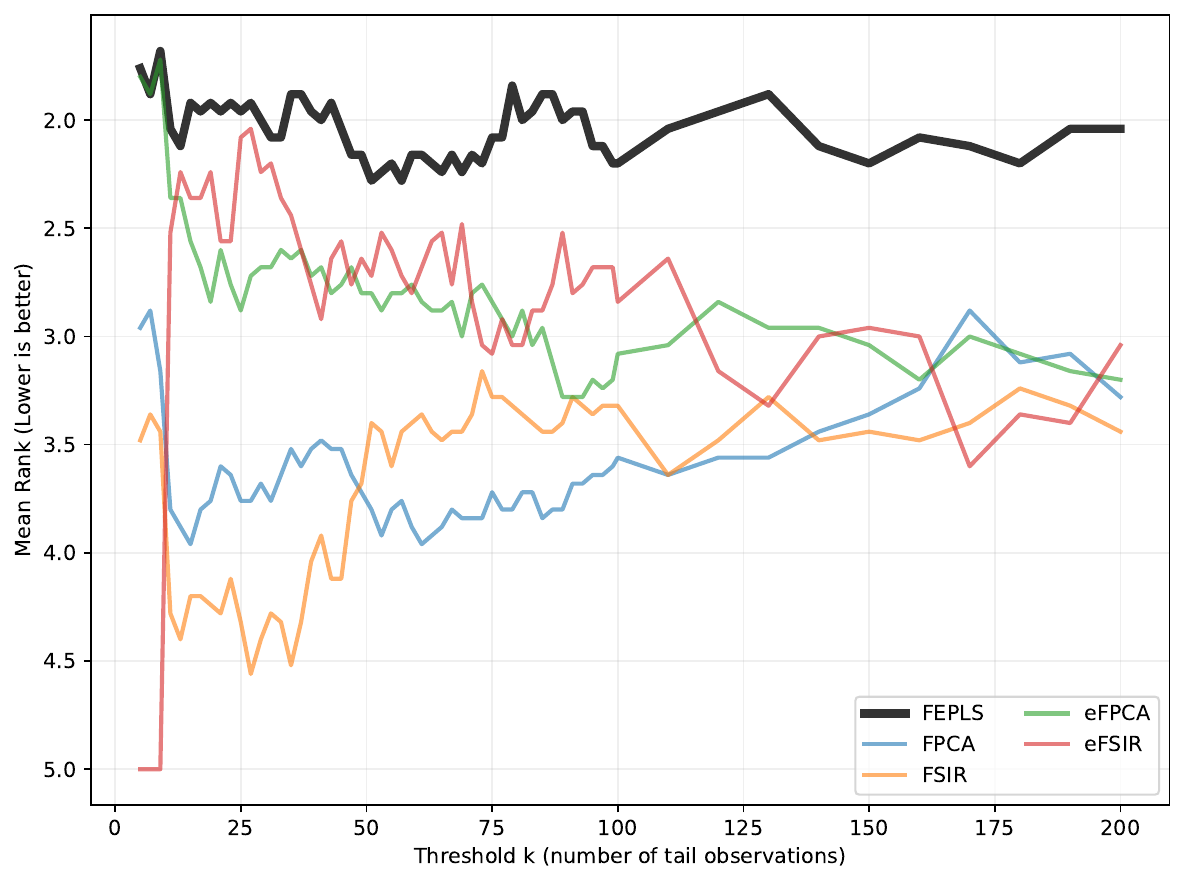}
			\caption{Evolution of mean rank performance across thresholds for the different methods. Results are based on the evaluation over 25 asset pairs.}
			\label{fig:competition}
		\end{figure}

		\rae{Finally, as a complement to the previous plot,} Figure~\ref{fig:win_percentage} displays the win rate of FEPLS, \emph{i.e.}, the proportion of pairs on which it outperforms all competitors, across different threshold levels $k$. To formally support the claim that the superior performance of FEPLS is statistically significant and not attributable to random variations, \rae{a one-sided binomial test is employed} under the null hypothesis of equal performance across all five methods. Under the null hypothesis, each method has equal $20\%$ of winning probability. With $25$ (supposedly independent) datasets, the number of wins $W$ for FEPLS follows $W \sim \text{Binomial}(25, 0.20)$. The critical value for $5\%$-significance corresponds to $9$ wins over those $25$ trials, meaning a win rate of 36\%, since $\mathbb{P}(W \geq 9) = 0.0468$, as visualized by the horizontal green line. The plot shows that FEPLS achieves a win rate consistently above 36\% across all dataset-threshold combinations, providing evidence against the null hypothesis.

		\begin{figure}[!htbp]
			\centering
			\includegraphics[width=0.8\linewidth]{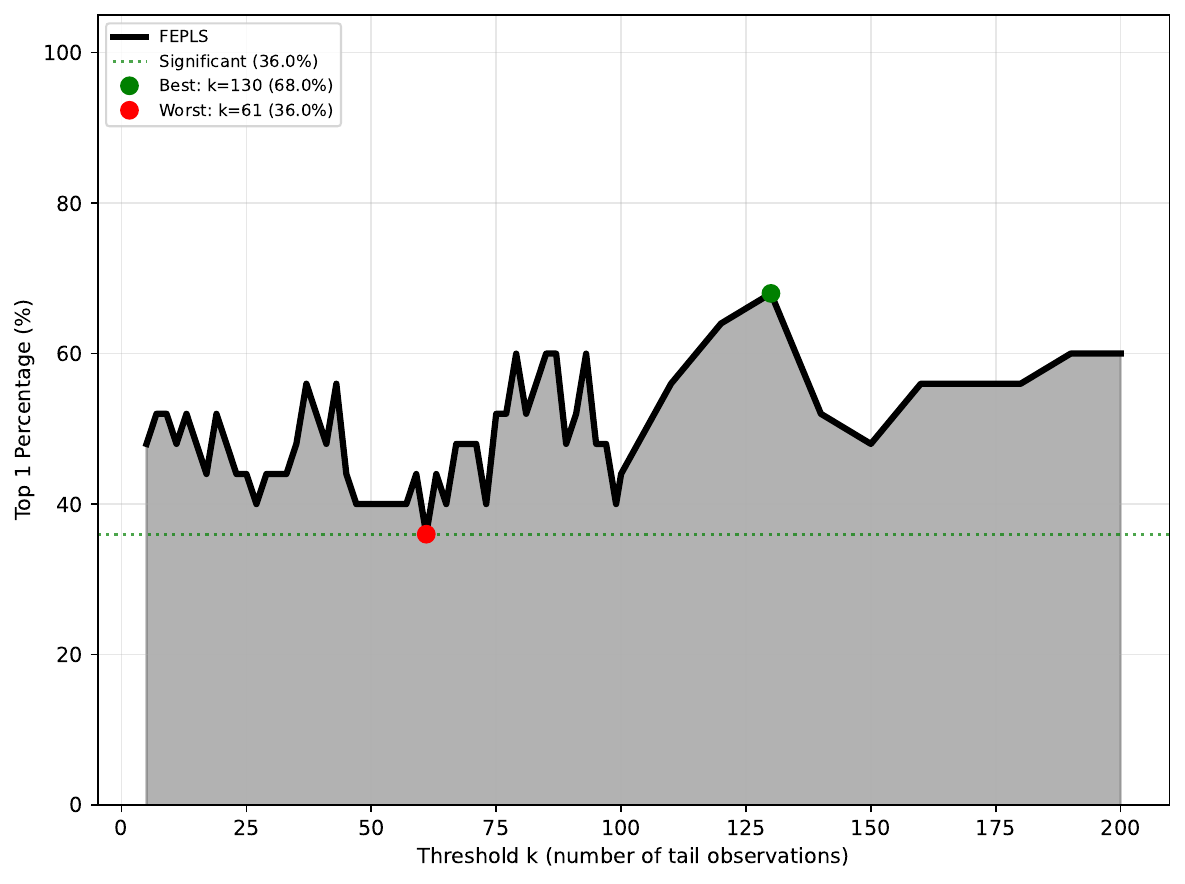}
			\caption{
				Win rate of FEPLS, \emph{i.e.}, proportion of the $25$ pairs on which it ranks first, across threshold values $k$. The green line shows the statistical significance 
				threshold derived from a binomial test at level $5\%$ with 
				$25$ datasets and $5$ competing methods. 
			}
			\label{fig:win_percentage}
		\end{figure}

		\subsection{Conditional extreme quantile inference}
		\label{sub-extr-cond}
		To illustrate how \rae{FEPLS} can be used for risk measurement in the extreme regime, \rae{consider again} the particular case $(\sg{a_1},\sg{a_2})=(\text{DAX30}, \text{ETH})$ and study its related conditional quantile inference. For a generic random variable $V\in H$, \rae{denote} by $F_{Y|V=v}$ the conditional cumulative distribution function of $Y$ given $V=v$, where $v\in H$. 
		The (functional) Nadaraya--Watson estimator of $F_{Y|V=v}$ is
		\begin{align}\label{eq:cdf_estimator}
			\hat{F}_{Y|V=v}( y) &:=\sum_{i=1}^n {K\left( \f{{\|V_i-v\|} }{h_n}\right) 1_{\{ Y_i\le y\}}}\left/{\sum_{s=1}^nK\left( \f{{\|V_s-v\|} }{h_n}\right)}\right.,\quad y\in \R.
		\end{align}
		Here, the kernel $K$ is Gaussian and the parameter $h_n\downarrow 0$ is the window bandwidth. 
		The conditional quantile function of $Y|V=v$ is defined thanks to the generalized inverse $q_{Y|V=v}:=F^{-}_{Y|V=v}$ and the associated estimator is
		$$ 
		\hat{q}_{Y|V=v}(\alpha) :=\inf \{y>0,\hat{F}_{Y|V=v}(y)\ge \al \},\quad \alpha\in (0,1),
		$$
		see \citet[Section~6.4]{FerratyVieu2006} for details.

		Figure~\ref{fig:quantile_curves} shows the scatterplot for the pair DAX30-ETH of the projected data on the FEPLS direction $(\langle \tilde X_i,\tilde  \beta_{\vfi}\rangle,Y_i)$ for $1\le i \le n$ endowed with the estimated conditional Value-at-Risk curves $t\mapsto \hat{q}_{Y|\langle \tilde X,\tilde  \beta_{\vfi}\rangle=t}(\alpha)$ associated with risk levels $\alpha \in \{ 0.95, 0.98,0.995\}$. Herein, the kernel inference relies on inverting the univariate version of~\eqref{eq:cdf_estimator} with bandwidth $h_n$ arbitrarily chosen to be $5\cdot 10^{-5}$. The red dots are associated with
		the $\tilde k$ pairs with highest response values.
		Such a graphic may be used as a visualization tool to detect the most risky situations. Indeed, given a new pair $(x_0,y_0)$, one may determine whether $y_0$ given $\tilde X=x_0$ exceeds the 
		conditional Value-at-Risk $\hat{q}_{Y|\langle \tilde X,\tilde \beta_{\vfi}\rangle=\langle x_0,\tilde \beta_{\vfi}\rangle}$ simply by comparing its projected position $(\langle x_0,\tilde \beta_{\vfi}\rangle,y_0)$ to the quantile curve in Figure~\ref{fig:quantile_curves}. Additional experiments (not reported here) show that the FEPLS outputs are \rae{only weakly} sensitive to $\tau$ in this real data application as well as to $\alpha\le 0.98$ for the extreme quantile estimation. 

		\begin{figure}[!htbp]
			\centering   \includegraphics[width=0.8\textwidth]{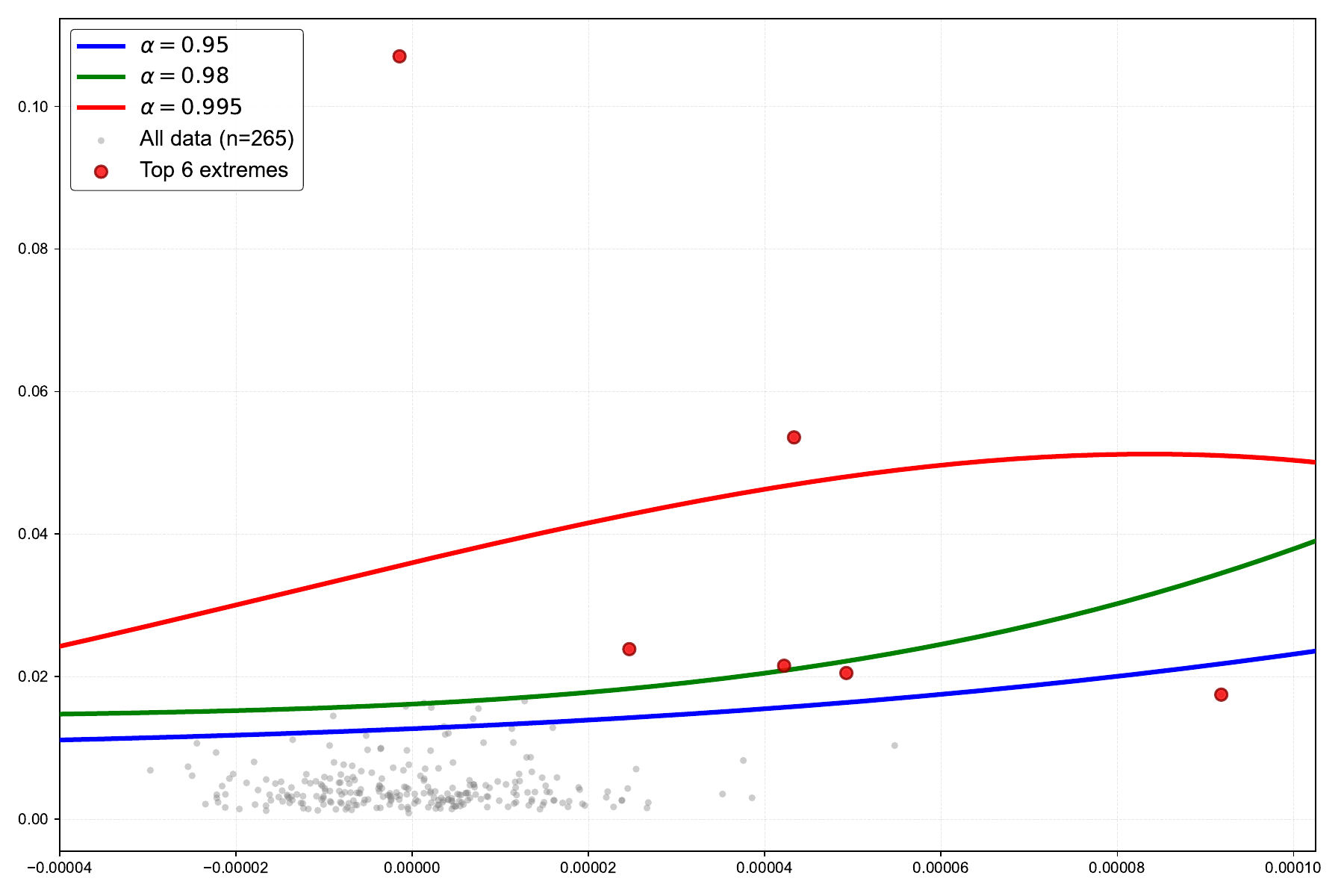}
			\caption{Conditional extreme quantile curves estimated via FEPLS for the DAX30-ETH pair. The plot shows the FEPLS projection score $Z = \langle \tilde{X}, \tilde{\beta}_{\varphi}\rangle_{\Psi}$ versus the response $Y$, with estimated conditional Value-at-Risk curves at different levels. Red dots mark the $\tilde{k}=6$ most extreme observations; gray points show all projected data. 
			}
			\label{fig:quantile_curves}
		\end{figure}

		\clearpage
		\appendix

		\section{Technical results}\label{sec-proofs}

		This section gathers all the mathematical details behind Section~\ref{sec-theo} and Section~\ref{sec-inf}. \rae{The section starts} with \ref{subsection:prelim} providing useful technical tools which make it possible to manipulate inner products under conditional expectations, or the distribution of order statistics allowing an extension of \citet[Lemma~2]{Bousebata2023} to the random threshold framework. This yields theoretical tail-moments asymptotics which will be extensively used \rae{in the proofs}. The section also contains the proof of Proposition~\ref{prop:solution}. Next, the asymptotic behavior of the relevant empirical tail-moments, involving the noise and its inner product with $\beta$, is investigated in \ref{subsection:empirical_tail_moment}. Finally, \rae{the} main Theorem~\ref{th:main} is established in \ref{subsection:consistency_random} under the inverse single-index model, \sg{as well as Proposition~\ref{prop:modelfree} holding for general empirical tail-moments in a model-free setup.} 

		\subsection{Preliminaries}
		\label{subsection:prelim}
		A standard result \rae{for} Bochner spaces, \emph{e.g.}, \citet[Eq~(1.2)]{bochner}, states that Bochner integrals commute with bounded linear operators such as orthogonal projections. \rae{This} is the definition of Pettis integrals. In the Hilbert framework, projections are generally expressed in terms of the inner product. The following technical tool considers this commutation operation with the inner product when the two random variables are conditionally stochastically independent. The case of one variable being deterministic will also often occur in the proofs. \sg{\rae{The} conditional expectation $\E_\mathcal{F}$ and its Bochner version $\E_{\mathcal{F},\mathcal{B}}$ \rae{are those} introduced in Section~\ref{sec-inf}.}
		\cp{%
		\begin{lem}\label{lem:bochner_random}
			Let $\mathcal{F}$ be a sub-$\sigma$-field of $\mathcal{A}$ and let $W_1,W_2$ be Bochner-integrable random variables having values in $(H,\langle\cdot,\cdot\rangle)$ a separable Hilbert space. Assume that $W_1$ and $W_2$ are independent and, in addition, independent conditionally on $\mathcal{F}$. Then $\langle W_1,W_2\rangle$ is integrable and
			$$
			\E_\mathcal{F} (\langle W_1,W_2\rangle) = \langle \E_{\mathcal{F},\mathcal{B}}(W_1),\E_{\mathcal{F},\mathcal{B}}(W_2)\rangle \quad \text{a.s.}
			$$
		\end{lem}
		\begin{proof}
			\rae{Fix} a countable orthonormal basis $(e_j)_{j\ge 1}$ of $H$. The Cauchy--Schwarz inequality in $H$, the independence of $W_1$ and $W_2$ and Tonelli's theorem entail
			\begin{align}\label{eq:bochner-integrability}
				\E(|\langle W_1,W_2\rangle|)\le\E(\|W_1\|\,\|W_2\|)=\E(\|W_1\|)\,\E(\|W_2\|)<+\infty ,
			\end{align}
			so that all the conditional expectations at hand below are well defined. Denote $\mathcal{G}:=\mathcal{F}\vee\sigma(W_2)$.

			The first step is to establish the following pull-out property: for any sub-$\sigma$-field $\mathcal{H}$ of $\mathcal{A}$, any Bochner-integrable random variable $U$ having values in $H$ and any $\mathcal{H}$-measurable random variable $V$ having values in $H$ such that $\E(\|U\|\,\|V\|)<+\infty$, one has
			\begin{align}\label{eq:pullout}
				\E_{\mathcal{H}}(\langle U,V\rangle)=\langle \E_{\mathcal{H},\mathcal{B}}(U),V\rangle \quad\text{a.s.}
			\end{align}
			First, when $V=\sum_{m=1}^{M} 1_{A_m} z_m$ is a simple function, with $A_1,\dots,A_M\in\mathcal{H}$ pairwise disjoint and $z_1,\dots,z_M\in H$, \eqref{eq:pullout} follows from the linearity of the conditional expectation, the $\mathcal{H}$-measurability of the indicator functions and the fact that Bochner conditional expectations commute with bounded linear forms such as $\langle\cdot,z_m\rangle$, see \citet[Section~2.6]{bochner}.
			In the general case, the strong $\mathcal{H}$-measurability of $V$ provides simple $\mathcal{H}$-measurable functions $\tilde V_M\to V$ pointwise. Replacing $\tilde V_M$ by $V_M:=\tilde V_M 1_{\{\|\tilde V_M\|\le 2\|V\|\}}$, which is still simple and $\mathcal{H}$-measurable, \rae{one gets} $V_M\to V$ with $\|V_M\|\le 2\|V\|$ pointwise. Since $|\langle U,V_M\rangle|\le 2\|U\|\,\|V\|$ is integrable, the conditional dominated convergence theorem applies to the left-hand side of \eqref{eq:pullout} written for $V_M$, while the right-hand side converges almost surely by continuity of the inner product. This proves~\eqref{eq:pullout}.

			\rae{The next step is to show that}
			\begin{align}\label{eq:condind-bochner}
				\E_{\mathcal{G},\mathcal{B}}(W_1)=\E_{\mathcal{F},\mathcal{B}}(W_1)\quad\text{a.s.}
			\end{align}
			For any $z\in H$, the conditional independence of $W_1$ and $W_2$ given $\mathcal{F}$ entails $\E_{\mathcal{G}}(\langle W_1,z\rangle)=\E_{\mathcal{F}}(\langle W_1,z\rangle)$ a.s., see \citet[Proposition~6.6]{kallenberg2002}.
			Commuting again the Bochner conditional expectation with bounded linear forms, this rewrites $\langle\E_{\mathcal{G},\mathcal{B}}(W_1),z\rangle = \langle\E_{\mathcal{F},\mathcal{B}}(W_1),z\rangle$ a.s., and letting $z$ run through the countable basis $(e_j)_{j\ge 1}$ yields \eqref{eq:condind-bochner}.

			The last step is to combine \eqref{eq:pullout} and \eqref{eq:condind-bochner}. The random variable $\E_{\mathcal{F},\mathcal{B}}(W_1)$ is $\mathcal{F}$-measurable and the conditional Jensen inequality $\|\E_{\mathcal{F},\mathcal{B}}(W_1)\|\le\E_{\mathcal{F}}(\|W_1\|)$ combined with the conditional independence of $\|W_1\|$ and $\|W_2\|$ given $\mathcal{F}$ yields
			\begin{align*}
			\E(\|\E_{\mathcal{F},\mathcal{B}}(W_1)\|\,\|W_2\|)&\le \E\big(\E_{\mathcal{F}}(\|W_1\|)\,\E_{\mathcal{F}}(\|W_2\|)\big)=\E\big(\E_{\mathcal{F}}(\|W_1\|\,\|W_2\|)\big)\\
			&=\E(\|W_1\|\,\|W_2\|)<+\infty,
			\end{align*}
			in view of \eqref{eq:bochner-integrability}. Whence, the tower property, \eqref{eq:pullout} applied with $(U,V,\mathcal{H})=(W_1,W_2,\mathcal{G})$, identity \eqref{eq:condind-bochner}, and again \eqref{eq:pullout}, now applied with $\mathcal{H}=\mathcal{F}$, $U=W_2$ and $V=\E_{\mathcal{F},\mathcal{B}}(W_1)$, entail
			\begin{align*}
			\E_\mathcal{F}(\langle W_1,W_2\rangle)=\E_\mathcal{F}\big(\E_\mathcal{G}(\langle W_1,W_2\rangle)\big)&=\E_\mathcal{F}\big(\langle \E_{\mathcal{F},\mathcal{B}}(W_1),W_2\rangle\big)
			=\langle \E_{\mathcal{F},\mathcal{B}}(W_1),\E_{\mathcal{F},\mathcal{B}}(W_2)\rangle\quad\text{a.s.},
			\end{align*}
			which ends the proof.
		\end{proof}
}%

		\noindent \rae{The explicit solution} of the optimization problem \eqref{eq:fepls} under the inverse single-index model \eqref{eq:single_index_model} \rae{may now be established.} The following proof is a simpler version of \citet[Proposition~1]{Bousebata2023} but it extends to the functional case. 
		\cp{%
		\begin{proof}[Proof of Proposition~\ref{prop:solution}]
			\rae{The first step is to rewrite} the conditional covariance. Let $w\in H$ be such that $\|w\|=1$ and let $y\in\R$ be such that $\bar{F}(y)>0$. By the Cauchy--Schwarz inequality in $H$, $|\langle w,X\rangle Y|\le \|X\|\,|Y|$ and $|\langle w,X\rangle|\le \|X\|$, so that the moment assumption ensures that all the expectations at hand below are finite; in particular, $m_{XY}(y)$ and $m_X(y)$ are well-defined Bochner integrals. Then,
			\begin{align*}
				\Cov(\langle w,X\rangle,Y | Y\geq y) &=  \frac{\E(\langle w,X\rangle Y  1_{\{Y\geq y\}})}{\bar{F}(y)}- \frac{\E(\langle w,X\rangle1_{\{Y\geq y\}})\,\E(Y  1_{\{Y\geq y\}})}{\bar{F}(y)^2}.\end{align*}
			Next, since Bochner integrals commute with bounded linear forms such as $\langle w,\cdot\rangle$, see \citet[Eq~(1.2)]{bochner}, one has $\E(\langle w,X\rangle Y1_{\{Y\ge y\}})=\langle w,m_{XY}(y)\rangle$ and $\E(\langle w,X\rangle 1_{\{Y\ge y\}})=\langle w,m_{X}(y)\rangle$, whence
			\begin{align*}
				\bar{F}^2(y)\, \Cov(\langle w,X\rangle,Y | Y\geq y)   &= \bar{F}(y)\langle w,m_{XY}(y)\rangle - \langle w,m_Y(y)m_{X}(y)\rangle   = \langle w, v(y)\rangle.
			\end{align*}
			The optimization problem now becomes $\argmax_{\|{ w}\|=1} \langle w,v(y)\rangle$. By the Cauchy--Schwarz inequality, $\langle w,v(y)\rangle \le  \|v(y)\|$ for any $w\in H$ such that ${\|{ w}\|=1}$, with equality if and only if $w=v(y)/\|v(y)\|$ since $v(y)\neq 0$. This proves both optimality and uniqueness, and ends the proof.
		\end{proof}
}%

		\noindent
		\cp{%
		\noindent The next lemma describes the exact conditional distribution of the whole sample given the intermediate order statistic $Y_{n-k+1,n}$. It is the key tool to handle tail-moments at the random threshold: it extends \citet[Lemma~2]{Bousebata2023} to this framework, see Lemma~\ref{lem:bousebata_random} below, and provides the second-moment identity of Lemma~\ref{lem:second_moment}, both of which are extensively used in the sequel. \rae{Given} $\{Y_{n-k+1,n}=y\}$, the conditional expectation is denoted by $\E_{k,y}(\cdot):=\E(\cdot\mid Y_{n-k+1,n}=y)$ and the conditional tail-moment of a generic random object $W_i$, $1\le i \le n$, by $m_{k,y}(W_i):=\E_{k,y}(W_i1_{\{Y_i\ge Y_{n-k+1,n}\}})$. In the sequel, $\mathcal{L}(\cdot)$ and $\mathcal{L}(\cdot\mid\cdot)$ stand for the distribution and the conditional distribution of the random object at hand.
\begin{lem}\label{lem:cond_sample}
			Let $(X_i,Y_i)_{1\le i \le n}$ be independent copies of $(X,Y)\in H\times \R$ and assume that $Y$ admits a density $f$. Let $k$ be an integer such that $2\le  k < n$ and let $y\in\R$ be such that $f(y)>0$ and $0<F(y)<1$. Then, conditionally on $\{Y_{n-k+1,n}=y\}$, the sample $(X_i,Y_i)_{1\le i\le n}$ is distributed as follows: a pair $(i_0,S)$ is drawn uniformly at random among the $n\binom{n-1}{k-1}$ pairs made of an index $i_0\in\{1,\dots,n\}$ and a subset $S\subset\{1,\dots,n\}\setminus\{i_0\}$ with $|S|=k-1$; given $(i_0,S)$, the pairs $(X_i,Y_i)$, $1\le i\le n$, are independent, with
			\begin{align*}
			(X_j,Y_j)&\sim \mathcal{L}\big((X,Y)\mid Y>y\big) \mbox{ for } j\in S\quad\mbox{and}\\
			(X_j,Y_j)&\sim \mathcal{L}\big((X,Y)\mid Y<y\big) \mbox{ for } j\notin S\cup\{i_0\},
			\end{align*}
			while $Y_{i_0}=y$ and $X_{i_0}\sim \mathcal{L}(X\mid Y=y)$. In words, $i_0$ is the index of the observation achieving the order statistic, the indices of $S$ are those of the $k-1$ observations falling above the threshold $y$, and the $n-k$ remaining observations fall below. In particular, writing $e_>(\Psi):=\E(\Psi(X,Y)\mid Y>y)$ and $e_=(\Psi):=\E(\Psi(X,Y)\mid Y=y)$ for any measurable function $\Psi:H\times\R\to[0,+\infty]$, one has, for any such $\Psi,\Psi'$ and any $1\le i\ne \ell\le n$:
			\begin{align}
				\label{eq:strata-one}
				\E_{k,y}\big(\Psi(X_i,Y_i)1_{\{Y_i\ge y\}}\big) &= \f{k-1}{n}\,e_>(\Psi)+\ff{n}\,e_=(\Psi)\quad\text{and}\\
				\nonumber
				\E_{k,y}\big(\Psi(X_i,Y_i)\Psi'(X_\ell,Y_\ell)1_{\{Y_i\ge y\}}1_{\{Y_\ell\ge y\}}\big) &= \f{(k-1)(k-2)}{n(n-1)}\,e_>(\Psi)\,e_>(\Psi')\\
				\label{eq:strata-two}
				&+\f{k-1}{n(n-1)}\Big(e_>(\Psi)\,e_=(\Psi')+e_=(\Psi)\,e_>(\Psi')\Big).
			\end{align}
			\cp{By decomposing into positive and negative parts, \eqref{eq:strata-one} extends to any measurable $\Psi:H\times\R\to\R$ such that $\E_{k,y}\big(|\Psi(X_i,Y_i)|1_{\{Y_i\ge y\}}\big)<+\infty$.}
		\end{lem}
		\begin{proof}
			\rae{The representation is first proved} for the sub-sample $(Y_1,\dots,Y_n)$. Since $F$ is continuous, the values $Y_1,\dots,Y_n$ are a.s. pairwise distinct. For $(i_0,S)$ as in the statement, define the event
			\begin{align*}
			E(i_0,S):=\big\{Y_j>Y_{i_0}\ \mbox{for any}\ j\in S\big\}\cap\big\{Y_j<Y_{i_0}\ \mbox{for any}\ j\notin S\cup\{i_0\}\big\}.
			\end{align*}
			On $E(i_0,S)$, exactly $k-1$ values exceed $Y_{i_0}$, so that $Y_{n-k+1,n}=Y_{i_0}$. Up to a $\p$-null set, the event $\{Y_{n-k+1,n}\le y\}$ is thus the disjoint union over the $n\binom{n-1}{k-1}$ pairs $(i_0,S)$ of the events $E(i_0,S)\cap\{Y_{i_0}\le y\}$. For any Borel sets $B_1,\dots,B_n$ of $\R$, conditioning on $Y_{i_0}$ and using independence yield
			\begin{align*}
			&\p\Big(\bigcap_{i=1}^n\{Y_i\in B_i\}\cap E(i_0,S)\cap\{Y_{i_0}\le y\}\Big)\\
			&\qquad=\int_{-\infty}^{y}f(u)\,1_{B_{i_0}}(u)\prod_{j\in S}\p\big(Y\in B_j,\, Y>u\big)\prod_{j\notin S\cup\{i_0\}}\p\big(Y\in B_j,\, Y<u\big)\,\dd u .
			\end{align*}
			Choosing $B_1=\dots=B_n=\R$ and summing over the pairs $(i_0,S)$ recovers the classical density of the order statistic, $f_{Y_{n-k+1,n}}(u)=n\binom{n-1}{k-1}f(u)\bar{F}^{k-1}(u)F^{n-k}(u)$. The conditional distribution given $\{Y_{n-k+1,n}=y\}$ is, for almost every such $y$, the ratio of the integrand at $u=y$ to this density:
			\begin{align*}
			&\p\Big(\bigcap_{i=1}^n\{Y_i\in B_i\}\cap E(i_0,S)\,\Big|\, Y_{n-k+1,n}=y\Big)\\
			&\qquad=\ff{n\binom{n-1}{k-1}}\,1_{B_{i_0}}(y)\prod_{j\in S}\p\big(Y\in B_j\mid Y>y\big)\prod_{j\notin S\cup\{i_0\}}\p\big(Y\in B_j\mid Y<y\big).
			\end{align*}
			In other words, each pair $(i_0,S)$ has probability $\big(n\binom{n-1}{k-1}\big)^{-1}$ and, given $(i_0,S)$, the random variables $Y_j$, $j\ne i_0$, are independent with $Y_j\sim\mathcal{L}(Y\mid Y>y)$ for $j\in S$ and $Y_j\sim\mathcal{L}(Y\mid Y<y)$ \rtworev{for $j\notin S\cup\{i_0\}$}, while $Y_{i_0}=y$.

			\rtworev{The representation is now extended from the responses $Y_i$ to the pairs $(X_i,Y_i)$.} \cp{Conditionally on} $(Y_1,\dots,Y_n)$, the random variables $X_1,\dots,X_n$ are independent with $X_i\sim\mathcal{L}(X\mid Y=Y_i)$. Both $(i_0,S)$ and the event $\{Y_{n-k+1,n}=y\}$ only depend on $(Y_1,\dots,Y_n)$, so this remains true under the above conditioning. Hence, given $(i_0,S)$, the pairs $(X_i,Y_i)$, $1\le i\le n$, are independent, $Y_{i_0}=y$ with $X_{i_0}\sim\mathcal{L}(X\mid Y=y)$, and, for $j\in S$, the pair $(X_j,Y_j)$ is obtained by drawing $Y_j\sim\mathcal{L}(Y\mid Y>y)$ and then $X_j\sim\mathcal{L}(X\mid Y=Y_j)$. This two-step draw has distribution $\mathcal{L}((X,Y)\mid Y>y)$ since, for any measurable sets $A\subset H$ and $B\subset\R$, the tower property gives
			\begin{align*}
			\p(X_j\in A,\,Y_j\in B)=\E\big(1_{\{Y\in B\}}\,\p(X\in A\mid Y)\,\big|\,Y>y\big)=\p\big(X\in A,\,Y\in B\mid Y>y\big).
			\end{align*}
			\rtworev{For $j\notin S\cup\{i_0\}$, the same two-step draw, with $\mathcal{L}(Y\mid Y<y)$ in place of $\mathcal{L}(Y\mid Y>y)$ and the event $\{Y<y\}$ in place of $\{Y>y\}$, gives} $(X_j,Y_j)\sim\mathcal{L}((X,Y)\mid Y<y)$, which proves the representation.

			It remains to derive \eqref{eq:strata-one} and \eqref{eq:strata-two}. Under the representation, $\{Y_i\ge y\}=\{i\in S\cup\{i_0\}\}$ a.s., so that, conditionally on $(i_0,S)$,
			\begin{align*}
			\E_{k,y}\big(\Psi(X_i,Y_i)1_{\{Y_i\ge y\}}\mid (i_0,S)\big)=e_>(\Psi)\,1_{\{i\in S\}}+e_=(\Psi)\,1_{\{i_0=i\}},
			\end{align*}
			and, for $i\ne\ell$, the conditional independence of $(X_i,Y_i)$ and $(X_\ell,Y_\ell)$ given $(i_0,S)$ entails
			\begin{align*}
			&\E_{k,y}\big(\Psi(X_i,Y_i)\Psi'(X_\ell,Y_\ell)1_{\{Y_i\ge y\}}1_{\{Y_\ell\ge y\}}\mid (i_0,S)\big)\\
			&\qquad=\Big(e_>(\Psi)\,1_{\{i\in S\}}+e_=(\Psi)\,1_{\{i_0=i\}}\Big)\Big(e_>(\Psi')\,1_{\{\ell\in S\}}+e_=(\Psi')\,1_{\{i_0=\ell\}}\Big).
			\end{align*}
			By exchangeability of $(i_0,S)$,
			\begin{align*}
			\p(i_0=i)=\ff{n},&\quad \p(i\in S)=\f{k-1}{n},\\
			\p(i\in S,\,\ell=i_0)=\f{k-1}{n(n-1)}&\quad\text{and}\quad \p(i\in S,\,\ell\in S)=\f{(k-1)(k-2)}{n(n-1)} ,
			\end{align*}
			\rae{while} $\{i_0=i\}\cap\{i_0=\ell\}=\emptyset$. Taking the expectation with respect to $(i_0,S)$ in the two previous displays yields \eqref{eq:strata-one} and \eqref{eq:strata-two}, which ends the proof.
		\end{proof}
		\noindent The first consequence extends \citet[Lemma~2]{Bousebata2023} to the random threshold framework; it corrects the deterministic-level statement by the diagonal contribution of the observation achieving the order statistic.
		\begin{lem}\label{lem:bousebata_random}
			Under the assumptions of Lemma~\ref{lem:cond_sample}, let $h:\R\to\R$ be a measurable function such that $\E(|h(Y)|1_{\{Y> y\}})<+\infty$. Then, for any $1\le i\le n$,
			\begin{align}\label{eq:bousebata-exact}
				m_{k,y}(h(Y_i))=\f{k-1}{n}\,\ff{\bar{F}(y)}\int_y^{+\infty}h(t)f(t)\,\dd t+\f{h(y)}{n} ;
			\end{align}
			in particular, $m_{k,y}(1)=k/n$. If moreover $h\in\RV_{\rho}(+\infty)$ with $\rho\gamma<1$ and \eqref{hyp:vonmises} holds, then, as $y\to+\infty$,
			\begin{align*}
				m_{k,y}(h(Y_i))=\f{k-1}{n}\,\f{h(y)}{1-\rho\gamma}\,\big(1+o(1)\big)+\f{h(y)}{n},
			\end{align*}
			where the $o(1)$ term depends neither on $k$ nor on $n$.
		\end{lem}
		\begin{proof}
			Identity \eqref{eq:bousebata-exact} is \eqref{eq:strata-one} applied to $\Psi(x,t)=h(t)$, since
			$$
			e_>(\Psi)=\bar{F}(y)^{-1}\int_y^{+\infty}h(t)f(t)\,\dd t \quad\mbox{ and }\quad e_=(\Psi)=h(y).
			$$ The asymptotic equivalent $\bar{F}(y)^{-1}\int_y^{+\infty}h(t)f(t)\dd t\to h(y)/(1-\rho\gamma)$ as $y\to+\infty$, which involves neither $k$ nor $n$, is provided by \citet[Lemma~2]{Bousebata2023}, which ends the proof.
		\end{proof}
		\begin{Remark}\label{rmk:bousebata-atom}
			The diagonal term $h(y)/n$ in \eqref{eq:bousebata-exact} is negligible in \rae{all the uses made below,} where $k=k_n\to+\infty$, but it must be kept for the identity to be exact, as the check $m_{k,y}(1)=k/n$ shows. In the sequel, \citet[Lemma~2]{Bousebata2023} remains the reference for tail-moments at a deterministic level, and Lemma~\ref{lem:bousebata_random} whenever the threshold is the random order statistic $Y_{n-k+1,n}$.
		\end{Remark}
		\noindent The second consequence is the exact second conditional moment of empirical tail-moments. \rae{The} inflation factors of the two conditional variances below are $(k-1)/n^2$ and $1/n^2$, and not $1/n$ times the corresponding probabilities: the indicator functions $1_{\{Y_i\ge Y_{n-k+1,n}\}}$, $1\le i \le n$, are not conditionally independent given $Y_{n-k+1,n}$, and the exact combinatorial structure of Lemma~\ref{lem:cond_sample} is what makes the computation below valid.
		\begin{lem}\label{lem:second_moment}
			Under the assumptions of Lemma~\ref{lem:cond_sample}, let $Z_i:=\Psi(X_i,Y_i)$, $1\le i\le n$, for some measurable function $\Psi:H\times\R\to H$ such that $e_>(\|\Psi\|^2)<+\infty$ and $e_=(\|\Psi\|^2)<+\infty$, and set
			\begin{align*}
			\hat m_Z(Y_{n-k+1,n})&:=\ff{n}\sum_{i=1}^n Z_i1_{\{Y_i\ge Y_{n-k+1,n}\}},\\
			\mu_>:=\E_{\mathcal{B}}(Z_1\mid Y_1>y)&\quad\mbox{and}\quad \mu_=:=\E_{\mathcal{B}}(Z_1\mid Y_1=y).
			\end{align*}
			Then, $m_{k,y}(Z_1)=\f{k-1}{n}\mu_>+\ff{n}\mu_=$ and
			\begin{align}
				\nonumber
				\E_{k,y}\big(\|\hat m_Z(Y_{n-k+1,n})\|^2\big)&=\|m_{k,y}(Z_1)\|^2+\f{k-1}{n^2}\Big(e_>(\|\Psi\|^2)-\|\mu_>\|^2\Big)+\ff{n^2}\Big(e_=(\|\Psi\|^2)-\|\mu_=\|^2\Big)\\
				\label{eq:second-moment-bound}
				&\le \|m_{k,y}(Z_1)\|^2+\ff{n}\,m_{k,y}\big(\|Z_1\|^2\big).
			\end{align}
			The same statements hold for real-valued $Z_i$, replacing inner products and norms accordingly.
		\end{lem}
		\begin{proof}
			The formula for $m_{k,y}(Z_1)$ is \eqref{eq:strata-one}; it extends to $H$-valued integrands by testing against a countable orthonormal basis, since bounded linear forms commute with Bochner conditional expectations. Next, expanding the squared norm,
			$$
			\E_{k,y}\big(\|\hat m_Z(Y_{n-k+1,n})\|^2\big)=\ff{n^2}\sum_{i\ne\ell}\E_{k,y}\big(\langle Z_i1_{\{Y_i\ge y\}},Z_\ell1_{\{Y_\ell\ge y\}}\rangle\big)+\ff{n^2}\sum_{i=1}^n\E_{k,y}\big(\|Z_i\|^21_{\{Y_i\ge y\}}\big) .
			$$
			\cp{Conditionally on} $(i_0,S)$ from the representation of Lemma~\ref{lem:cond_sample}, the random variables $Z_i1_{\{Y_i\ge y\}}$, $1\le i\le n$, are mutually independent, so that, for any $i\ne\ell$, the unconditional version of Lemma~\ref{lem:bochner_random} applied under $\p(\cdot\mid (i_0,S),\,Y_{n-k+1,n}=y)$ factorizes each off-diagonal term into the inner product of the two corresponding conditional means. Averaging over $(i_0,S)$ with the probabilities collected in the proof of Lemma~\ref{lem:cond_sample} gives, exactly as in \eqref{eq:strata-two},
			$$
			\sum_{i\ne\ell}\E_{k,y}\big(\langle Z_i1_{\{Y_i\ge y\}},Z_\ell1_{\{Y_\ell\ge y\}}\rangle\big)=(k-1)(k-2)\,\|\mu_>\|^2+2(k-1)\,\langle\mu_>,\mu_=\rangle .
			$$
			Besides, \eqref{eq:strata-one} applied to $\|\Psi\|^2$ yields $\sum_{i=1}^n\E_{k,y}(\|Z_i\|^21_{\{Y_i\ge y\}})=(k-1)\,e_>(\|\Psi\|^2)+e_=(\|\Psi\|^2)$, while
			\begin{align*}
			\|m_{k,y}(Z_1)\|^2&=\Big(\f{k-1}{n}\Big)^2\|\mu_>\|^2+\f{2(k-1)}{n^2}\langle\mu_>,\mu_=\rangle+\ff{n^2}\|\mu_=\|^2 .
			\end{align*}
			Collecting the last three displays, the cross terms $\langle\mu_>,\mu_=\rangle$ cancel out and the stated identity follows from $(k-1)(k-2)-(k-1)^2=-(k-1)$. At last, the bound \eqref{eq:second-moment-bound} is obtained by discarding the nonpositive terms $-\|\mu_>\|^2$ and $-\|\mu_=\|^2$ and recognizing $\ff{n}m_{k,y}(\|Z_1\|^2)=\f{k-1}{n^2}\,e_>(\|\Psi\|^2)+\ff{n^2}\,e_=(\|\Psi\|^2)$ from \eqref{eq:strata-one}, which ends the proof.
		\end{proof}
		\begin{Remark}\label{rmk:second-moment-centred}
			If, in addition, $\E_{\mathcal{B}}(Z_1\mid Y_1=t)=0$ for almost every \rorange{$t$, then, for almost every $y$ with respect to the law of $Y_1$,} $\mu_>=\mu_==0$, $m_{k,y}(Z_1)=0$, and \eqref{eq:second-moment-bound} is an equality:
			$$\E_{k,y}\big(\|\hat m_Z(Y_{n-k+1,n})\|^2\big)=\ff{n}\,m_{k,y}(\|Z_1\|^2).$$ This covers in particular the conditionally centered residual $\xi$ of Section~\ref{sec-inf} through $Z_i=\xi_i$, and makes rigorous the corresponding step in the proof of Proposition~\ref{prop:modelfree}.
		\end{Remark}
}%

		\subsection{Empirical tail-moments}
		\label{subsection:empirical_tail_moment}

		The next lemma establishes a bound on the (random) tail-moments of $\varphi(Y)\eps$.
		\begin{lem}\label{lem:norm_noise}
			Assume that \eqref{hyp:test}, \eqref{hyp:link}, \eqref{hyp:vonmises},   \eqref{hyp:noise_cond}, \eqref{hyp:2cgamma} and \eqref{hyp:qcgamma} hold. Let $k:=k_n\to+\infty$ be an integer deterministic sequence such that $ k / n \to 0$ and $y_{n,k}\sim U(n/k)$ as $n\to+\infty$. Let $\delta_{n,k}:=(g(y_{n,k})(k/n)^{1/q})^{-1}$. Then,
			\begin{align*}
				\f{\|\hat m_{\vfi(Y)\eps}(Y_{n-k+1,n})\|}{\sg{m_{\vfi  g(Y)}}(y_{n,k})} = O_{\p}\left( \delta_{n,k} \right) \xrightarrow[n\to +\infty]{\cp{\p}}0.
			\end{align*}
		\end{lem}
		\begin{proof}
			For clarity, denote $Z_1:=\vfi(Y_1)\eps_1$ and $Y_{n,k}:=Y_{n-k+1,n}$. Let $y>0$\cp{. Highlighting} the dependence on $y$, \rae{denote} the conditional expectation given $\{Y_{n,k}=y\}$ by $\E_{k,y}(\cdot):=\E(\cdot|Y_{n,k}=y)$ and\cp{, similarly,} \rae{denote} the conditional tail-moments of any generic random variable $W_i$, $1\le i \le n$, by 
			$$
			m_{k,y}(W_i) := \E_{k,y} ( W_i 1_{\{Y_i\ge Y_{n,k}\}}).
			$$ 
			\rae{The first step is to compute} the expectation of $\|\hat m_{Z}(Y_{n,k})\|^2$ which is by construction,
			\begin{align*}
				\|\hat m_{Z}(Y_{n,k})\|^2&= \ff{n^2} \sum_{i_1,i_2=1}^n \langle Z_{i_1}  ,Z_{i_2}  \rangle 1_{\{ Y_{i_1}\ge Y_{n,k} \}}1_{\{ Y_{i_2}\ge Y_{n,k} \}}.
			\end{align*}
			Taking the conditional expectation and splitting the sum, it follows,
			\begin{align*}
				\E_{k,y}(\|\hat m_{Z}(Y_{n,k})\|^2) &=\ff{n^2} \sum_{i_1\ne i_2}^n \E_{k,y} \left(\langle Z_{i_1} 1_{\{ Y_{i_1}\ge Y_{n,k} \}}  , Z_{i_2} 1_{\{ Y_{i_2}\ge Y_{n,k} \}}  \rangle \right)
				+ \ff{n^2} \sum_{i=1}^n \E_{k,y} \left(\|Z_{i} \|^2 1_{\{ Y_{i}\ge Y_{n,k} \}}\right).
			\end{align*}
			\cp{%
Lemma~\ref{lem:second_moment} and equidistribution entail
			\begin{align*}
				\E_{k,y}(\|\hat m_{Z}(Y_{n,k})\|^2)&\le 
				\|m_{k,y}(Z_1)\|^2+ \ff{n}\, m_{k,y}(\|Z_1\|^2).
			\end{align*}
}%
			The main tool of the proof is the conditional Markov inequality, which gives, for any $\eta'>0$, 
			\begin{align}\label{eq:markov}
				\p_{k,y}\left( {\f{\|\hat m_{Z}(Y_{n,k})\|}{\sg{m_{\vfi g(Y)}}(y_{n,k})}} >\eta'\right) \le \f{1}{\eta'^2} \E_{k,y}\left(\|\hat m_{Z}(y)\|^2\right) m^{-2}_{\sg{\vfi g(Y)}}(y_{n,k}).
			\end{align}
			\sg{Here $y_{n,k}$ is deterministic, so $m_{\vfi g(Y)}(y_{n,k})$ is a tail-moment at a deterministic level.}
            According to \citet[Lemma~2]{Bousebata2023} under \eqref{hyp:2cgamma}, there exists some constant $c\in (0,+\infty)$ independent of $n$ such that, for $n$ large enough, Markov's inequality~\eqref{eq:markov} becomes
			$$ 
			\p_{k,y}\left( {\f{\|\hat m_{Z}(Y_{n,k})\|}{\sg{m_{\vfi g(Y)}}(y_{n,k})}} >\eta'\right) \le \f{c}{\eta'^2} \sg{(\vfi g  \bar{F})}^{-2}(y_{n,k}) \left( \|m_{k,y}(Z_1)\|^2 \vee \ff{n}m_{k,y}(\|Z_1\|^2)\right) .
			$$
			\cp{%
The next step is to bound the term involving the conditional tail-moments. \rae{The argument uses the fact} that the conditional expectation is a contraction for the norm of $H$, by the conditional Jensen inequality for Bochner integrals, so that $\|m_{k,y}(Z_1)\|\le m_{k,y}(\|Z_1\|)= m_{k,y}({\vfi(Y_1)}\|\eps_1\|)$, where the last equality uses $\vfi\ge 0$. By the conditional Cauchy--Schwarz inequality in $L^2(\Omega,\mathcal{A},\p)$, \rae{one} may then bound 
			\begin{align*}
				\|m_{k,y}(Z_1)\|^2\le m^2_{k,y}({\vfi(Y_1)}\|\eps_1\|) 
				\le m_{k,y}(\|Z_1\|^2)\,m_{k,y}(1),
			\end{align*}
}%
			and therefore
			\begin{align*} 
				\| m_{k,y}(Z_1)\|^2 \vee \ff{n}m_{k,y}(\|Z_1\|^2)  &\le  m_{k,y}(\|Z_1\|^2)\left(\frac{1}{n} \vee m_{k,y}(1)\right).
			\end{align*}
			\cp{%
Besides, Lemma~\ref{lem:bousebata_random} yields $m_{k,y}(1)=k/n$, so that 
			\begin{align} 
				\label{eq-new-bound}
				\| m_{k,y}(Z_1)\|^2 \vee \ff{n}m_{k,y}(\|Z_1\|^2)  &\le  \frac{k}{n}\, m_{k,y}(\|Z_1\|^2).
			\end{align}
}%
			The focus now turns to the tail-moment of $\|Z_1\|^2$ in the previous inequality. The $q$-moment of $\|\eps_1\|$ is finite under~\eqref{hyp:noise_cond} so that Hölder's inequality yields
			\begin{align}
				\label
				{eq-tmp1}
				m_{k,y}(\|Z_1\|^2)&\le\E_{k,y}\left(\|\eps_1\|^{q}\right)^{2/q}   m^{1-2/q}_{k,y}({\vfi(Y_1)}^{{2q}/(q-2)}).
			\end{align}
			\rae{The} latter tail-moment is finite if ${2\tau q}/(q-2) < {1/\gamma}$ which is satisfied for any $q>2$ whenever $\tau\le 0$\rtworev{, since $\gamma>0$ by~\eqref{hyp:qcgamma}, as $q>2$ and $\kappa>0$}. This existence condition is also satisfied  when $\tau>0$ under~\eqref{hyp:2cgamma}--\eqref{hyp:qcgamma}. Indeed, the condition at hand is equivalent to $\tau<(1-2/q)/(2\gamma)$ while~\eqref{hyp:2cgamma} is the same as $\tau <(1-2\kappa\gamma)/(2\gamma)$ and~\eqref{hyp:qcgamma} writes as $2/q < 2\kappa\gamma$. Moreover, $\vfi^{{2q}/(q-2)} \in \RV_{{2\tau q}/(q-2)}(+\infty)$  so that Lemma~\ref{lem:bousebata_random} yields\cp{, uniformly in $y\ge y_{n,k}/2$,} as $n\to+\infty$: 
			\begin{align}
				\label{eq-tmp2}
				m^{1-2/q}_{k,y}({\vfi(Y_1)^{{2q}/(q-2)}}) \sim {\left(1- \f{2\tau q \gamma }{q-2}\right)^{(2/q)-1}} \vfi^2(y) \left(\f{k-1}{n}\right)^{1-2/q} . 
			\end{align}
			It follows from~\eqref{eq-tmp1} and~\eqref{eq-tmp2}
			that there exists a constant $c\in (0,+\infty)$ independent of $n$ such that for $n$ large enough\cp{, uniformly in $y\ge y_{n,k}/2$}, \begin{align}\label{eq:aux1}
				m_{k,y}(\|Z_1\|^2)&\le c \E_{k,y}\left(\|\eps_1\|^{q}\right)^{2/q} \vfi^2(y)\left(\f{k-1}{n}\right)^{1-2/q}.
			\end{align}
			Combining~\eqref{eq-new-bound} and~\eqref{eq:aux1}, one has for some constant $c>0$ independent of $n$ and for $n$ large enough,
			\begin{align*}
				\| m_{k,y}(Z_1)\|^2 \vee \ff{n}m_{k,y}(\|Z_1\|^2)  &\le
				c \E_{k,y}\left(\|\eps_1\|^{q}\right)^{2/q} \vfi^2(y)\left(\f{k-1}{n}\right)^{2(1-1/q)}.
			\end{align*}
			Going back to Markov's bound \eqref{eq:markov}, for some constant $c\in (0,+\infty)$ independent of $n$,
			\begin{align*}
				\p_{k,y}\left( {\f{\|\hat m_{\cp{Z}}(Y_{n,k})\|}{\sg{m_{\vfi g(Y)}}(y_{n,k})}} >\eta'\right) &\le \f{c}{\eta'^2} \E_{k,y}\left(\|\eps_1\|^{q}\right)^{2/q} \Delta_{n,k}(y),
			\end{align*}
			where 
			\begin{align}
				\label{eq-delta}
				\Delta_{n,k} (y):=  (\sg{\vfi  g  \bar{F}})^{-2}(y_{n,k})\left(\f{k-1}{n}\right)^{2(1-1/q)} \vfi^2(y).
			\end{align}
			Define two sequences $y^+_{n,k}:=3y_{n,k}/2$ and $y^-_{n,k}:=y_{n,k}/2$. \cp{Since $\vfi\in\RV_\tau(+\infty)$ \rtworev{and $y_{n,k}\to+\infty$ (shown below)}, the uniform convergence theorem for regularly varying functions \rtworev{\citep[Theorem~1.5.2]{bingham_goldie_teugels_1987}} yields
			$$
			\sup_{y\in [y^-_{n,k},y^+_{n,k}]}\f{\vfi(y)}{\vfi(y_{n,k})}=\sup_{\lam\in[1/2,3/2]}\f{\vfi(\lam y_{n,k})}{\vfi(y_{n,k})}\longrightarrow \left(\ff{2}\right)^{\tau}\vee \left(\f{3}{2}\right)^{\tau}\quad\text{as }n\to+\infty.
			$$
			Hence, when $y\in [y^-_{n,k},y^+_{n,k}]$ and $n$ is large enough, $ \vfi^2(y)\le c\vfi^2(y_{n,k})$ for some constant $c>0$ independent of $n$.} At last, $y_{n,k}\sim U(n/k)$ implies $\bar{F}(y_{n,k}) \sim \bar{F}(U(n/k) )=k/n$ as $n\to +\infty$.
			\rtworev{Indeed, $Y$ has a density by~\eqref{hyp:vonmises}, so that $F$ is continuous and $\bar F(U(n/k))=k/n$. Besides, $\gamma>0$ by~\eqref{hyp:qcgamma}, since $q>2$ and $\kappa>0$, so that the index $-1/\gamma-1$ in~\eqref{hyp:vonmises} is below $-1$ and $\bar F\in\RV_{-1/\gamma}(+\infty)$ by Karamata's theorem, as noted after~\eqref{hyp:vonmises}; being positive near $+\infty$ and non-increasing, $\bar F$ is positive on $\R$, so that $F<1$ on $\R$ and $U(t)=F^-(1-1/t)\to+\infty$ as $t\to+\infty$. Hence $U(n/k)\to+\infty$, since $n/k\to+\infty$, and $y_{n,k}\to+\infty$ as well. The uniform convergence theorem \citep[Theorem~1.5.2]{bingham_goldie_teugels_1987}, applied to $\bar F$, states that $\bar F(tx)/\bar F(t)\to x^{-1/\gamma}$ as $t\to+\infty$, uniformly in $x\in[1/2,2]$. Since $y_{n,k}/U(n/k)\to1$, this ratio lies in $[1/2,2]$ for $n$ large enough, so that $\big|\bar F(y_{n,k})/\bar F(U(n/k))-(y_{n,k}/U(n/k))^{-1/\gamma}\big|\le\sup_{x\in[1/2,2]}\big|\bar F(xU(n/k))/\bar F(U(n/k))-x^{-1/\gamma}\big|\to0$, while $(y_{n,k}/U(n/k))^{-1/\gamma}\to1$. Hence $\bar F(y_{n,k})/\bar F(U(n/k))\to1$. Together with the bound on $\vfi^2$ above, this gives}, for some other constant $c>0$,
			$$
			\Delta_{n,k} (y) \leq c g^{-2}(y_{n,k})\left(\f{k}{n}\right)^{-2/q} .
			$$
			Taking account of \eqref{hyp:noise_cond}, it follows that there exists a constant $c>0$ independent of $n$ such that for $n$ large enough,
			$$
			\sup_{y\in [y^-_{n,k},y^+_{n,k}]} \big\{\E_{k,y}\left(\|\eps_1\|^{q}\right)^{2/q}\Delta_{n,k} (y)\big\}\le c \delta^2_{n,k}. 
			$$
			\cp{%
\rae{Let} $\eta>0$. Since $\eta'$ is arbitrary, one may pick $\eta'$ of the form $\delta_{n,k} ( 2c/\eta)^{1/2}$ and combining with what precedes, \rae{one has} shown that~\eqref{eq:markov} becomes, when $y^-_{n,k}\le y\le y^+_{n,k}$,
			\begin{align}\label{eq:markov_random}
				\p_{k,y}\left(  \delta^{-1}_{n,k} {\f{\|\hat m_{Z}(Y_{n,k})\|}{m_{\vfi g(Y)}(y_{n,k})}} >(2c/\eta)^{1/2}\right) &\le \f{\eta}{2} .
			\end{align}
			In order to derive the result, \rae{one needs} to integrate the conditional probability~\eqref{eq:markov_random} with respect to $y$ on the whole real \cp{line}, in view of Bayes' formula. After splitting the integral domain, the proof is complete if the next quantity is negligible when $n$ is large, 
			\begin{align*}
				P_{n,k,\eta}(\cp{-\infty},+\infty) &= P_{n,k,\eta}(\cp{-\infty},y^-_{n,k})+ P_{n,k,\eta}(y^-_{n,k},y^+_{n,k})+ P_{n,k,\eta}(y^+_{n,k},+\infty),
			\end{align*}
			where 
			$$
			P_{n,k,\eta}(a,b):=\int_{a}^{b}\p_{k,y}\left(  \delta^{-1}_{n,k} {\f{\|\hat m_{Z}(Y_{n,k})\|}{m_{\vfi g(Y)}(y_{n,k})}} >(2c/\eta)^{1/2}\right) f_{Y_{n,k}}(y) \dd y ,
			$$
			for any $-\infty\le a\le b \le +\infty$.
			First, \eqref{eq:markov_random} readily provides that $P_{n,k,\eta}(y^-_{n,k},y^+_{n,k})\le \eta/2$ for any $\eta>0$. Concerning the remaining two terms, \rae{the conditional probability is bounded} by one, \rae{and the argument uses} the complementary event and \rae{invokes} the fact that $Y_{n,k}/U(n/k)\xrightarrow{\p}1$ as $n\to+\infty$ 
			to get, for $n$ large enough:
			\begin{align*}
				P_{n,k,\eta}(\cp{-\infty},y^-_{n,k}) + P_{n,k,\eta}(y^+_{n,k},+\infty)&\le  1- \p\left(   y^-_{n,k}<  Y_{n,k}< y^+_{n,k} \right)\le \f{\eta}{2}.
			\end{align*}
			By Bayes' formula, the conclusion follows \rae{since it has been} shown that, for any $\eta>0$ and $n$ large enough, 
			$$
			\p\left(  \delta^{-1}_{n,k} {\f{\|\hat m_{Z}(Y_{n,k})\|}{m_{\vfi g(Y)}(y_{n,k})}} >(2c/\eta)^{1/2}\right)\le \f{\eta}{2}+\f{\eta}{2}=\eta,
			$$
			which is exactly the claimed $O_\p(\delta_{n,k})$ bound.
}
		\end{proof}

		Finally, Lemma~\ref{lem:inner_prod_noise_L2_random} below is dedicated to the control of the (random) tail-moments of $\langle  \beta,\vfi(Y) \eps \rangle$.

		\begin{lem}\label{lem:inner_prod_noise_L2_random}
			Let $\beta\in H$ be such that $\|\beta\|=1$ \cp{and suppose \eqref{hyp:vonmises} holds.} Assume that \eqref{hyp:test}, \eqref{hyp:link}, \eqref{hyp:noise_cond}, \eqref{hyp:2cgamma} and \eqref{hyp:qcgamma} hold. Let $k:=k_n\to+\infty$ be an integer deterministic sequence such that $ k / n\to 0$ and $y_{n,k}\sim U(n/k)$ as $n\to+\infty$. Let $\delta_{n,k}:=(g(y_{n,k})(k/n)^{1/q})^{-1}$. Then, 
			\begin{align*}
				\f{\hat m_{\langle  \beta,\vfi(Y) \eps \rangle}(Y_{n-k+1,n})}{ \sg{m_{\vfi  g(Y)}}(y_{n,k})} &=O_{\p}\left( \delta_{n,k}\right) \xrightarrow[n\to +\infty]{\cp{\p}}0.
			\end{align*}
		\end{lem}
		\begin{proof}
			Let $y>0$. Again, \rae{denote} $Z_1:=\vfi(Y_1)\eps_1$, $Y_{n,k}:=Y_{n-k+1,n}$, the conditional expectation given $\{Y_{n,k}=y\}$ by $\E_{k,y}(\cdot):=\E(\cdot|Y_{n,k}=y)$ and $m_{k,y}(W_i) := \E_{k,y} ( W_i 1_{\{Y_i\ge Y_{n,k}\}})$.
			The proof relies on the conditional Markov inequality, which states that 
			\begin{align}\label{eq:chebyshev}
				\p_{k,y}\left( \left| \f{\hat m_{\langle \beta,  \vfi(Y) \eps \rangle}(Y_{n,k})}{ \sg{m_{\vfi  g(Y)}}(y_{n,k})} \right| >\eta'\right) &\le \frac{1}{\eta'^2}\f{\E_{k,y}\left(\hat m^2_{\langle   \beta, \vfi(Y) \eps \rangle}(Y_{n,k})\right) }{\sg{m^2_{\vfi g(Y)}}(y_{n,k})},\quad \forall \eta'>0. 
			\end{align}
			The goal is to show that the RHS in \eqref{eq:chebyshev} converges to zero for large $n$.
			\rae{The first step is} a preliminary computation of the expectation. One has 
			\begin{align*}
				\hat m^2_{\langle   \beta,  \vfi(Y)\eps \rangle}(Y_{n,k})&=  \ff{n^2} \sum_{i,j=1}^n \langle  \beta ,\vfi(Y_i) \eps_{i} \rangle \langle  \beta ,\vfi(Y_j) \eps_{j} \rangle 1_{\{ Y_i\ge Y_{n,k} \}}  1_{\{ Y_j\ge Y_{n,k} \}} .
			\end{align*}
			\cp{%
Taking the conditional expectation and applying the real-valued case of Lemma~\ref{lem:second_moment} to $z_i:=\langle\beta,\vfi(Y_i)\eps_i\rangle$, $1\le i\le n$, together with equidistribution, yield
			\begin{align}
				\nonumber
				\E_{k,y}\big(\hat m^2_{\langle   \beta,  \vfi(Y)\eps \rangle}(Y_{n,k})\big)&\le  \E_{k,y}\big(\vfi(Y_{1})\langle  \beta , \eps_{1} \rangle 1_{\{ Y_{1}\ge Y_{n,k} \}} \big)^2\\
				\label{eq-numer1}
				&+\ff{n} \E_{k,y}\big(\vfi^2(Y_{1})\langle  \beta , \eps_{1} \rangle^2 1_{\{ Y_{1}\ge Y_{n,k} \}}  \big) .
			\end{align}
}%
			\cp{On the one hand,}  the Cauchy--Schwarz inequality in $H$ with $\| \beta \| = 1$ as well as  the bound~\eqref{eq:aux1} in the proof of Lemma~\ref{lem:norm_noise}, which is allowed since \eqref{hyp:noise_cond}, \eqref{hyp:2cgamma} and \eqref{hyp:qcgamma} hold, entail that there exists a constant $c\in (0,+\infty)$ independent of $n$ such that for $n$ large enough\cp{, uniformly in $y\ge y_{n,k}/2$}, 
			\begin{align} \nonumber
				\E_{k,y} \big(\vfi^2(Y_{1})\langle  \beta , \eps_{1} \rangle^2 1_{\{ Y_{1}\ge Y_{n,k} \}} \big) &\le  m_{k,y}(\vfi^2(Y)\|\eps\|^2)\\
				\label{eq-numer2}
				&\le c \E_{k,y}\left(\|\eps\|^{q}\right)^{2/q} \vfi^2(y)\left(\f{k-1}{n}\right)^{1-2/q}.
			\end{align}
			On the other hand, recalling that $Z_1:=\vfi(Y_1)\eps_1$, one may apply the Cauchy--Schwarz inequality in~$H$ to obtain
			\begin{align}
				\label{eq-tmp-a}
				\E_{k,y}\big(\vfi(Y_{1})\langle  \beta , \eps_{1} \rangle 1_{\{ Y_{1}\ge Y_{n,k} \}} \big)^2&\le   \E_{k,y}\big(\vfi(Y_{1})\|\eps_{1} \|1_{\{ Y_{1}\ge Y_{n,k} \}} \big)^2 = m^2_{k,y}(\| Z_1\|).
			\end{align}
			The focus now turns to the conditional tail-moment of $\|Z_1\|$ in the previous inequality. The $q$-moment of $\|\eps\|$ is finite under~\eqref{hyp:noise_cond} so that Hölder's inequality allows \rae{one} to write,
			\begin{align}
				\label{eq-tmp-b}
				m^2_{k,y}(\|Z_1\|)&\le\E_{k,y}\left(\|\eps_1\|^{q}\right)^{2/q}   m^{2(1-1/q)}_{k,y}({\vfi^{{q}/(q-1)}(Y_1)}).
			\end{align}
			Lemma~\ref{lem:bousebata_random} applies since ${\tau q}/(q-1) < 1/\gamma$ whenever $\tau\le 0$\rtworev{, because $\gamma>0$ by~\eqref{hyp:qcgamma}, as $q>2$ and $\kappa>0$}. This very condition is also satisfied  when $\tau>0$ under~\eqref{hyp:2cgamma} and~\eqref{hyp:qcgamma}. Indeed, the condition at hand is equivalent to $\tau<(1-1/q)/\gamma$ while~\eqref{hyp:2cgamma} gives $\tau < (1/2 -\kappa\gamma)/{\gamma}< (1 -\kappa\gamma)/{\gamma}$ and~\eqref{hyp:qcgamma} writes as $1/q<\kappa\gamma$ which concludes. Finally, it yields, for some constant $c>0$ independent of $n$, for $n$ large enough\cp{, uniformly in $y\ge y_{n,k}/2$},
			\begin{align}
				\label{eq-tmp-c}
				m^{2(1-1/q)}_{k,y}({\vfi^{{q}/(q-1)}(Y_1)})&\le c\vfi^2(y) \left(\f{k-1}{n}\right)^{2(1-1/q)} .
			\end{align}
			Collecting \eqref{eq-tmp-a}, \eqref{eq-tmp-b} and \eqref{eq-tmp-c} entails
			\begin{align}
				\label{eq-numer3}
				\E_{k,y}\big(\vfi(Y_{1})\langle  \beta , \eps_{1} \rangle 1_{\{ Y_{1}\ge Y_{n,k} \}} \big)^2&\le c \E_{k,y}\left(\|\eps_1\|^{q}\right)^{2/q}  {{\vfi}^{2}(y) }  \left(\f{k-1}{n}\right)^{2(1-1/q)},
			\end{align}
			for $n$ large enough, where $c$ is another constant independent of $n$.
			It readily follows from~\eqref{eq-numer1}, \eqref{eq-numer2} and~\eqref{eq-numer3} that
			the numerator of the RHS in~\eqref{eq:chebyshev} can be upper bounded as
			\begin{align}
				\label{eq-numer-fin}
				\E_{k,y}\big(\hat m^2_{\langle   \beta,  \vfi(Y)\eps \rangle}(Y_{n,k})\big) \leq  c \E_{k,y}\left(\|\eps_1\|^{q}\right)^{2/q}  {{\vfi}^{2}(y) }  \left(\f{k-1}{n}\right)^{2(1-1/q)}.
			\end{align}
			Besides, the denominator of the RHS in~\eqref{eq:chebyshev} is controlled through \citet[Lemma~2]{Bousebata2023}, for some constant $c>0$ independent of $n$ and for $n\to+\infty$,
			\begin{align}
				\label{eq-denom-fin}
				m^2_{\vfi g(Y)}(y_{n,k}) \sim c \vfi^2(y_{n,k}) g^2(y_{n,k}) \bar{F}^2(y_{n,k}) .
			\end{align}
			Thus, \eqref{eq-numer-fin} together with \eqref{eq-denom-fin} gives, for some constant $c>0$ independent of $n$, for $n$ large enough,
			\begin{align*}
				\f{\E_{k,y}\left(\hat m^2_{\langle   \beta, \vfi(Y) \eps \rangle}(Y_{n,k})\right) }{m^2_{\vfi g(Y)}(y_{n,k})}&\le  c (\vfi g \bar{F})^{-2}(y_{n,k})\E_{k,y}\left(\|\eps\|^{q}\right)^{2/q} \vfi^2(y)\left(\f{k-1}{n}\right)^{2(1-1/q)}.
			\end{align*}
			Overall, Markov's inequality \eqref{eq:chebyshev} becomes, for some constant $c\in (0,+\infty)$ independent of~$n$, for any $n$ large enough and any $\eta'>0$,
			\begin{align*}
				\p_{k,y}\left( \left| \f{\hat m_{\langle \beta,  \vfi(Y) \eps \rangle}(Y_{n,k})}{ m_{\vfi  g(Y)}(y_{n,k})} \right| >\eta'\right) &\le \f{c}{\eta'^2} \E_{k,y}\left(\|\eps\|^{q}\right)^{2/q} 
				\Delta_{n,k}(y),
			\end{align*}
			where $\Delta_{n,k}(y)$ is defined by~\eqref{eq-delta} in the proof of Lemma~\ref{lem:norm_noise}.

			The remaining part of the proof is a mere repetition of the same steps as in the proof of Lemma~\ref{lem:norm_noise}.
		\end{proof}

		\subsection{Consistency}\label{subsection:consistency_random}

		\cp{%
		Theorem~\ref{th:main} follows from the upcoming Proposition~\ref{prop:1}, which is proved based on Lemma~\ref{lem:norm_noise} and Lemma~\ref{lem:inner_prod_noise_L2_random}, combined with Proposition~\ref{prop:2}: the triangle inequality entails
		\begin{align*}
		\|\tilde\beta_\vfi-\beta\|&\le\|\tilde\beta_\vfi-\hat\beta_\vfi\|+\|\hat\beta_\vfi-\beta\|\\
		&=O_\p\big(N^{-\al_b}\vee\{N^{-\al_e}\delta_{n,k}\}\big)+O_\p(\delta_{n,k})=O_\p\big(\delta_{n,k}\vee N^{-\al_b}\big),
		\end{align*}
		all quantities being evaluated at $Y_{n-k+1,n}$, where the last equality uses $N^{-\al_e}\delta_{n,k}\le\delta_{n,k}$.
}%

		\begin{Prop}\label{prop:1}
			\cp{%
Assume \eqref{hyp:test}, \eqref{hyp:2rv}, \eqref{eq:single_index_model},  \eqref{hyp:link}, \eqref{hyp:vonmises}, \eqref{hyp:noise_cond}, \eqref{hyp:2cgamma} and \eqref{hyp:qcgamma} hold, with $\sqrt{k}A(n/k)=O(1)$. Assume $\vfi$ and $g$ are continuously differentiable in a neighborhood of infinity such that $t(\vfi g)'(t)/(\vfi g)(t) \to \tau+\kappa$ as $t\to+\infty$.
			Let $k:=k_n\to+\infty$ be an intermediate sequence and let $\delta_{n,k}:= (g(y_{n,k})(k/n)^{1/q})^{-1}$ where $y_{n,k}\sim U(n/k)$ as $n\to +\infty$.
}%
			$$ 
			\| \hat \beta_{\vfi} (Y_{n-k+1,n})- \beta \| = O_{\p}( \delta_{n,k} )\xrightarrow[n\to +\infty]{\cp{\p}}0  .$$
		\end{Prop}
		\begin{proof}[Proof of Proposition~\ref{prop:1}] 
			Denote $Y_{n,k}:=Y_{n-k+1,n}$. Under \eqref{eq:single_index_model}, the estimator $\sg{\hat m_{X\vfi(Y)}}$ may be expressed as
			\begin{align*}  \sg{\hat m_{X\vfi(Y)}}(Y_{n,k}) &= \sg{\hat m_{\vfi  g(Y_1)}}(Y_{n,k})  \beta+\hat m_{ \vfi(Y_1) \eps_1}(Y_{n,k}).
			\end{align*}
			\rae{Consequently,} since $\|  \beta\|=1$, one may write the inner product between $  \beta$ and $ \sg{\hat m_{X\vfi(Y)}}(Y_{n,k})$ as
			\begin{align*}
				\langle \sg{\hat m_{X\vfi(Y)}}(Y_{n,k}) , \beta\rangle &=\hat m_{\vfi g(Y_1)}(Y_{n,k}) + \hat m_{\langle \vfi(Y_1) \eps_1, \beta \rangle}(Y_{n,k}),
			\end{align*} and
			\begin{align*}
				\|\sg{\hat m_{X\vfi(Y)}}(Y_{n,k})\|^2 &=\hat m_{\vfi  g(Y_1)}^2(Y_{n,k})+\|\hat m_{\vfi(Y_1) \eps_1}(Y_{n,k})\|^2
				+2\hat m_{\vfi  g(Y_1)}(Y_{n,k}) \hat m_{\langle  \vfi(Y_1) \eps_1, \beta \rangle}(Y_{n,k}) .
			\end{align*}
			It follows that
			\begin{align*}
				1-\f{\langle  \sg{\hat m_{X\vfi(Y)}}(Y_{n,k}) ,  \beta\rangle^2}{\| \sg{\hat m_{X\vfi(Y)}}(Y_{n,k})\|^2}&= \f{\|\hat m_{  \vfi(Y_1) \eps_1  }(Y_{n,k})\|^2-\hat m_{\langle   \vfi(Y_1) \eps_1,  \beta \rangle}^2(Y_{n,k})}{\hat m_{\vfi g(Y_1)}^2(Y_{n,k}) +2\hat m_{\vfi g(Y_1)}(Y_{n,k}) \hat m_{\langle \vfi(Y_1)  \eps_1,  \beta \rangle}(Y_{n,k})+\|\hat m_{  \vfi(Y_1)  \eps_1  }(Y_{n,k})\|^2} , 
			\end{align*}
			which equals
			$$\f{ \left( {\|\hat m_{  \vfi(Y_1) \eps_1}(Y_{n,k})\|}/{\hat m_{\vfi g(Y_1)}(Y_{n,k})}\right)^2 -\left({\hat m_{\langle  \vfi(Y_1) \eps_1,  \beta \rangle}(Y_{n,k})}/{\hat m_{\vfi g(Y_1)}(Y_{n,k})} \right)^2}{1+2 \left({\hat m_{\langle \vfi(Y_1)  \eps_1,  \beta \rangle}(Y_{n,k})}/{\hat m_{\vfi g(Y_1)}(Y_{n,k})} \right)+\left( {\|\hat m_{ \vfi(Y_1) \eps_1}(Y_{n,k})\|}/{\hat m_{\vfi g(Y_1)}(Y_{n,k})}\right)^2 }.
			$$
			\cp{\rae{All the conditions} needed to apply \citet[Theorem~2]{Stupfler2019_Random_threshold} to the function $\vfi g$ \rae{are fulfilled.} Indeed, $t(\vfi g)'(t)/(\vfi g)(t)\to a:=\tau+\kappa$ as $t\to+\infty$ with $a>0$ in view of \eqref{hyp:2cgamma}, so that $\vfi g$ is ultimately increasing and $(\vfi g)'\in\RV_{a-1}(+\infty)$; the second-order condition on $U$ is \eqref{hyp:2rv}; the moment condition $0<2a\gamma<1$ is exactly \eqref{hyp:2cgamma}; and the bias condition $\sqrt{k}A(n/k)=O(1)$ is assumed. The required conditions for Lemma~\ref{lem:norm_noise} and Lemma~\ref{lem:inner_prod_noise_L2_random} are also fulfilled.} Therefore, \rtworev{the combination of these three results} together with \citet[Lemma~2]{Bousebata2023} and \rtworev{the relation} $\bar{F}(y_{n,k}) \sim \bar{F}(U(n/k) )=k/n$\rtworev{, established in the proof of Lemma~\ref{lem:norm_noise},} allows \rae{one} to chain asymptotic equivalences in probability:
			\begin{align}\label{eq:stupfler}
				\| \sg{\hat m_{X\vfi(Y)}}(Y_{n,k}) \|^2 =\hat m^2_{  \vfi g(Y_1) }(Y_{n,k}) (1+o_\p(1))&= m^2_{ \vfi g(Y) }(y_{n,k}) (1+o_\p(1)).
			\end{align}
			Whence after a first-order Taylor expansion, $1-{\langle  \sg{\hat m_{X\vfi(Y)}}(Y_{n,k}) ,  \beta\rangle^2}/{\| \sg{\hat m_{X\vfi(Y)}}(Y_{n,k})\|^2}$ is asymptotically equivalent, with respect to the convergence in probability, to \begin{align*}
				\f{ \left( {\|\hat m_{  \vfi(Y_1) \eps_1}(Y_{n,k})\|}/{ m_{\vfi g(Y)}(y_{n,k})}\right)^2 -\left({\hat m_{\langle  \vfi(Y_1) \eps_1,  \beta \rangle}(Y_{n,k})}/{ m_{\vfi g(Y)}(y_{n,k})} \right)^2}{1+2 \left({\hat m_{\langle \vfi(Y_1)  \eps_1,  \beta \rangle}(Y_{n,k})}/{ m_{\vfi g(Y)}(y_{n,k})} \right)+\left( {\|\hat m_{ \vfi(Y_1) \eps_1}(Y_{n,k})\|}/{m_{\vfi g(Y)}(y_{n,k})}\right)^2 }.
			\end{align*}
			Again, Lemma~\ref{lem:norm_noise} and Lemma~\ref{lem:inner_prod_noise_L2_random}, together with \cp{a} Taylor expansion, yield $$ 1-\f{\langle  \sg{\hat m_{X\vfi(Y)}}(Y_{n,k}) ,  \beta\rangle^2}{\| \sg{\hat m_{X\vfi(Y)}}(Y_{n,k})\|^2}=O_{\p}( \delta_{n,k}^2 ) \xrightarrow[n\to +\infty]{\cp{\p}}0 .$$
			Finally, \rae{the identity} $\| \hat \beta_{\vfi} (Y_{n,k})- \beta \|^2=2(1-{\langle  \sg{\hat m_{X\vfi(Y)}}(Y_{n,k}) ,  \beta\rangle}/{\| \sg{\hat m_{X\vfi(Y)}}(Y_{n,k})\|}) $ \rae{concludes} the proof \rae{together} with the classic identity $1-\lam^2 = (1+\lam)(1-\lam)$\cp{, where $\lam:={\langle  \sg{\hat m_{X\vfi(Y)}}(Y_{n,k}) ,  \beta\rangle}/{\| \sg{\hat m_{X\vfi(Y)}}(Y_{n,k})\|}$. Indeed, $\langle \sg{\hat m_{X\vfi(Y)}}(Y_{n,k}),\beta\rangle=\hat m_{\vfi g(Y_1)}(Y_{n,k})(1+o_\p(1))$ with $\hat m_{\vfi g(Y_1)}\ge0$, so that $1+\lam\ge1$ with probability tending to one and $0\le1-\lam\le1-\lam^2=O_\p(\delta^2_{n,k})$}.
		\end{proof}

\cp{%
		\noindent The next proposition controls the effect of the grid-based reconstruction on the FEPLS estimator.
		\begin{Prop}\label{prop:2}
			Assume \eqref{eq:single_index_model} holds with  $\beta\in\mathcal{C}^{0,\al_b}([0,1],\R)$ for some $\al_b\in(0,1]$, together with \eqref{hyp:test}, \rtworev{\eqref{hyp:2rv}}, \eqref{hyp:link}, \eqref{hyp:vonmises},  \eqref{hyp:noise_cond}, \eqref{hyp:grid}, \eqref{hyp:2cgamma} and \eqref{hyp:qcgamma}. Suppose  that $\vfi$ and $g$ are continuously differentiable in a neighborhood of infinity such that $t(\vfi g)'(t)/(\vfi g)(t)\to\tau+\kappa$ as $t\to+\infty$. Let $k:=k_n\to+\infty$ be an intermediate sequence such that $\sqrt{k}A(n/k)=O(1)$ and let $y_{n,k}\sim U(n/k)$ as $n\to+\infty$. Then, for any deterministic sequence $N:=N_n\to+\infty$, it holds that:
			\begin{align}
				\label{eq:prop2-norm}
				\big\|\hat m_{\tilde X\vfi(Y)}(Y_{n-k+1,n})\big\| &= m_{\vfi g(Y)}(y_{n,k})\,\big(1+o_\p(1)\big)\quad\text{and}\\
				\label{eq:prop2-dir}
				\big\|\tilde\beta_\vfi(Y_{n-k+1,n})-\hat\beta_\vfi(Y_{n-k+1,n})\big\| &= O_\p\Big( N^{-\al_b}\vee \big\{N^{-\al_e}\,\delta_{n,k}\big\}\Big) ,
			\end{align}
			where $\delta_{n,k}:=(g(y_{n,k})(k/n)^{1/q})^{-1}$.
		\end{Prop}
		\begin{proof}
			Denote $Y_{n,k}:=Y_{n-k+1,n}$, $\hat m:=\hat m_{X\vfi(Y)}(Y_{n,k})$, $\tilde m:=\hat m_{\tilde X\vfi(Y)}(Y_{n,k})$ and $m:=m_{\vfi g(Y)}(y_{n,k})$, and\rae{, as in} the proof of Lemma~\ref{lem:norm_noise}\rae{,} the two sequences $y^+_{n,k}:=3y_{n,k}/2$ and $y^-_{n,k}:=y_{n,k}/2$.

			\rae{The pathwise discretization error is controlled first.} Let $\al\in(0,1]$ and $h\in\mathcal{C}^{0,\al}([0,1],\R)$, \sg{\rae{where, as in} Section~\ref{sec-inf}\rae{,} $\|h\|_{0,\al}$ denotes the H\"older norm of $h$ on $[0,1]$; in particular} $|h(s)-h(t)|\le \|h\|_{0,\al}|s-t|^{\al}$ for any $s,t\in[0,1]$. Since $\max_{1\le j \le N}|I_j|\le c_0/N$ by \eqref{hyp:grid} and $t_j\in I_j$, one has
			\begin{align}\label{eq:interp}
				\Big\| h-\sum_{j=1}^N h(t_j)1_{I_j}\Big\|^2=\sum_{j=1}^N\int_{I_j}\big(h(t)-h(t_j)\big)^2\dd t\le \|h\|_{0,\al}^2\,\Big(\f{c_0}{N}\Big)^{2\al}\sum_{j=1}^N|I_j| = \|h\|_{0,\al}^2\,\Big(\f{c_0}{N}\Big)^{2\al}.
			\end{align}
			Under \eqref{eq:single_index_model}, $\tilde X_i-X_i=g(Y_i)\big(\sum_{j=1}^N\beta(t_j)1_{I_j}-\beta\big)+\big(\sum_{j=1}^N\eps_i(t_j)1_{I_j}-\eps_i\big)$, so that \eqref{eq:interp}, applied to $\beta$ and to $\eps_i$, yields almost surely, for any $1\le i\le n$,
			\begin{align}\label{eq:prop2-step1}
				\|\tilde X_i-X_i\|\le c_0^{\al_b}\,\|\beta\|_{0,\al_b}\,g(Y_i)\,N^{-\al_b}+c_0^{\al_e}\,\|\eps_i\|_{0,\al_e}\,N^{-\al_e} .
			\end{align}
			Since $\vfi\ge0$ and $g\ge0$, the triangle inequality and \eqref{eq:prop2-step1} entail
			\begin{align}\nonumber
				\|\tilde m-\hat m\| &= \Big\|\ff{n}\sum_{i=1}^n(\tilde X_i-X_i)\,\vfi(Y_i)1_{\{Y_i\ge Y_{n,k}\}}\Big\|\le \hat m_{\|\tilde X-X\|\vfi(Y)}(Y_{n,k})\\
				\label{eq:prop2-step2}
				&\le c\,N^{-\al_b}\,\hat m_{\vfi  g(Y)}(Y_{n,k})+c\,N^{-\al_e}\,\hat m_{\vfi(Y)\|\eps\|_{0,\al_e}}(Y_{n,k}),
			\end{align}
			for some constant $c>0$ depending only on $c_0$, $\al_b$, $\al_e$ and $\|\beta\|_{0,\al_b}$.

			Concerning the link-function term, all the assumptions of Proposition~\ref{prop:1} are in force, so that~\eqref{eq:stupfler} in its proof gives both $\hat m_{\vfi g(Y)}(Y_{n,k})=m\,(1+o_\p(1))$ and $\|\hat m\|=m\,(1+o_\p(1))$.

			\rae{The focus now turns to} the noise term $S:=\hat m_{\vfi(Y)\|\eps\|_{0,\al_e}}(Y_{n,k})$. By equidistribution and the conditional H\"older inequality with conjugate exponents $q/(q-1)$ and $q$,
			$$
			\E_{k,y}(S)=m_{k,y}\big(\vfi(Y_1)\|\eps_1\|_{0,\al_e}\big)\le m_{k,y}\big(\vfi^{q/(q-1)}(Y_1)\big)^{(q-1)/q}\,\E_{k,y}\big(\|\eps_1\|_{0,\al_e}^{q}\big)^{1/q} .
			$$
			\rae{The inequality} $\tau\gamma\, q/(q-1)<1$ \rae{holds}: this is immediate when $\tau\le 0$ while, when $\tau>0$, \eqref{hyp:2cgamma} and~\eqref{hyp:qcgamma} give $\tau\gamma<1/2-\kappa\gamma<1/2-1/q<(q-1)/q$. Since $\vfi^{q/(q-1)}\in\RV_{\tau q/(q-1)}(+\infty)$, Lemma~\ref{lem:bousebata_random} thus provides a constant $c'>0$ independent of $n$ such that $m_{k,y}(\vfi^{q/(q-1)}(Y_1))\le c'\, \vfi^{q/(q-1)}(y)\, k/n$ for all $y$ large enough, uniformly in $n$ and $k$. Besides, \eqref{hyp:noise_cond} provides a constant $C>0$ independent of $n$ such that, for $n$ large enough, $\sup_{y\geq 0}\E_{k,y}(\|\eps_1\|^q_{0,\al_e})\le C$. Collecting these bounds, for $n$ large enough,
			$$
			\sup_{y\in[y^-_{n,k},y^+_{n,k}]}\f{\E_{k,y}(S)}{\vfi(y)}\le c''\,(k/n)^{1-1/q},
			$$
			for some constant $c''>0$ independent of $n$. In addition, \citet[Lemma~2]{Bousebata2023} yields $m\ge c'''\,(\vfi g\bar{F})(y_{n,k})$ for $n$ large enough, since $(\tau+\kappa)\gamma<1$ by \eqref{hyp:2cgamma}. \rae{Let} $M>0$. \rtworev{Splitting the event $\{S>M\,\delta_{n,k}\,m\}$ according to whether $Y_{n,k}$ lies outside $[y^-_{n,k},y^+_{n,k}]$ or in it, bounding the probability of the first part by that of $\{Y_{n,k}\notin[y^-_{n,k},y^+_{n,k}]\}$, and applying to the second part the Markov inequality of order one to $S\ge0$ given $\{Y_{n,k}=y\}$, for $y\in[y^-_{n,k},y^+_{n,k}]$, entail}
			$$
			\p\big( S> M\,\delta_{n,k}\,m\big)\le \p\big(Y_{n,k}\notin[y^-_{n,k},y^+_{n,k}]\big)+\sup_{y\in[y^-_{n,k},y^+_{n,k}]}\f{\E_{k,y}(S)}{M\,\delta_{n,k}\,m} .
			$$
			For $n$ large enough, $\vfi\in\RV_{\tau}(+\infty)$ gives $\sup_{y\in[y^-_{n,k},y^+_{n,k}]}\vfi(y)\le c''''\,\vfi(y_{n,k})$, while $\bar{F}(y_{n,k})\sim k/n$. Hence, for $n$ large enough,
			\begin{align*}
			\sup_{y\in[y^-_{n,k},y^+_{n,k}]}\f{\E_{k,y}(S)}{M\,\delta_{n,k}\,m}&\le \f{c''c''''\,\vfi(y_{n,k})\,(k/n)^{1-1/q}}{M\,\delta_{n,k}\,c'''\,(\vfi g\bar{F})(y_{n,k})}\\
			&\le \f{c''c''''}{c'''}\cdot\f{2}{M\,\delta_{n,k}\,g(y_{n,k})\,(k/n)^{1/q}}= \f{2c''c''''}{c'''\,M},
			\end{align*}
			since $\delta_{n,k}\,g(y_{n,k})\,(k/n)^{1/q}=1$ by definition of $\delta_{n,k}$. As $\p(Y_{n,k}\notin[y^-_{n,k},y^+_{n,k}])\to 0$ when $n\to+\infty$, letting $n\to+\infty$ and then $M\to+\infty$ shows that
			\begin{align}\label{eq:prop2-noise}
				S/m=O_\p(\delta_{n,k}) .
			\end{align}
			Collecting \eqref{eq:prop2-step2}, the link-function term and \eqref{eq:prop2-noise},
			\begin{align}\label{eq:prop2-RN}
				\f{\|\tilde m-\hat m\|}{m}=O_\p\big(N^{-\al_b}\big)+N^{-\al_e}\,O_\p\big(\delta_{n,k}\big)=o_\p(1),
			\end{align}
			since $\delta_{n,k}\to 0$, see Proposition~\ref{prop:1}. The triangle inequality $\big|\,\|\tilde m\|-\|\hat m\|\,\big|\le\|\tilde m-\hat m\|=o_\p(m)$ combined with $\|\hat m\|=m(1+o_\p(1))$ proves \eqref{eq:prop2-norm}. Next, on the event $\{\|\tilde m-\hat m\|\le\|\hat m\|/2\}$, whose probability tends to one by \eqref{eq:prop2-RN}, one has $\tilde m\ne0$ and, for any $u,w\in H\setminus\{0\}$,
			$$
			\Big\|\f{u}{\|u\|}-\f{w}{\|w\|}\Big\|\le\f{\|u-w\|}{\|u\|}+\|w\|\,\Big|\ff{\|u\|}-\ff{\|w\|}\Big|\le\f{2\,\|u-w\|}{\|u\|},
			$$
			where the last inequality stems from $|\,\|u\|-\|w\|\,|\le\|u-w\|$. Applying this bound with $u=\hat m$ and $w=\tilde m$ yields
			\begin{align*}
			\big\|\tilde\beta_\vfi(Y_{n,k})-\hat\beta_\vfi(Y_{n,k})\big\|&\le\f{2\,\|\tilde m-\hat m\|}{\|\hat m\|}=O_\p\Big(N^{-\al_b}\vee\big\{N^{-\al_e}\,\delta_{n,k}\big\}\Big),
			\end{align*}
			which proves \eqref{eq:prop2-dir} and ends the proof.
		\end{proof}
}%

		\cp{
		\begin{proof}[Proof of Proposition~\ref{prop:modelfree}]
	For clarity, denote $Y_{n,k}:=Y_{n-k+1,n}$, $y_{n,k}:=U(n/k)$ and $m_{s^2(Y)}(y):=\E(s^2(Y)1_{\{Y\ge y\}})$.
	\rae{The bounds below are taken} relative to $m_{s(Y)}(y_{n,k})$\rae{,} and \rae{the empirical threshold is considered} at the end. Since
	$\E_{\mathcal{F},\mathcal{B}}(\xi)=0$, one has $m_W=m_{\mu(Y)}$ and the following expansion holds:
	$$
	\hat m_W(Y_{n,k})-m_W(y_{n,k})=A_{1,n}+A_{2,n}+A_{3,n},
	$$
	where $A_{1,n}:=\hat m_{\mu(Y)}(Y_{n,k})-\hat m_{\mu(Y)}(y_{n,k})$, $A_{2,n}:=\hat m_{\mu(Y)}(y_{n,k})-m_W(y_{n,k})$ and
	$A_{3,n}:=\hat m_{\xi}(Y_{n,k})$, with $\hat m_{\mu(Y)}(Y_{n,k})=\ff{n}\sum_{i=1}^n\mu(Y_i)1_{\{Y_i\ge Y_{n,k}\}}$.

	\rae{The deterministic tail-moments are computed first.} Since $s\in\RV_\lambda(+\infty)$ with $\lambda\gamma<1$,
	which follows from $2\lambda\gamma<1$, the integrand $sf\in\RV_{\lambda-1/\gamma-1}(+\infty)$ has a regular variation index below
	$-1$, so that Karamata's theorem together with $yf(y)\sim\gamma^{-1}\bar{F}(y)$ yield
	$m_{s(Y)}(y)\sim s(y)\bar{F}(y)/(1-\lambda\gamma)$. \rtworev{Karamata's theorem, applied in the same way to $s^2f\in\RV_{2\lambda-1/\gamma-1}(+\infty)$, gives} $m_{s^2(Y)}(y)\sim s^2(y)\bar{F}(y)/(1-2\lambda\gamma)$
	since $2\lambda\gamma<1$. Using $\bar{F}(y_{n,k})\sim k/n$, it follows that
	\begin{align}\label{eq:mf_var}
		\f{m_{s^2(Y)}(y_{n,k})}{n\,m^2_{s(Y)}(y_{n,k})}\sim\f{(1-\lambda\gamma)^2}{1-2\lambda\gamma}\cdot\ff{k}.
	\end{align}
	Besides, the triangle inequality for the Bochner integral gives $\|m_W(y)\|\le m_{s(Y)}(y)$ for every $y$.

	The focus now turns to the term $A_{2,n}$, the fluctuation at the deterministic threshold. The variables
	$Z_i:=\mu(Y_i)1_{\{Y_i\ge y_{n,k}\}}$ are i.i.d. in $H$ with $\E_{\mathcal{B}}(Z_1)=m_W(y_{n,k})$ and
	$\|Z_1\|=s(Y_1)1_{\{Y_1\ge y_{n,k}\}}$, so that $A_{2,n}=\ff{n}\sum_{i=1}^n(Z_i-\E_{\mathcal{B}}(Z_i))$. By the
	orthogonality of centered i.i.d. summands in $H$, which follows from Lemma~\ref{lem:bochner_random}, and
	by~\eqref{eq:mf_var}, $$
	\E(\|A_{2,n}\|^2)=\ff{n}\E(\|Z_1-\E_{\mathcal{B}}(Z_1)\|^2)\le\ff{n}\E(\|Z_1\|^2)=\ff{n}m_{s^2(Y)}(y_{n,k}).
	$$
	Markov's inequality and the deterministic denominator \rtworev{then} give
	\begin{align}\label{eq:mf_det}
		\f{\|A_{2,n}\|}{m_{s(Y)}(y_{n,k})}=O_\p(k^{-1/2}).
	\end{align}

	The next step is to bound the term $A_{1,n}$. The indicators $1_{\{Y_i\ge Y_{n,k}\}}$ and $1_{\{Y_i\ge y_{n,k}\}}$
	are nested, so that their difference has the same sign for every $i$. Since $\|\mu(Y_i)\|=s(Y_i)\ge0$, the
	triangle inequality in $H$ gives
	$$
	\|A_{1,n}\|\le\ff{n}\sum_{i=1}^n s(Y_i)\big|1_{\{Y_i\ge Y_{n,k}\}}-1_{\{Y_i\ge y_{n,k}\}}\big|
	=\big|\hat m_{s(Y)}(Y_{n,k})-\hat m_{s(Y)}(y_{n,k})\big|.
	$$
	\rtworev{The norm} $s$ is continuously differentiable near infinity \rtworev{with $ts'(t)/s(t)\to\lambda$, where $\lambda>0$ by the lower bound in~\eqref{hyp:mu_norm}. Since $s'(t)=\big(ts'(t)/s(t)\big)\,s(t)/t$, where the first factor tends to $\lambda>0$ and $s(t)/t$ is positive for $t$ large and in $\RV_{\lambda-1}(+\infty)$ by~\eqref{hyp:mu_norm}, $s$ is ultimately increasing and} $s'\in\RV_{\lambda-1}(+\infty)$\rtworev{. Moreover, $0<2\lambda\gamma<1$ follows from $\lambda>0$ and the upper bound in~\eqref{hyp:mu_norm}, and this} is exactly the \rtworev{moment} condition of \citet[Theorem~2]{Stupfler2019_Random_threshold}, as in the proof of Proposition~\ref{prop:1} \rtworev{with $\lambda$ in place of $\tau+\kappa$.
	The remaining conditions of that theorem hold as well: its second-order condition is~\eqref{hyp:2rv}, on a neighborhood of infinity where the auxiliary function $A$ has constant sign; the conditions on $k$ are assumed; and, since $s\ge0$, the conditional mean $\E(s(Y)\mid Y>t)$ is at most $m_{s(Y)}(t)/\bar F(t)$, which is finite for $t$ large by the computation of $m_{s(Y)}$ above.
	Applied to the sample $Y_1,\dots,Y_n$ of independent copies of $Y$ and to the function $s$, that theorem bears on the mean of the $s(Y_i)$ over the observations strictly above $Y_{n-k,n}$, divided by $\E(s(Y)\mid Y>y_{n,k})$.
	Since $Y$ has a density, see~\eqref{hyp:vonmises}, its cdf $F$ is continuous, as required in the framework of that theorem, so that almost surely there are no ties: the observations strictly above $Y_{n-k,n}$ are exactly the $k$ ones with $Y_i\ge Y_{n-k+1,n}=Y_{n,k}$, in proportion $k/n$. The continuity of $F$ also gives $\bar F(y_{n,k})=\bar F(U(n/k))=k/n$ and $\p(Y=y_{n,k})=0$, so that $\E(s(Y)\mid Y>y_{n,k})=m_{s(Y)}(y_{n,k})/\bar F(y_{n,k})$.
	The ratio above is therefore $\hat m_{s(Y)}(Y_{n,k})/m_{s(Y)}(y_{n,k})$, and the theorem gives the convergence in distribution of $\sqrt{k}\,\big(\hat m_{s(Y)}(Y_{n,k})/m_{s(Y)}(y_{n,k})-1\big)$. Since every sequence of real random variables converging in distribution is uniformly tight, that is, bounded in probability, it follows that}
	$\hat m_{s(Y)}(Y_{n,k})=m_{s(Y)}(y_{n,k})(1+O_\p(k^{-1/2}))$. Besides, the scalar version of~\eqref{eq:mf_det},
	with $s$ in place of $\mu$, gives $\hat m_{s(Y)}(y_{n,k})=m_{s(Y)}(y_{n,k})(1+O_\p(k^{-1/2}))$. Subtracting, it
	follows that $\|A_{1,n}\|/m_{s(Y)}(y_{n,k})=O_\p(k^{-1/2})$.

It remains to bound the residual $A_{3,n}$. As in the proof of Lemma~\ref{lem:norm_noise}, let
	$\E_{k,y}(\cdot):=\E(\cdot\mid Y_{n,k}=y)$ and $m_{k,y}(W_i):=\E_{k,y}(W_i1_{\{Y_i\ge Y_{n,k}\}})$. \cp{The conditional centring $\E_{\mathcal{F},\mathcal{B}}(\xi)=0$ reads
	$\E_{\mathcal{B}}(\xi_1\mid Y_1=t)=0$ for almost every $t$, while the moment conditions of
	Lemma~\ref{lem:second_moment} hold in view of \eqref{hyp:res_cond} and Jensen's inequality, so that
	Remark~\ref{rmk:second-moment-centred}, a
	consequence of Lemma~\ref{lem:cond_sample} and Lemma~\ref{lem:second_moment}, applies to $Z_i=\xi_i$ and yields
	\rorange{for almost every $y$ with respect to the law of $Y_{n,k}$, which has a density (proof of Lemma~\ref{lem:cond_sample}),} $m_{k,y}(\xi_1)=0$ together with
	$$
	\E_{k,y}(\|A_{3,n}\|^2)=\ff{n}\,m_{k,y}(\|\xi_1\|^2).
	$$
	Contrary to the proof of Lemma~\ref{lem:norm_noise}, where the noise $\eps$
	in~\eqref{eq:single_index_model} is not assumed conditionally centered and the tail-moment $m_{k,y}(Z_1)$ does
	not vanish, only the diagonal term remains here.} By Hölder's inequality
	with exponents $(q/2,q/(q-2))$, \eqref{hyp:res_cond}\rtworev{, which by Lemma~\ref{lem:cond_sample} yields $\sup_{y}\E\big(\|\xi_1\|^q\mid Y_{n-k+1,n}=y\big)<+\infty$ uniformly in $n$ and $k$,} and \cp{$m_{k,y}(1)=k/n$ from Lemma~\ref{lem:bousebata_random},}
	there is a constant $c>0$ independent of $(n,y)$ such that, for $n$ large enough and \rorange{for almost every $y$},
	$m_{k,y}(\|\xi\|^2)\le c(k/n)^{1-2/q}$. \rtworev{The} tower property and
	$m^2_{s(Y)}(y_{n,k})\sim \rorange{c'}s^2(y_{n,k})(k/n)^2$ then give, with $\delta^\mu_{n,k}:=(s(y_{n,k})(k/n)^{1/q})^{-1}$,
	$\|A_{3,n}\|/m_{s(Y)}(y_{n,k})=O_\p(\delta^\mu_{n,k}/\sqrt{k})$. Besides, \rtworev{the lower bound in}~\eqref{hyp:mu_norm} ensures that
	$1/\delta^\mu_{n,k}$ is regularly varying of positive index $\lambda\gamma-1/q$ in $n/k$, hence
	$\delta^\mu_{n,k}\to0$, so this term is $o_\p(k^{-1/2})$.

	Dividing the splitting by $m_{s(Y)}(y_{n,k})$, the three terms are $O_\p(k^{-1/2})$.
	It remains to replace $y_{n,k}$ by $Y_{n,k}$
	inside $m_W$ and $m_{s(Y)}$. \cp{To this end, \eqref{hyp:vonmises} ensures $\sqrt{k}\,(Y_{n,k}/y_{n,k}-1)=O_\p(1)$, see
	\citet[Theorem~2.2.1]{Haan2007};
	since $m_{s(Y)}$ is differentiable with
	$-y\,m_{s(Y)}'(y)/m_{s(Y)}(y)=y\,s(y)f(y)/m_{s(Y)}(y)\to(1-\lambda\gamma)/\gamma$, the mean value theorem
	combined with the uniform convergence theorem for regularly varying functions \rtworev{\citep[Theorem~1.5.2]{bingham_goldie_teugels_1987}, applied on $[1/2,2]$ to $m_{s(Y)}\in\RV_{-(1-\lambda\gamma)/\gamma}(+\infty)$,} yields
	$$
	m_{s(Y)}(Y_{n,k})/m_{s(Y)}(y_{n,k})=1+O_\p(k^{-1/2}).
	$$
	}\rtworev{The Bochner triangle inequality, as for $A_{1,n}$, gives in addition}
	$\|m_W(Y_{n,k})-m_W(y_{n,k})\|\le|m_{s(Y)}(Y_{n,k})-m_{s(Y)}(y_{n,k})|$.

\rtworev{%
As for the direction, since $y_{n,k}=U(n/k)\to+\infty$, condition~\eqref{hyp:nondeg} gives $\|m_W(y_{n,k})\|\ge c_\star\,m_{s(Y)}(y_{n,k})$ for some $c_\star\in(0,1]$ and $n$ large enough, so that $m_W(y_{n,k})\neq0$ and $\bar a(y_{n,k})$ is well defined. Denote $\Pi_n:=\hat m_W(Y_{n,k})/m_{s(Y)}(y_{n,k})$ and $\bar\Pi_n:=m_W(y_{n,k})/m_{s(Y)}(y_{n,k})$, so that $\|\bar\Pi_n\|\in[c_\star,1]$ by the bound $\|m_W\|\le m_{s(Y)}$ established above, and $\bar\Pi_n/\|\bar\Pi_n\|=\bar a(y_{n,k})$, together with
\[
	\Theta_n:=\|\Pi_n-\bar\Pi_n\|
	\qquad\text{and}\qquad
	\Xi_n:=\f{\big|m_{s(Y)}(Y_{n,k})-m_{s(Y)}(y_{n,k})\big|}{m_{s(Y)}(y_{n,k})} .
\]
Both are $O_\p(k^{-1/2})$, hence $o_\p(1)$ since $k\to+\infty$: $\Theta_n$ is bounded by $\sum_{j=1}^3\|A_{j,n}\|/m_{s(Y)}(y_{n,k})$, controlled by the three bounds above, and $\Xi_n$ by the display $m_{s(Y)}(Y_{n,k})/m_{s(Y)}(y_{n,k})=1+O_\p(k^{-1/2})$ of the threshold swap. The argument rests on the normalized triangle inequality: for $u,z\in H\setminus\{0\}$, $\|u/\|u\|-z/\|z\|\|\le\|u/\|u\|-u/\|z\|\|+\|u-z\|/\|z\|=\big|\|z\|-\|u\|\big|/\|z\|+\|u-z\|/\|z\|\le2\|u-z\|/\|z\|$, the last step by the reverse triangle inequality. The denominator is that of the \emph{second} argument, a role held below by the population quantity $m_W(y_{n,k})$, deterministic and bounded below by~\eqref{hyp:nondeg}.

First, $\|\Pi_n\|\ge\|\bar\Pi_n\|-\Theta_n\ge c_\star-\Theta_n$, so that the event $\{\Theta_n\le c_\star/2\}$ has probability tending to one; on it $\|\Pi_n\|\ge c_\star/2$, hence $\hat m_W(Y_{n,k})\neq0$, and the normalized triangle inequality with $u=\Pi_n$ and $z=\bar\Pi_n$ gives
\[
	\Big\|\f{\hat m_W(Y_{n,k})}{\|\hat m_W(Y_{n,k})\|}-\bar a(y_{n,k})\Big\|\le\f{2\,\Theta_n}{c_\star} .
\]
Second, the reverse triangle inequality and the Bochner triangle inequality $\|m_W(Y_{n,k})-m_W(y_{n,k})\|\le\big|m_{s(Y)}(Y_{n,k})-m_{s(Y)}(y_{n,k})\big|$, established above, combined with $\|m_W(y_{n,k})\|\ge c_\star\,m_{s(Y)}(y_{n,k})$, give $\big|\|m_W(Y_{n,k})\|/\|m_W(y_{n,k})\|-1\big|\le\Xi_n/c_\star$. On the event $\{\Xi_n<c_\star/2\}$, whose probability tends to one, the ratio $\|m_W(Y_{n,k})\|/\|m_W(y_{n,k})\|$ therefore exceeds $1/2$, so that $m_W(Y_{n,k})\neq0$, and the normalized triangle inequality with $u=m_W(Y_{n,k})$ and $z=m_W(y_{n,k})$ gives
\[
	\big\|\bar a(Y_{n,k})-\bar a(y_{n,k})\big\|\le\f{2\,\Xi_n}{c_\star} .
\]
The intersection of the two events has probability tending to one; on it, $\hat m_W(Y_{n,k})$ and $m_W(Y_{n,k})$ are both non-zero, so that it is contained in the event of the statement, whose probability therefore tends to one, and the triangle inequality in $H$ bounds the left-hand side of the statement by $2(\Theta_n+\Xi_n)/c_\star$. For $M>0$, the probability that this left-hand side exceeds $Mk^{-1/2}$ is thus at most the sum of the probability of the complement of the intersection and of $\p\big(2(\Theta_n+\Xi_n)/c_\star>Mk^{-1/2}\big)$. The first term tends to $0$ as $n\to+\infty$, and the limit superior in $n$ of the second tends to $0$ as $M\to+\infty$, since $\Theta_n+\Xi_n=O_\p(k^{-1/2})$. This proves the claim.
}
\end{proof}
		}

		\section*{Acknowledgements}

		This work is partially supported by the French National Research Agency (ANR) under the grant ANR-23-CE40-0009. 
		The authors also would like to thank the Associate Editor and the reviewers for their suggestions to improve the paper.
		S.~Girard acknowledges the support of the Chair ``Stress Test, Risk Management and Financial Steering'', led by the French Ecole Polytechnique and its Foundation and sponsored by BNP Paribas.

		\bibliographystyle{elsarticle-harv} 
		\bibliography{biblio_revised}
	\end{document}